\documentclass[11pt,reqno]{amsart}
\usepackage{mathrsfs}
\usepackage{amsmath,amscd}
\usepackage{amssymb, amsmath}
\usepackage{graphicx}
\usepackage{color}

\usepackage{tikz}
\usepackage{amsaddr}

\allowdisplaybreaks

\newcommand{\newcom}{\newcommand}

\newcom{\al}{\alpha}
\newcom{\Del}{\Delta}
\newcom{\be}{\beta}
\newcom{\eps}{\epsilon}
\newcom{\e}{\varepsilon}
\newcom{\ga}{\gamma}
\newcom{\Ga}{\Gamma}
\newcom{\ka}{\kappa}
\newcom{\Lam}{\Lambda}
\newcom{\lam}{\lambda}
\newcom{\Om}{\Omega}
\newcom{\om}{\omega}
\newcom{\Si}{\Sigma}
\newcom{\si}{\sigma}
\newcom{\tht}{\theta}
\newcom{\dtri}{\nabla}
\newcom{\tri}{\triangle}
\newcom{\oo}{\infty}
\newcom{\vphi}{\varphi}
\newcom{\cA}{{\mathcal A}}
\newcom{\cB}{{\mathcal B}}
\newcom{\cC}{{\mathcal C}}
\newcom{\cD}{{\mathcal D}}
\newcom{\cE}{{\mathcal E}}
\newcom{\ce}{{\mathcal e}}
\newcom{\cF}{{\mathcal F}}

\newcom{\cJ}{{\mathcal J}}
\newcom{\cK}{{\mathcal K}}
\newcom{\cL}{{\mathcal L}}
\newcom{\cM}{{\mathcal M}}
\newcom{\cP}{{\mathcal P}}
\newcom{\cR}{{\mathcal R}}
\newcom{\cS}{{\mathcal S}}
\newcom{\cQ}{{\mathcal Q}}
\newcom{\cT}{{\mathcal T}}
\newcom{\cU}{{\mathcal U}}
\newcom{\cY}{{\mathcal Y}}
\newcom{\cZ}{{\mathcal Z}}
\newcom{\R}{\mathbb R}
\newcom{\T}{\mathbb T}
\newcom{\N}{\mathbb N}
\newcom{\Z}{\mathbb Z}
\newcom{\C}{\mathbb C}
\newcom{\E}{\mathbb E}

\newcom{\f}{\frac}
\newcom{\di}{\displaystyle\int}
\newcom{\ds}{\displaystyle\sum}
\newcom{\dl}{\displaystyle\lim}
\newcom{\ov}{\overline}
\newcom{\sset}{\subset}
\newcom{\wt}{\widetilde}
\newcom{\wh}{\widehat}
\newcom{\pa}{\partial}
\newcom{\p}{\partial}
\newcom\na{\nabla}
\newcom\lan{\langle}
\newcom\ran{\rangle}
\newcom{\suml}{\sum\limits}
\newcom{\supl}{\sup\limits}
\newcom{\intl}{\int\limits}
\newcom{\infl}{\inf\limits}
\newcom{\disp}{\displaystyle}
\newcom{\non}{\nonumber}
\newcom{\no}{\noindent}
\newcom{\QED}{$\square$}

\def\div{\mathop{\rm div}\nolimits}

\def\ef{\hphantom{MM}\hfill\llap{$\square$}\goodbreak}

\newtheorem{athm}{\bf \t}[section]
\newenvironment{thm} [1] {\def\t{#1}\begin{athm} \bf \rm} {\end{athm}}
\newcom{\bthm}{\begin{thm}}\newcom{\ethm}{\end{thm}}

\newtheorem{theorem}{Theorem}[section]
\newtheorem{lemma}{Lemma}[section]
\newtheorem{remark}{Remark}[section]

\newtheorem{proposition}{Proposition}[section]

\newcom{\beq}{\begin{equation}}
\newcom{\eeq}{\end{equation}}
\newcom{\ben}{\begin{eqnarray}}
\newcom{\een}{\end{eqnarray}}
\newcom{\beno}{\begin{eqnarray*}}
\newcom{\eeno}{\end{eqnarray*}}
\newcom{\bal}{\begin{aligned}}
\newcom{\eal}{\end{aligned}}

\numberwithin{equation}{section}

\begin{document}

\title[ 2D MHD equations on the half space]
{Global well-posedness and long time behavior of  2D MHD equations  with   partial dissipation in half space}

\author{Jiakun Jin,  Xiaoxia Ren*, Lei Wang}
\address{Department of Mathematics and Physics, North China Electric Power University, 102206, Beijing, P. R. China.}
\footnotetext{*Corresponding author}
\email{jinjiakun18@163.com(J.Jin), xiaoxiaren@ncepu.edu.cn(X.Ren), 50901924@ncepu.edu.cn(L.Wang)}

\maketitle

\begin{abstract}

In this paper, we  obtain the low order global well-posedness and the asymptotic behavior of solution of  2D MHD problem with  partial dissipation in half space with non-slip boundary condition. When magnetic field equal zero, the system be reduced to partial dissipation Navier-Stokes equation, so this  result also implies  the stabilizing effects of magnetic field in electrically conducting fluids. We use the resolvent estimate  method to obtain the long time behavior for the solution of  weak diffusion system, which is not necessary  to prove   global well-posedness.
\end{abstract}

{\small

{\textbf{AMS (2010) Subject Classifications:}} 35A01; 35Q35; 76W05 }

 \section{Introduction}
\setcounter{equation}{0}

In this paper, we investigate two dimensional  partial dissipation magneto-hydrodynamical equations
\ben\label{eq:MHD0}
\left\{
\begin{array}{l}
\p_t u-\pa_{x_2}^2 u+u\cdot\na u+\na p=B\cdot \na B,  \ \ \ \ x\in\Omega,\ t>0, \\
\p_tB-\pa_{x_1}^2 B+u\cdot\na B=B\cdot \na u,  \ \ \ \ x\in\Omega,\ t>0,\\
\div u=\div B=0,  \ \ \ \ x\in\Omega,\ t>0,\\
u(x,0)=u_0(x),\ \ B(x,0)=B_0(x),\ \ \ \ x\in \Omega,
\end{array}\right.
\een
here
 $u=(u_1, u_2)$ denotes the velocity field,  $B=(B_1,B_2)$ is the magnetic field and $p$ is the pressure.
The MHD system  is coupled from the Navier-Stokes equation and the Maxwell equation. In the past, many physical experiments and numerical analyses have investigated the effect of magnetic field changes on the behavior of the fluid in the MHD system. For the physical background of the system, we can get more details from  \cite{Cab,G}.

The aim of this study is to give the global well-posedness for the mixed partial dissipative MHD system in half space, and the asymptotic behavior obtained from the mathematical research are conducive to providing theoretical basis for the experimental data.

For the fully dissipative MHD equations, Duvaut and Lions \cite{D}  and Sermange and Temam \cite{S} have done a lot of research on it. For the  non-resistant MHD system, Lin, Xu and Zhang \cite{LXZ} gave the  global well-posedness for the Cauchy problem around the equilibrium state  $(0, {\bf e}_1)$ for a class of  admissible perturbations.
Ren, Wu, Xiang and Zhang \cite{RWXZ} established the global existence and  time decay rate of smooth solutions. Ren, Xiang and Zhang \cite{RXZ} proved the MHD equations have a unique global strong solution around the equilibrium state in the strip domain.
Jin,  Kagei, Ren, Wang and Zhai \cite{J} obtain the global well-posedness and the asymptotic behavior of solution of non-resistive 2D MHD problem on the half space.

For the mixed partial dissipation MHD system \eqref{eq:MHD0}, Cao and Wu \cite{CW} gave the global regularity of its classical solution in $\R^2$. Lin, Ji, Wu and Yan \cite{L} provided  a global stability result on perturbations near a background magnetic field to the MHD systems  in $\R^2$.

In fact, when $B=0$, The study of systems \eqref{eq:MHD0} can be regarded as the study of the following systems:
\ben\label{eq:MHD00}
\left\{
\begin{array}{l}
\p_t u-\pa_{x_2}^2 u+u\cdot\na u+\na p=0,  \ \ \ \ x\in\Omega,\ t>0, \\
\div u=0,  \ \ \ \ x\in\Omega,\ t>0,\\
u(x,0)=u_0(x),\ \ \ \ x\in \Omega,
\end{array}\right.
\een
therefore, this problem is very valuable to study, and it can also show that the magnetic field has a good effect on the system.

In this paper, we consider the global well-posedness and long time behavior for \eqref{eq:MHD0} in the half space
$$
 \Om:=\big\{x=(x_1,x_2)\ |\ x_1 \in \R,\: x_2\in \R^+ \big\},
$$
with the most common boundary condition, i.e. the velocity field satisfy the classical non-slip boundary condition
$$
u=0\ \ \ \ \mathrm{on}\ \p\Om,
$$
and the container  is perfactly conducting for the magnetic field
$$
B\cdot n=0 \ \ \ \mathrm{on}\ \p\Om,
$$
here $n$  denotes the outward normal vector of $\p\Om$.

Motivated by  \cite{LXZ},  we will investigate   small perturbation of the system \eqref{eq:MHD0} around the  equilibrium state $(0, {\bf e}_1)$. Thus,  we can set  $b=B-{\bf e}_1$ and reformulate our first problem as the following initial-boundary problem
\ben\label{eq:MHDT}
\left\{
\begin{array}{l}
\p_t u-\pa_{x_2}^2 u-\p_{x_1} b+\na p=-u\cdot\na u+b\cdot \na b,  \ \ \ \ x\in\Omega,\ t>0,\\
\p_tb-\p_{x_1}^2 b-\p_{x_1} u=-u\cdot\na b+b\cdot \na u,  \ \ \ \ x\in\Omega,\ t>0,\\
\div u=\div b=0,  \ \ \ \ x\in\Omega,\ t>0,\\
u=0, \ \ \ b_2=0,   \ \ \ \ x\in\p\Omega,\ t>0,\\
u(x,0)=u_0(x),\ \ b(x,0)=b_0(x),\ \ \ \ x\in \Om.
\end{array}\right.
\een

We have the low order global solution
\begin{theorem}\label{thm:main}
Assume that the initial data $(u_0, b_0)$ satisfies $(u_0, b_0)\in  H^2(\Om)\cap L^2_{\sigma}(\Om)$ and
\begin{align*}
\|(u_0, b_0)\|_{H^2(\Om)}\lesssim \e
\end{align*}
with  $\e$ is a small positive constants, then  the MHD system (\ref{eq:MHDT}) has a unique global solution $(u, b)$ satisfying
\beno
(u, b)\in C([0,+\infty); H^2(\Om)).
\eeno
\end{theorem}
\begin{remark}
Here we need a zero boundary value on $b_1$ in order to avoid  the trouble boundary terms in {\it a-priori} estimate. This boundary condition could not be imposed. Fortunately, the boundary condition $b_1=0$ on $\p\Omega$  can be propagated by the equation \eqref{eq:MHDT}$_2$ if $b_1(0)=0$  on $\p\Omega$.

\end{remark}
\begin{remark}

We list the following the compatibility condition.

In deed, let $b_0, f \in C^\infty_0 (\Om)$ with $\div b_0=0, \div f=0$, if $\|b_0\|_{H^2}+\|f\|_{L^2}\leq \delta _0$, ($\delta _0>0$ is a constant), then there exists a $(u_0, p_0)$ with $u_0\in H^2(\Om)\cap L^2_{\sigma}(\Om)$, $\na p_0\in L^2(\Om)$ satisfies
\[
\left\{
\begin{array}{l}
u_0-\pa_{x_2}^2 u_0-\p_{x_1} b_0+\na p_0+u_0\cdot\na u_0-b_0\cdot \na b_0=f,  \ \ \ \ x\in\Omega,\ t>0,\\
\div u_0=0,  \ \ \ \ x\in\Omega,\ t>0,\\
u_0=0,   \ \ \ \ x\in\p\Omega,\ t>0,
\end{array}\right.
\]
then $(u_0, b_0)$ satisfies the compatibility condition since
\begin{align*}
-\pa_{x_2}^2 u_0+\na p_0-\p_{x_1} b_0+u_0 \cdot\na u_0-b_0\cdot\na b_0= f-u_0, \quad\cP(f-u_0)=f-u_0\in L^2_{\sigma}(\Om),
\end{align*}
where $u(0)=u_0$, $\p_t u(0)=f-u_0$.
\end{remark}

In addition, we also have asymptotic behavior 
\begin{theorem}\label{thm:main1}
Assume that the initial data  satisfies $(u_0, b_0)\in W^{1,1}(\Om)\cap H^2(\Om)$ and $(\pa_{x_1}u_0, \pa_{x_1}b_0)\in  H^2(\Om)$ and
\begin{align*}
\|(\pa_{x_1} u_0, \pa_{x_1} b_0)\|_{H^2(\Om)}\lesssim \e,
\end{align*}
with  $\e$ is a small positive constants, then  the MHD system (\ref{eq:MHDT}) has a unique global solution $(u, b)$ satisfying
\beno
(\pa_{x_1}u, \pa_{x_1}b)\in C([0,+\infty); H^2(\Om)).
\eeno
Moreover, it holds that
\begin{align}\label{eq:decay}
&\quad\|u\|_{L^2}\lesssim \lan t \ran^{-\f12},\qquad \|b_1\|_{L^2}\lesssim \lan t \ran^{-\f14},\qquad \|b_2\|_{L^2}\lesssim \lan t \ran^{-(\f12-\delta)},\quad\|\pa_{x_1} u\|_{L^2} \lesssim \lan t\ran^{-1},\non\\
&\|\p_{x_2} u_1\|_{L^2}\lesssim \lan t\ran^{-\f12}, \quad\|\p_{x_1} b_1\|_{L^2}\lesssim \lan t\ran^{-\f34}, \quad  \|\p_{x_1} b_2\|_{L^2}\lesssim \lan t\ran^{-(\f78-2\delta)},
\end{align}
for any $t\in [0,+\infty)$.
\end{theorem}

\begin{remark}

In fact, we do not need  the $L^\infty$ decay rate  to close the asymptotic behavior  $u$ and $b$, but the  $L^ \infty$ decay rate can be inferred through interpolation inequality.
\begin{align*}
\|u\|_{L^\infty}\lesssim\|u\|_{L^2}^\f14\|\pa_{x_1}u\|_{L^2}^\f14\|\pa_{x_2}u\|_{L^2}^\f14\|\pa_{x_1}\pa_{x_2}u\|_{L^2}^\f14\lesssim\lan t\ran^{-\f12},\\
\|b\|_{L^\infty}\lesssim \|b\|_{L^2}^\f14\|\pa_{x_1}b\|_{L^2}^\f14\|\pa_{x_2}b\|_{L^2}^\f14\|\pa_{x_1}\pa_{x_2}b\|_{L^2}^\f14\lesssim\lan t\ran^{-\f14}.
\end{align*}
Since we did not make the $L^\infty$ estimates, our $L^2$ attenuation is not optimal. In fact, we can also make the optimal attenuation, but because it makes the proof more tedious, for the sake of simplicity, we only need to pursue a long time behavior. 

\end{remark}

Next we introduce the main ideas and methods used in the proof.

\smallskip

{\bf Damped wave equation}

For full-space problems, the linear part of \eqref{eq:MHDT} is
\ben\label{eq:MHDdml}
\left\{
\begin{array}{l}
\p_t u-\pa_{x_2}^2 u-\pa_{x_1}b=0,\\
\p_t b-\pa_{x_1}^2 b-\pa_{x_1}u=0,
\end{array}\right.
\een
then taking the time derivative of both sides of \eqref{eq:MHDdml}, and use \eqref{eq:MHDdml}, we have the following form:
\begin{align*}
Y_{tt}-\Delta Y_{t}+\pa_{x_1}^2\pa_{x_2}^2Y-\pa_{x_1}^2Y=0.
\end{align*}

Inspired by \cite{RWXZ}, the solution has weak dissipation, and $Y$ can behave as
\begin{align*}
\wh{Y}(t,\xi)\sim a(\xi)e^{\lam_{+}(\xi)t}+b(\xi)e^{\lam_{-}(\xi)t},
\end{align*}
here
\begin{align}\label{lampm}
\lam_\pm=\left\{
\begin{array}{l}
-\f{|\xi_1|^2+|\xi_2|^2}{2}\pm\f{i \sqrt{4|\xi_1|^2-(|\xi_1|^2-|\xi_2|^2)^2}}{2}, \quad \big||\xi_1|^2-|\xi_2|^2\big|\leq 2|\xi_1|,\\
-\f{|\xi_1|^2+|\xi_2|^2}{2} \pm \f{\sqrt{(|\xi_1|^2-|\xi_2|^2)^2-4|\xi_1|^2}}{2} ,\quad \big||\xi_1|^2-|\xi_2|^2\big|> 2|\xi_1|.
\end{array}\right.\end{align}

{\bf  Global well-posedness:}

We conclude the global well-posedness in $H^2$ norm in half space with non-slip boundary condition and the main idea here is the use of Stokes estimate.

Firstly, by the non-slip  boundary conditions, the difficulty  is the energy estimate of $\pa_{x_2}^2u$ and $\pa_2^2b$.  For the energy  of $\pa_{x_2}^2u$, we use Stokes estimate to transform to $\p_t u$, and for the energy of $\pa_{x_2}^2b$, we need the dissipation estimate of $\pa_{x_2}^3u$, which can be rewritten as $\pa_{x_2}\cP\pa_{x_2}^2u$, by the appearance of pressure.  Secondly, in the analysis of nonlinear term of $\pa_{x_2}^2b$ energy estimate, due to the lack vertical magnetic dissipation, the  most trouble term is
 $$\lan \p_{x_2}^2 u, \p_{x_2}^2(b\cdot\na b)\ran,$$
we use the velocity  field equation and introduce $\na p$ dissipation. Lastly, we   also need the  cross term to provide the dissipation $\|\pa_{x_1}^2u\|_{L^2}$ to close the Stokes estimate.

{\bf Resolvent estimate:}

Inspired by \cite{J}, we will use the resolvent estimate method. The explicit solution for the linearized system  can be divided into full-space part and half-space part.

To the full-space part, because of the structure of $\lam_\pm$, we can discuss the division of $|\xi_1|,|\xi_2|$ into the following four regions: $\big||\xi_1|^2-|\xi_2|^2\big|\leq |\xi_1|$, $|\xi_1|<\big||\xi_1|^2-|\xi_2|^2\big|\leq 2|\xi_1|$, $2|\xi_1|<\big||\xi_1|^2-|\xi_2|^2\big|\leq 4|\xi_1|$, $\big||\xi_1|^2-|\xi_2|^2\big|> 4|\xi_1|$, and the  kernel $e^{-\f{|\xi_1|^2}{|\xi_2|^2}t-|\xi_1|^2t}$ get a weakest magnetic field decay rate.

 To the half-space part, by analyzing the structure of the solution, we find that the branch point is  $\lam=\lam_+^{'}$, $\lam=\lam_-^{'}$ and $\lam=-|\xi_1|^2$, so we choose the contour around the branch point and   consider each part of the contour seperately.

{\bf  Asymptotic behavior:}

With the help of linear decay rate, we consider the long time behavior of the nonlinear system. Here we did not make the $L^\infty$ estimates (for simplicity), and pursue low requirements for initial values and boundaries condition.  In fact, we utilize the idea of frequency localization and the structure of equations to close the nonlinear  decay estimates, with the aid of boundedness of $\|(u,b)\|_{H^2}$ and $\|(\pa_{x_1}u,\pa_{x_1}b)\|_{H^2}$.

\smallskip

The paper is organized as follows.  In section 2, we obtain the global well-posedness in $H^2$ without the aid of the decay rate of low order energy. In section 3, we present the solution formula  of the linearized problem and obtain the resolvent estimates for $u, b$ directly. In section 4,  we obtain the  large time behavior of the nonlinear estimate based on the linearized analysis and boundedness of global $H^2$ energy and $\|(\pa_{x_1}u,\pa_{x_1}b)\|_{H^2}$ energy in appendix,  which is a reasonable result since $\pa_{x_1}$ does not change the boundary conditions.

\smallskip

\no{\bf Notations.} Throughout this paper, for simplicity, we set  $\p_i=\f{\p}{\p x^i}$ for $i=1,2,3$,  $\p_t=\f{\p}{\p t}$. We will also use $A\lesssim B$ to denote the statement that $A \le CB$ for some absolute constant $C > 0$, which may be different on different lines. $\lan t\ran:=(1+t^2)^\f12$. $\wh{*}$ means the horizontal Fourier transform. Let $\vphi(\xi)$ be a smooth bump function adapted to $\{|\xi|\leq 2\}$ and equal to  1 on $\{|\xi|\leq 1\}$. For $j>0$, we define the Fourier multipliers
\begin{align*}
P_{\leq j} f&:=\cF^{-1}\Big(\varphi(\f{|\xi|}{j})\,\cF f(\xi)\Big), \quad P_{\geq j} f:=\cF^{-1}\Big((1-\varphi(\f{|\xi|}{j}))\,\cF f(\xi)\Big), \non\\
P_{j} f&:=\cF^{-1}\Big((\varphi(\f{2|\xi|}{j})-\varphi(\f{|\xi|}{j})\,\cF f(\xi)\Big),
\end{align*}
where j are dyadic number, that is the form of $2^{\Z}$ in general.\\

\section{Global well-posedness}
In this section, we will prove the $H^2$ a-priori estimate in half space and  obtain the Theorem \ref{thm:main} by continuous argument (the local well-posedness is  in appendix)  

We first introduce the following energy
\begin{align}\label{cE}
\cE^2(t):=\|u(t)\|^2_{H^2}+\|b(t)\|^2_{H^2}+\|\na p\|^2_{L^2}+\|(u_t,b_t)\|^2_{L^2},
\end{align}
and the dissipated energy
\begin{align}\label{cF}
\cF^2(t):&=\|\na u\|_{H^1}^2+\|\p_1 b\|_{H^2}^2 +\| u_t\|_{L^2}^2+\|b_t\|_{L^2}^2+\|\pa_1^2\pa_2 u\|_{L^2}^2\non\\
&+\|\pa_2 u_t\|_{L^2}^2+\|\pa_1 b_t\|_{L^2}^2+ \|\na p\|_{L^2}^2+\|\cP\pa_2^2u\|_{H^1}^2.
\end{align}

\begin{proposition}\label{high order}
Assume that the solution $(u,b)$ of the system \eqref{eq:MHDT} satisfies
\beno
\sup\limits_{0\le t\le T}\big(\|u(t)\|_{H^2}^2+\|b(t)\|_{H^2}^2\big)\leq  c_0^2.  \eeno
If $ c_0  $ is suitable small, then there hold that
\beno
\cE^2(t)+ \int_0^t\cF^2(s)ds\lesssim \|u_0\|^2_{H^2}+\|b_0\|^2_{H^2}
\eeno
for any  $t\in [0,T]$.
\end{proposition}

\begin{proof} For the half space problem, we will prove the a-priori estimate step by step.

{\bf Step 1.} $L^2$ estimate of $(u,b)$.

Thanks to $u=b=0$ on $\p\Om$, we take the $L^2$ inner product of equations $(\ref{eq:MHDT})_1$ and $(\ref{eq:MHDT})_2$ with $u$ and $b$, respectively, and integrate by parts to obtain
\begin{align}\label{L2}
\f12\f d {dt}\big(\|u(t)\|^2_{L^2}+\|b(t)\|^2_{L^2}\big)+\|\pa_2 u(t)\|^2_{L^2}+\|\pa_1 b(t)\|^2_{L^2}=0
\end{align}
for any $t\in [0, T]$.

{\bf Step 2.} $\dot{H}^1$ estimate of $(u,b)$.

To obtain the $\dot{H}^1$ estimate of $u$, we introduce the  Helmholtz projection
\beno
\cP: \ L^2(\Om)\rightarrow L^2_{\sigma}(\Om), \qquad L^2_{\sigma}(\Om)=\big\{v\:|\: v\in L^2, \div v=0, v\cdot n=0 \  \mathrm{on}\ \p\Om\big\}
\eeno
 to  eliminate the pressure term.

  We take the $L^2$ product of equation $(\ref{eq:MHDT})_1$ with $\cA u:=-\cP\Del u$, apply $\na$ to equation $(\ref{eq:MHDT})_2$ and then take the $L^2$ inner product of the resulting equation with $\na b$, to have
\begin{align}\label{dotH1}
&\f12\f d {dt}\big(\|\na u\|^2_{L^2}+\|\na b\|^2_{L^2}\big)+\|\cP \pa_2^2 u\|^2_{L^2}+\|\pa_1 \pa_2 u\|^2_{L^2}+\|\pa_1^2 b\|^2_{L^2}+\|\pa_1\pa_2 b\|^2_{L^2}\non\\
&=-\lan u\cdot\na u,\cA u\ran+\lan b\cdot\na b,\cA u\ran-\lan \na(u\cdot\na b),\na b\ran
+\lan\na(b\cdot\na u),\na b\ran,
\end{align}
where we use the integration by parts and boundary condition.

Next we deal with the right term of \eqref{dotH1}, by $div u=div b=0$, we have
\begin{align*}
&\lan u\cdot\na u, \cP\pa_2^2 u\ran+\lan b\cdot\na b, \cP\pa_2^2 u\ran\non\\
\lesssim&\|u_1\|_{L^\infty_{x_1}L^2_{x_2} } \|\pa_1 u\|_{L^2_{x_1}L^\infty_{x_2}}\|\cP\pa_2^2 u\|_{L^2}+\|u_2\|_{L^2_{x_1} L^\infty_{x_2}} \|\pa_2 u\|_{L^\infty_{x_1}L^2_{x_2}}\|\cP\pa_2^2 u\|_{L^2}\\
&+\|b_1\|_{ L^\infty_{x_1}L^2_{x_2}} \|\pa_1b \|_{L^2_{x_1}L^\infty_{x_2}}\|\cP\pa_2^2 u\|_{L^2}+\|b_2\|_{ L^2_{x_1}L^\infty_{x_2}} \|\pa_2b \|_{L^\infty_{x_1}L^2_{x_2}}\|\cP\pa_2^2 u\|_{L^2}\non\\
\lesssim&\|u\|_{L^2}^\f12\|\pa_2u_2\|_{L^2}^\f12\|u\|_{H^1}^\f12\|\pa_1\pa_2u\|_{L^2}^\f12\|\cP\pa_2^2 u\|_{L^2}+\|u\|_{L^2}^\f12\|\pa_2u_2\|_{L^2}^\f12\|\pa_2u\|_{L^2}^\f12\|\pa_1\pa_2u\|_{L^2}^\f12\|\cP\pa_2^2 u\|_{L^2}\\
&+\|b\|_{L^2}^\f12\|\pa_1b_1\|_{L^2}^\f12\|\pa_1b\|_{L^2}^\f12\|\pa_1b\|_{H^1}^\f12\|\cP\pa_2^2 u\|_{L^2}+\|b\|_{L^2}^\f12\|\pa_1b_1\|_{L^2}^\f12\|\pa_2b\|_{L^2}^\f12\|b\|_{H^1}^\f12\|\pa_1b\|_{H^1}^\f12\|\cP\pa_2^2 u\|_{L^2}\\
\lesssim&\|(u,b)\|_{H^1} \Big( \|\p_1 b\|_{H^1}^2 + \|\cP\pa_2^2 u\|_{L^2}^2+\|\pa_1\pa_2 u\|_{L^2}^2+\|\pa_2 u\|_{L^2}^2\Big),
\end{align*}
and
\begin{align*}
&\lan u\cdot\na u, \pa_1^2 u\ran+\lan b\cdot\na b, \pa_1^2 u\ran\non\\
=&-\lan \pa_1u_1\pa_1 u_1, \pa_1 u_1\ran-\lan \pa_1u_2\pa_2 u_1, \pa_1 u_1\ran-\lan \pa_1b\cdot\na b, \pa_1 u\ran-\lan b\cdot\na \pa_1b, \pa_1 u\ran\non\\
\lesssim&\|\pa_1u_1\|_{L^\infty_{x_1} L^2_{x_2}} \|\pa_1 u_1\|_{L^2_{x_1}L^\infty_{x_2}}\|\pa_1 u_1\|_{L^2}+\|\pa_1u_2\|_{L^2} \|\pa_2 u_1\|_{L^\infty_{x_1} L^2_{x_2}}\|\pa_1 u_1\|_{L^2_{x_1} L^\infty_{x_2}}\\
&+\|\pa_1b\|_{ L^2} \|\na b \|_{L^\infty_{x_1}L^2_{x_2}}\|\pa_1 u\|_{L^2_{x_1}L^\infty_{x_2}}+\|b\|_{L^\infty_{x_1} L^2_{x_2}} \|\pa_1\na b \|_{L^2}\|\pa_1 u\|_{ L^2_{x_1}L^\infty_{x_2}}\non\\
\lesssim&\|u\|_{H^1}\|\pa_1\pa_2u_1\|_{L^2}^\f12(\|\pa_1\pa_2u_2\|_{L^2}^\f12+\|\pa_1\pa_2u_1\|_{L^2}^\f12)\|\pa_2u_2\|_{L^2}+\|b\|_{H^1}^\f12\|\pa_1b\|_{H^1}^\f32\|u\|_{H^1}^\f12\|\pa_1\pa_2u\|_{L^2}^\f12\\
\lesssim&\|(u,b)\|_{H^1} \Big( \|\p_1 b\|_{H^1}^2 + \|\cP\pa_2^2 u\|_{L^2}^2+\|\pa_1\pa_2 u\|_{L^2}^2+\|\pa_2 u\|_{L^2}^2\Big),
\end{align*}
where we use the fact that
\begin{align*}
\lan u\cdot\na \pa_1u, \pa_1 u\ran=0,
\end{align*}
and
\begin{align*}
\lan \pa_1u_1\pa_1 u_2, \pa_1 u_2\ran+\lan \pa_1u_2\pa_2 u_2, \pa_1 u_2\ran=0.
\end{align*}

Next we deal with the remain terms of the term on the right side of \eqref{dotH1},
\begin{align*}
&\lan \na(u\cdot\na b), \na b\ran+\lan \na( b\cdot\na u), \na b\ran\non\\
=&\lan \pa_1 u\cdot\na b, \pa_1 b\ran+\lan \pa_2 u_1 \pa_1 b, \pa_2 b\ran+\lan \pa_2 u_2 \pa_2 b, \pa_2 b\ran+\lan b\cdot\na u, -\pa_1^2 b\ran\\
&+\lan \pa_2 b_1 \pa_1 u, \pa_2 b\ran+\lan  b_1 \pa_1 \pa_2u, \pa_2 b\ran+\lan \pa_2 b_2 \pa_2 u, \pa_2 b\ran+\lan  b_2 \pa_2 \pa_2u, \pa_2 b\ran
:=\sum\limits_{k=1}^{8}\mathcal{I}_{1k}.
\end{align*}

Then we have
\begin{align*}
\mathcal{I}_{11}+\mathcal{I}_{12}+\mathcal{I}_{13}&=\lan \pa_1u\cdot\na b, \pa_1 b\ran+\lan \pa_2 u_1 \pa_1 b, \pa_2 b\ran-\lan \pa_1 u_1 \pa_2 b, \pa_2 b\ran\\
&=\lan \pa_1u\cdot\na b, \pa_1 b\ran+\lan \pa_2 u_1 \pa_1 b, \pa_2 b\ran+2\lan  u_1 \pa_1\pa_2 b, \pa_2 b\ran\\
&\lesssim\|\pa_1u\|_{L^2_{x_1}L^\infty_{x_2}} \|\na b\|_{L^\infty_{x_1}L^2_{x_2}}\|\pa_1 b\|_{L^2}+\|\pa_2u\|_{L^2} \|\pa_1 b\|_{L^2_{x_1}L^\infty_{x_2}}\|\pa_2 b\|_{L^\infty_{x_1}L^2_{x_2}}\\
&\quad+\|u_1\|_{L^2_{x_1}L^\infty_{x_2}} \|\pa_1\pa_2 b\|_{L^2}\|\pa_2 b\|_{L^\infty_{x_1}L^2_{x_2}}\\
&\lesssim\| u\|_{H^1}^\f12\|\pa_1\pa_2u\|_{L^2}^\f12\|b\|_{H^1}^\f12\|\pa_1b\|_{H^1}^\f32+\|u\|_{H^1}^\f12\|\pa_2u\|_{L^2}^\f12\|\pa_1b\|_{H^1}^\f32\|b\|_{H^1}^\f12\\
&\quad+\| u\|_{H^1}^\f12\|\pa_2u\|_{L^2}^\f12\|b\|_{H^1}^\f12\|\pa_1b\|_{H^1}^\f32\\
&\lesssim\|(u,b)\|_{H^1} \Big( \|\p_1 b\|_{H^1}^2 + \|\cP\pa_2^2 u\|_{L^2}^2+\|\pa_1\pa_2 u\|_{L^2}^2+\|\pa_2 u\|_{L^2}^2\Big),
\end{align*}
and
\begin{align*}
\mathcal{I}_{14}+\mathcal{I}_{15}&=\lan b_1\pa_1 u, -\pa_1^2 b\ran+\lan b_2\pa_2 u, -\pa_1^2 b\ran+\lan \pa_2 b_1 \pa_1 u_1, \pa_2 b_1\ran+\lan \pa_2 b_1 \pa_1 u_2, \pa_2 b_2\ran\\
&=\lan b_1\pa_1 u, -\pa_1^2 b\ran+\lan b_2\pa_2 u, -\pa_1^2 b\ran-2\lan \pa_1\pa_2 b_1  u_1, \pa_2 b_1\ran+\lan \pa_2 b_1 \pa_1 u_2, \pa_2 b_2\ran\\
&\lesssim\|b_1\|_{L^\infty_{x_1}L^2_{x_2}} \|\pa_1 u\|_{L^2_{x_1}L^\infty_{x_2}}\|\pa_1^2 b\|_{L^2}+\|b_2\|_{L^2_{x_1}L^\infty_{x_2}} \|\pa_2 u\|_{L^\infty_{x_1}L^2_{x_2}}\|\pa_1^2 b\|_{L^2}\\
&\quad+\|\pa_1\pa_2b_1\|_{L^2} \| u_1\|_{L^2_{x_1}L^\infty_{x_2}}\|\pa_2 b_1\|_{L^\infty_{x_1}L^2_{x_2}}+\|\pa_2b_1\|_{L^\infty_{x_1}L^2_{x_2}} \| \pa_1u_2\|_{L^2_{x_1}L^\infty_{x_2}}\|\pa_2 b_2\|_{L^2}\\
&\lesssim\|b\|_{H^1}^\f12 \|\pa_1 b\|_{H^1}^\f12\|u\|_{H^1}^\f12\|\pa_1\pa_2u\|_{L^2}^\f12\|\pa_1b\|_{H^1}\\
&\quad+\|\pa_1 b\|_{H^1}\|u\|_{H^1}^\f12\|\pa_2u\|_{L^2}^\f12\|b\|_{H^1}^\f12\|\pa_1b\|_{H^1}^\f12+\|b\|_{H^1}^\f12\|\pa_1 b\|_{H^1}^\f12\|u\|_{H^1}^\f12\|\pa_1\pa_2u\|_{L^2}^\f12\|\pa_1b\|_{H^1}\\
&\lesssim\|(u,b)\|_{H^1} \Big( \|\p_1 b\|_{H^1}^2 + \|\cP\pa_2^2 u\|_{L^2}^2+\|\pa_1\pa_2 u\|_{L^2}^2+\|\pa_2 u\|_{L^2}^2\Big),
\end{align*}
and
\begin{align*}
\mathcal{I}_{16}+\mathcal{I}_{17}&=\lan  b_1 \pa_1 \pa_2u_1, \pa_2 b_1\ran+\lan  b_1 \pa_1 \pa_2u_2, \pa_2 b_2\ran+\lan \pa_2 b_2 \pa_2 u, \pa_2 b\ran\\
&=-\lan  \pa_1b_1  \pa_2u_1, \pa_2 b_1\ran-\lan  b_1  \pa_2u_1, \pa_1\pa_2 b_1\ran+\lan  b_1 \pa_1 \pa_2u_2, \pa_2 b_2\ran+\lan \pa_2 b_2 \pa_2 u, \pa_2 b\ran\\
&\lesssim\|\pa_1b_1\|_{L^2_{x_1}L^\infty_{x_2}} \|\pa_2 u_1\|_{L^2}\|\pa_2 b_1\|_{L^\infty_{x_1}L^2_{x_2}}+\|b_1\|_{L^2_{x_1}L^\infty_{x_2}} \|\pa_2 u_1\|_{L^\infty_{x_1}L^2_{x_2}}\|\pa_1\pa_2 b_1\|_{L^2}\\
&\quad+\|b_1\|_{L^\infty_{x_1}L^2_{x_2}} \|\pa_1\pa_2 u_2\|_{L^2}\|\pa_2 b_2\|_{L^2_{x_1}L^\infty_{x_2}}+\|\pa_2b_2\|_{L^2_{x_1}L^\infty_{x_2}} \| \pa_2u\|_{L^2}\|\pa_2 b\|_{L^\infty_{x_1}L^2_{x_2}}\\
&\lesssim\|\pa_1b\|_{H^1}\| u\|_{H^1}^\f12 \|\pa_2 u\|_{L^2}^\f12\|b\|_{H^1}^\f12\|\pa_1b\|_{H^1}^\f12+\|b\|_{H^1}\|\pa_2 u\|_{L^2}^\f12 \|\pa_1\pa_2 u\|_{L^2}^\f12 \|\pa_1 b\|_{H^1}\\
&\quad+\| b\|_{H^1}\|\pa_1\pa_2u\|_{L^2}\|\pa_1b\|_{H^1}+\|b\|_{H^1}^\f12\|\pa_1 b\|_{H^1}^\f12\|\pa_2u\|_{L^2}\|b\|_{H^1}^\f12\|\pa_1b\|_{H^1}^\f12\\
&\lesssim\|(u,b)\|_{H^1} \Big( \|\p_1 b\|_{H^1}^2 + \|\cP\pa_2^2 u\|_{L^2}^2+\|\pa_1\pa_2 u\|_{L^2}^2+\|\pa_2 u\|_{L^2}^2\Big),
\end{align*}
and
\begin{align*}
\mathcal{I}_{18}&=\lan  b_2  \pa_2^2u, \pa_2 b\ran\\
&=\lan  b_2\pa_2 b_1,  (\cP\pa_2^2u)_1 \ran+\lan  b_2\pa_2 b_1,  \pa_1\phi \ran+\lan  b_2\pa_2 b_2,  \pa_2^2u_2   \ran\\
&=\lan  b_2\pa_2 b_1,  (\cP\pa_2^2u)_1 \ran+\lan  -\pa_2b_2 b_1,  \pa_1\phi \ran+\lan  b_2 b_1,  -\pa_1\pa_2\phi \ran+\lan  b_2\pa_2 b_2,  -\pa_1\pa_2u_1   \ran\\
&=\lan  b_2\pa_2 b_1,  (\cP\pa_2^2u)_1 \ran+\lan  \pa_1b_1 b_1,  \pa_1\phi \ran+\lan  \pa_1b_2 b_1+b_2 \pa_1b_1,  \pa_2\phi \ran+\lan  b_2\pa_2 b_2,  -\pa_1\pa_2u_1   \ran\\
&=\lan  b_2\pa_2 b_1,  (\cP\pa_2^2u)_1 \ran+\lan  \pa_1b_1 b_1,  \pa_2^2u_1-(\cP\pa_2^2u)_1 \ran+\lan  \pa_1b_2 b_1+b_2 \pa_1b_1,  \pa_2^2u_2-(\cP\pa_2^2u)_2 \ran\\
&\quad+\lan  b_2\pa_2 b_2,  -\pa_1\pa_2u_1   \ran\\
&=\lan  b_2\pa_2 b_1,  (\cP\pa_2^2u)_1 \ran-\lan  \pa_1\pa_2b_1 b_1,  \pa_2u_1 \ran-\lan  \pa_1b_1 \pa_2b_1,  \pa_2u_1 \ran-\lan  \pa_1b_1 b_1,  (\cP\pa_2^2u)_1 \ran\\
&\quad+\lan  \pa_1b_2 b_1+b_2 \pa_1b_1,  -\pa_1\pa_2u_1-(\cP\pa_2^2u)_2 \ran+\lan  b_2\pa_2 b_2,  -\pa_1\pa_2u_1   \ran\\
&=\lan  b_2\pa_2 b_1,  (\cP\pa_2^2u)_1 \ran-\lan  \pa_1\pa_2b_1 b_1,  \pa_2u_1 \ran-\lan  \pa_1b_1 \pa_2b_1,  \pa_2u_1 \ran-\lan  \pa_1b_1 b_1,  (\cP\pa_2^2u)_1 \ran\\
&\quad+\lan  \pa_1b_2 b_1,  -\pa_1\pa_2u_1 \ran+\lan  \pa_1b_2 b_1+b_2 \pa_1b_1,  -(\cP\pa_2^2u)_2 \ran\\
&\lesssim\|b_2\|_{L^2_{x_1}L^\infty_{x_2}} \|\pa_2 b_1\|_{L^\infty_{x_1}L^2_{x_2}}\|\cP(\pa_2^2u)\|_{L^2}+\|\pa_1\pa_2b_1\|_{L^2} \| b_1\|_{L^2_{x_1}L^\infty_{x_2}}\|\pa_2 u_1\|_{L^\infty_{x_1}L^2_{x_2}}\\
&\quad+\|\pa_1b_1\|_{L^2_{x_1}L^\infty_{x_2}} \|\pa_2 b_1\|_{L^\infty_{x_1}L^2_{x_2}}\|\pa_2 u_1\|_{L^2}+\|\pa_1b_1\|_{L^2_{x_1}L^\infty_{x_2}} \| b_1\|_{L^\infty_{x_1}L^2_{x_2}}\|\cP(\pa_2^2u)\|_{L^2}\\
&\quad+\|\pa_1b_2\|_{L^2_{x_1}L^\infty_{x_2}} \| b_1\|_{L^\infty_{x_1}L^2_{x_2}}\|\pa_1\pa_2 u_1\|_{L^2}+\|\pa_1b\|_{L^2_{x_1}L^\infty_{x_2}} \| b\|_{L^\infty_{x_1}L^2_{x_2}}\|\cP(\pa_2^2u)\|_{L^2}\\
&\lesssim\|b\|_{H^1}\|\pa_1b\|_{H^1}\|\cP(\pa_2^2u)\|_{L^2} +\|b\|_{H^1}\|\pa_1b\|_{H^1}\|\pa_2u\|_{L^2}^\f12\|\pa_1\pa_2u\|_{L^2}^\f12\\
&\quad+\|b\|_{H^1}\|\pa_1b\|_{H^1}\|\pa_2u\|_{L^2}+\|b\|_{H^1}\|\pa_1b\|_{H^1}\|\cP(\pa_2^2u)\|_{L^2}\\
&\quad+\|b\|_{H^1}\|\pa_1b\|_{H^1}\|\pa_1\pa_2u\|_{L^2}+\|b\|_{H^1}\|\pa_1b\|_{H^1}\|\cP(\pa_2^2u)\|_{L^2}\\
&\lesssim\|(u,b)\|_{H^1} \Big( \|\p_1 b\|_{H^1}^2 + \|\cP\pa_2^2 u\|_{L^2}^2+\|\pa_1\pa_2 u\|_{L^2}^2+\|\pa_2 u\|_{L^2}^2\Big),
\end{align*}
where we use the fact that for any $v\in L^2(\Omega)$, there exists a unique Helmholtz decomposition
\beno
 \pa_2^2u=\cP\pa_2^2u+\na\phi.
\eeno

Thus we proof the result
\begin{align}\label{H1}
&\f12\f d {dt}\big(\| \na u\|^2_{L^2}+\| \na b\|^2_{L^2}\big)+\|\cP \pa_2^2 u\|^2_{L^2}+\|\pa_1 \pa_2 u\|^2_{L^2}+\|\pa_1^2 b\|^2_{L^2}+\|\pa_1\pa_2 b\|^2_{L^2}\non\\
&\lesssim\|(u,b)\|_{H^1} \Big( \|\p_1 b\|_{H^1}^2 + \|\cP\pa_2^2 u\|_{L^2}^2+\|\pa_1\pa_2 u\|_{L^2}^2+\|\pa_2 u\|_{L^2}^2\Big).
\end{align}

{\bf Step 3.} Dissipation estimate of  $\p_1u$.

By taking the $L^2$ product of equations $(\ref{eq:MHDT})_1$ and  $(\ref{eq:MHDT})_2$ with $\p_1 b$ and $-\p_1 u$, respectively, and using the integration by parts,  we deduce that
\begin{align}\label{p1u}
&\f d {dt}\lan \p_1b, u\ran+\|\p_1 u\|_{L^2}^2-\|\p_1 b\|_{L^2}^2+\lan \pa_2 u, \p_1\p_2 b\ran+\lan \pa_1 u, \p_1^2 b\ran\nonumber\\
&=-\lan \p_1 b,u\cdot\na u\ran+\lan \p_1b,b\cdot\na b\ran+\lan\p_1 u,u\cdot\na b\ran-\lan\p_1u,b\cdot\na u\ran\non\\
&\lesssim  \|\pa_1 b\|_{L^2}(\|u_1\|_{L^\infty_{x_1}L^2_{x_2}}\|\pa_1u\|_{L^2_{x_1}L^\infty_{x_2}}+\|u_2\|_{L^2_{x_1}L^\infty_{x_2}}\|\pa_2u\|_{L^\infty_{x_1}L^2_{x_2}})\non\\
&\quad+  \|\pa_1 b\|_{L^2}(\|b_1\|_{L^\infty_{x_1}L^2_{x_2}}\|\pa_1b\|_{L^2_{x_1}L^\infty_{x_2}}+\|b_2\|_{L^2_{x_1}L^\infty_{x_2}}\|\pa_2b\|_{L^\infty_{x_1}L^2_{x_2}})\non\\
&\quad+  \|\pa_1 u\|_{L^2}(\|u_1\|_{L^\infty_{x_1}L^2_{x_2}}\|\pa_1b\|_{L^2_{x_1}L^\infty_{x_2}}+\|u_2\|_{L^2_{x_1}L^\infty_{x_2}}\|\pa_2b\|_{L^\infty_{x_1}L^2_{x_2}})\non\\
&\quad+  \|\pa_1 u\|_{L^2}(\|b_1\|_{L^\infty_{x_1}L^2_{x_2}}\|\pa_1u\|_{L^2_{x_1}L^\infty_{x_2}}+\|b_2\|_{L^2_{x_1}L^\infty_{x_2}}\|\pa_2u\|_{L^\infty_{x_1}L^2_{x_2}})\non\\
&\lesssim  \|\pa_1 b\|_{L^2}\|u\|_{H^1}\|\p_2 u\|_{L^2}^\f12\|\pa_1\p_2 u\|_{L^2}^\f12+\|\pa_1 b\|_{L^2}\|b\|_{H^1}\|\p_1 b\|_{H^1}\non\\
&\quad+  \|\pa_1 u\|_{L^2}\|u\|_{H^1}^\f12\|b\|_{H^1}^\f12\|\p_2 u\|_{L^2}^\f12\|\p_1 b\|_{H^1}^\f12+\|\pa_1 u\|_{L^2}\|u\|_{H^1}^\f12\|b\|_{H^1}^\f12\|\pa_1\p_2 u\|_{L^2}^\f12\|\p_1 b\|_{H^1}^\f12\non\\
&\lesssim \|(u, b)\|_{H^1}\Big(\|\p_1 b\|_{H^1}^2 +\|\pa_1\pa_2 u\|_{L^2}^2+\|\pa_2 u\|_{L^2}^2+\|\pa_1 u\|_{L^2}^2 \Big).
\end{align}

{\bf Step 4.} Dissipation estimate of $(u_t,b_t)$.

Taking the $L^2$ product of  equations $(\ref{eq:MHDT})_1$  and  $(\ref{eq:MHDT})_2$  with $u_t$ and $b_t$, and using the integration by parts,   we have
\begin{align}\label{ut}
&\f12\f d {dt}(\|\pa_2 u\|^2_{L^2}+\|\pa_1 b\|^2_{L^2})+\|u_t\|^2_{L^2}+\|b_t\|_{L^2}^2-\lan u_t,\p_1 b\ran-\lan b_t,\p_1 u\ran\nonumber\\
&=-\lan u_t,u\cdot\na u\ran+\lan u_t, b\cdot\na b\ran
 -\lan b_t,u\cdot\na b\ran+\lan b_t,b\cdot\na u \ran\non\\
 &\lesssim  \|u_t\|_{L^2}(\|u_1\|_{L^\infty_{x_1}L^2_{x_2}}\|\pa_1u\|_{L^2_{x_1}L^\infty_{x_2}}+\|u_2\|_{L^2_{x_1}L^\infty_{x_2}}\|\pa_2u\|_{L^\infty_{x_1}L^2_{x_2}})\non\\
&\quad+  \|u_t\|_{L^2}(\|b_1\|_{L^\infty_{x_1}L^2_{x_2}}\|\pa_1b\|_{L^2_{x_1}L^\infty_{x_2}}+\|b_2\|_{L^2_{x_1}L^\infty_{x_2}}\|\pa_2b\|_{L^\infty_{x_1}L^2_{x_2}})\non\\
&\quad+  \|b_t\|_{L^2}(\|u_1\|_{L^\infty_{x_1}L^2_{x_2}}\|\pa_1b\|_{L^2_{x_1}L^\infty_{x_2}}+\|u_2\|_{L^2_{x_1}L^\infty_{x_2}}\|\pa_2b\|_{L^\infty_{x_1}L^2_{x_2}})\non\\
&\quad+  \|b_t\|_{L^2}(\|b_1\|_{L^\infty_{x_1}L^2_{x_2}}\|\pa_1u\|_{L^2_{x_1}L^\infty_{x_2}}+\|b_2\|_{L^2_{x_1}L^\infty_{x_2}}\|\pa_2u\|_{L^\infty_{x_1}L^2_{x_2}})\non\\
&\lesssim  \|u_t\|_{L^2}\|u\|_{H^1}\|\p_2 u\|_{L^2}^\f12\|\pa_1\p_2 u\|_{L^2}^\f12+\|u_t\|_{L^2}\|b\|_{H^1}\|\p_1 b\|_{H^1}\non\\
&\quad+  \|b_t\|_{L^2}\|u\|_{H^1}^\f12\|b\|_{H^1}^\f12\|\p_2 u\|_{L^2}^\f12\|\p_1 b\|_{H^1}^\f12+\|b_t\|_{L^2}\|u\|_{H^1}^\f12\|b\|_{H^1}^\f12\|\pa_1\p_2 u\|_{L^2}^\f12\|\p_1 b\|_{H^1}^\f12\non\\
&\lesssim \|(u, b)\|_{H^1}\Big(\|u_t\|_{L^2}^2+\|b_t\|_{L^2}^2+\|\p_1 b\|_{H^1}^2 +\|\pa_1\pa_2 u\|_{L^2}^2+\|\pa_2 u\|_{L^2}^2\Big).
\end{align}

{\bf Step 5.} Energy estimate of  $\p_1^2u$ and $\p_1^2b$ .

We apply $\pa_1^2$ to equation  $(\ref{eq:MHDT})_1$, $(\ref{eq:MHDT})_2$ and take the $L^2$ inner product of the resulting equation with $\pa_1^2 u$, $\pa_1^2 b$ to obtain
\begin{align}\label{p12u,p12b}
&\f12\f d {dt}(\|\pa_1^2 u\|_{L^2}^2+\|\pa_1^2 b\|_{L^2}^2)+(\|\p_1^2 \pa_2 u\|_{L^2}^2+\|\p_1^3  b\|_{L^2}^2)\non\\
=&-\lan \p_1^2 u ,\pa_1^2(u\cdot\na u)\ran+\lan \p_1^2 u ,\pa_1^2(b\cdot\na b)\ran-\lan \p_1^2 b ,\pa_1^2(u\cdot\na b)\ran+\lan \p_1^2 b ,\pa_1^2(b\cdot\na u)\ran\non\\
=&-\lan \p_1^2 u_1 ,\pa_1^2u_1\pa_1 u_1+\pa_1^2u_2\pa_2 u_1+2\pa_1u_1\pa_1^2 u_1+2\pa_1u_2\pa_1\pa_2 u_1\ran
\non\\
&+\lan \p_1^2 u ,\pa_1^2b_1\pa_1 b+\pa_1^2b_2\pa_2 b+2\pa_1b_1\pa_1^2 b+2\pa_1b_2\pa_1\pa_2 b\ran\non\\
&-\lan \p_1^2 b ,\pa_1^2u_1\pa_1 b+\pa_1^2u_2\pa_2 b+2\pa_1u_1\pa_1^2 b+2\pa_1u_2\pa_1\pa_2 b\ran\non\\
&+\lan \p_1^2 b ,\pa_1^2b_1\pa_1 u+\pa_1^2b_2\pa_2 u+2\pa_1b_1\pa_1^2 u+2\pa_1b_2\pa_1\pa_2 u\ran\non\\
&-\lan \p_1^2 u_2 ,\pa_1^2u_1\pa_1 u_2+\pa_1^2u_2\pa_2 u_2+2\pa_1u_1\pa_1^2 u_2+2\pa_1u_2\pa_1\pa_2 u_2\ran
\non\\
\lesssim& \|\pa_1^2 u_1\|_{L^2_{x_1}L^\infty_{x_2}}\|\pa_1\p_2u\|_{L^\infty_{x_1}L^2_{x_2}}\|\pa_1u_1\|_{L^2}+\|\pa_1^2 u_1\|_{L^2}\|\pa_1^2u_2\|_{L^2_{x_1}L^\infty_{x_2}}\|\pa_2u_1\|_{L^\infty_{x_1}L^2_{x_2}}\non\\
&+ \|\pa_1^2 u\|_{L^2_{x_1}L^\infty_{x_2}}(\|\pa_1\p_2b\|_{L^2}+\|\pa_1^2b\|_{L^2})(\|\pa_1b\|_{L^\infty_{x_1}L^2_{x_2}}+\|\pa_2b\|_{L^\infty_{x_1}L^2_{x_2}})\non\\
&+ \|\pa_1^2 b\|_{L^2}\|\pa_1^2 u\|_{L^2_{x_1}L^\infty_{x_2}}(\|\pa_1 b\|_{L^\infty_{x_1}L^2_{x_2}}+\|\pa_2 b\|_{L^\infty_{x_1}L^2_{x_2}})+ \|\pa_1^2 b\|_{L^2}\|\pa_1\pa_2 u\|_{L^\infty_{x_1}L^2_{x_2}}\|\pa_1 b\|_{L^2_{x_1}L^\infty_{x_2}}\non\\
&+ \|\pa_1^2 b\|_{L^2}\|\pa_1^2 b\|_{L^\infty_{x_1}L^2_{x_2}}(\|\pa_1 u\|_{L^2_{x_1}L^\infty_{x_2}}+\|\pa_2 u\|_{L^2_{x_1}L^\infty_{x_2}})+ \|\pa_1^2 b\|_{L^2}\|\pa_1\pa_2 b\|_{L^\infty_{x_1}L^2_{x_2}}\|\pa_1 u\|_{L^2_{x_1}L^\infty_{x_2}}\non\\
&+(\| \pa_1 b_{t}\|_{L^2}+\|\pa_1 b\|_{H^2}+\|\pa_1u_1\|_{L^2}\|\pa_1b_2\|_{L^\infty}+\|\pa_1u_2\|_{L^2}\|\pa_2b_2\|_{L^\infty}+\|u\|_{L^\infty}\|\pa_1\na b\|_{L^2}+\|b\|_{L^\infty}\|\pa_1\na u\|_{L^2} )\non\\
&(\|\pa_1u_1\|_{L^\infty_{x_1}L^2_{x_2}}\|\pa_1^2 u_2\|_{L^2_{x_1}L^\infty_{x_2}}+\|\pa_1u_2\|_{L^2_{x_1}L^\infty_{x_2}}\|\pa_1\pa_2 u_2\|_{L^\infty_{x_1}L^2_{x_2}})
\non\\
\lesssim&\Big(\|(u,b)\|_{H^2}+\|(u,b)\|_{H^2}^2\Big) \Big(\|\p_1b_t\|_{L^2}+ \|\p_1 b\|_{H^2}^2 +\|\pa_1^2\pa_2 u\|_{L^2}^2+\|\pa_1\pa_2 u\|_{L^2}^2+\|\pa_2 u\|_{L^2}^2\Big),
\end{align}
where we use the magnetic equation and the fact
\begin{align*}
\lan \p_1^2 u ,u\cdot\na \pa_1^2 u\ran=0, \quad \lan \p_1^2 b ,u\cdot\na \pa_1^2 b\ran=0, \quad \lan \p_1^2 u ,b\cdot\na\pa_1^2 b\ran+\lan \p_1^2 b ,b\cdot\na \pa_1^2 u\ran=0,
\end{align*}
and
\begin{align*}
&\lan \p_1^2 u_2 ,\pa_1^2u_1\pa_1 u_2+\pa_1^2u_2\pa_2 u_2+2\pa_1u_1\pa_1^2 u_2+2\pa_1u_2\pa_1\pa_2 u_2\ran
\non\\
=&\lan \p_1^2 u_2 ,\pa_1u_1\pa_1^2 u_2+\pa_1u_2\pa_1\pa_2 u_2\ran
\non\\
=&\lan \pa_1 b_{2,t}-\pa_1^3 b_{2}+\pa_1(u\cdot\na b_2-b\cdot \na u_2) ,\pa_1u_1\pa_1^2 u_2+\pa_1u_2\pa_1\pa_2 u_2\ran
.
\end{align*}

{\bf Step 6.} Energy estimate of  $\p_1\pa_2u$ and $\p_1\pa_2b$ .

We apply $\pa_1$ to equation  $(\ref{eq:MHDT})_1$, and take the $L^2$ inner product of the resulting equation with $-\pa_1\cP\pa_2^2 u$,  to obtain
\begin{align}\label{61}
&\f12\f d {dt}\|\pa_1\pa_2 u\|_{L^2}^2+\|\p_1 \cP \pa_2^2 u\|_{L^2}^2+\lan\pa_1\pa_2^2 u  ,\pa_1^2b\ran\non\\
=&\lan \pa_1\cP\pa_2^2 u ,\pa_1(u\cdot\na u)\ran-\lan \pa_1\cP\pa_2^2 u ,\pa_1(b\cdot\na b)\ran\non\\
=&\lan \pa_1\cP\pa_2^2 u ,\pa_1u_1\pa_1 u\ran+\lan \pa_1\cP\pa_2^2 u ,\pa_1u_2\pa_2 u\ran+\lan \pa_1\cP\pa_2^2 u ,u_1\pa_1^2 u\ran+\lan \pa_1\cP\pa_2^2 u ,u_2\pa_1\pa_2 u\ran\non\\
&-\lan \pa_1\cP\pa_2^2 u ,\pa_1b_1\pa_1 b\ran-\lan \pa_1\cP\pa_2^2 u ,\pa_1b_2\pa_2 b\ran-\lan \pa_1\cP\pa_2^2 u ,b_1\pa_1^2 b\ran-\lan \pa_1\cP\pa_2^2 u ,b_2\pa_1\pa_2 b\ran\non\\
\lesssim& \|\pa_1\cP\pa_2^2 u\|_{L^2}\Big(\|\pa_1 u_1\|_{L^\infty_{x_1}L^2_{x_2}}\|\pa_1 u\|_{L^2_{x_1}L^\infty_{x_2}}+\|\pa_2 u\|_{L^\infty_{x_1}L^2_{x_2}}\|\pa_1 u_2\|_{L^2_{x_1}L^\infty_{x_2}}+\| u_1\|_{L^\infty_{x_1}L^2_{x_2}}\|\pa_1^2 u\|_{L^2_{x_1}L^\infty_{x_2}}\non\\
&+ \| u_2\|_{L^2_{x_1}L^\infty_{x_2}}\|\pa_1\pa_2 u\|_{L^\infty_{x_1}L^2_{x_2}}+\|\pa_1 b_1\|_{L^\infty_{x_1}L^2_{x_2}}\|\pa_1 b\|_{L^2_{x_1}L^\infty_{x_2}}+\|\pa_2 b\|_{L^\infty_{x_1}L^2_{x_2}}\|\pa_1 b_2\|_{L^2_{x_1}L^\infty_{x_2}}\non\\
&+\| b_1\|_{L^\infty_{x_1}L^2_{x_2}}\|\pa_1^2 b\|_{L^2_{x_1}L^\infty_{x_2}}+ \| b_2\|_{L^2_{x_1}L^\infty_{x_2}}\|\pa_1\pa_2 b\|_{L^\infty_{x_1}L^2_{x_2}}\Big)\non\\
\lesssim& \|\pa_1\cP\pa_2^2 u\|_{L^2}\|(u,b)\|_{H^2}(\|\pa_2 u\|_{L^2}+\|\pa_1\pa_2 u\|_{L^2}+\|\pa_1^2\pa_2 u\|_{L^2}+\|\pa_1 b\|_{H^2})\non\\
\lesssim&\|(u,b)\|_{H^2} \Big( \|\p_1 b\|_{H^2}^2 + \|\pa_1\cP\pa_2^2 u\|_{L^2}^2+\|\pa_1^2\pa_2 u\|_{L^2}^2+\|\pa_1\pa_2 u\|_{L^2}^2+\|\pa_2 u\|_{L^2}^2\Big).
\end{align}

We apply $\pa_1$ to equation  $(\ref{eq:MHDT})_2$, and take the $L^2$ inner product of the resulting equation with $-\pa_1\pa_2^2 b$,  to obtain
\begin{align}\label{62}
&\f12\f d {dt}\|\pa_1\pa_2 b\|_{L^2}^2+\|\p_1^2 \pa_2 b\|_{L^2}^2+\lan\pa_1\pa_2^2 b  ,\pa_1^2u\ran\non\\
=&\lan \pa_1\pa_2^2 b ,\pa_1(u\cdot\na b)\ran-\lan \pa_1\pa_2^2 b ,\pa_1(b\cdot\na u)\ran\non\\
=&\lan \pa_1\pa_2^2 b ,\pa_1u_1\pa_1 b\ran+\lan \pa_1\pa_2^2 b ,\pa_1u_2\pa_2 b\ran+\lan \pa_1\pa_2^2 b ,u_1\pa_1^2 b\ran+\lan \pa_1\pa_2^2 b ,u_2\pa_1\pa_2 b\ran\non\\
&-\lan \pa_1\pa_2^2 b ,\pa_1b_1\pa_1 u\ran-\lan \pa_1\pa_2^2 b ,\pa_1b_2\pa_2 u\ran-\lan \pa_1\pa_2^2 b ,b_1\pa_1^2 u\ran-\lan \pa_1\pa_2^2 b ,b_2\pa_1\pa_2 u\ran\non\\
\lesssim& \|\pa_1\pa_2^2 b\|_{L^2}\Big(\|\pa_1 u_1\|_{L^\infty_{x_1}L^2_{x_2}}\|\pa_1 b\|_{L^2_{x_1}L^\infty_{x_2}}+\|\pa_2 b\|_{L^\infty_{x_1}L^2_{x_2}}\|\pa_1 u_2\|_{L^2_{x_1}L^\infty_{x_2}}+\| u_1\|_{L^\infty_{x_1}L^2_{x_2}}\|\pa_1^2 b\|_{L^2_{x_1}L^\infty_{x_2}}\non\\
&+ \| u_2\|_{L^2_{x_1}L^\infty_{x_2}}\|\pa_1\pa_2 b\|_{L^\infty_{x_1}L^2_{x_2}}+\|\pa_1 b_1\|_{L^\infty_{x_1}L^2_{x_2}}\|\pa_1 u\|_{L^2_{x_1}L^\infty_{x_2}}+\|\pa_2 u\|_{L^\infty_{x_1}L^2_{x_2}}\|\pa_1 b_2\|_{L^2_{x_1}L^\infty_{x_2}}\non\\
&+\| b_1\|_{L^\infty_{x_1}L^2_{x_2}}\|\pa_1^2 u\|_{L^2_{x_1}L^\infty_{x_2}}+ \| b_2\|_{L^2_{x_1}L^\infty_{x_2}}\|\pa_1\pa_2 u\|_{L^\infty_{x_1}L^2_{x_2}}\Big)\non\\
\lesssim& \|\pa_1 b\|_{H^2}\|(u,b)\|_{H^2}(\|\pa_2 u\|_{L^2}+\|\pa_1\pa_2 u\|_{L^2}+\|\pa_1^2\pa_2 u\|_{L^2}+\|\pa_1 b\|_{H^2})\non\\
\lesssim&\|(u,b)\|_{H^2} \Big( \|\p_1 b\|_{H^2}^2 +\|\pa_1^2\pa_2 u\|_{L^2}^2+\|\pa_1\pa_2 u\|_{L^2}^2+\|\pa_2 u\|_{L^2}^2\Big).
\end{align}

 Combine \eqref{61} and \eqref{62}, we proof the result
\begin{align}\label{p1p2ub}
&\f12\f d {dt}\big(\| \pa_1\pa_2u\|^2_{L^2}+\| \pa_1\pa_2b\|^2_{L^2}\big)+\|\p_1 \cP \pa_2^2 u\|_{L^2}^2+\|\p_1^2 \pa_2 b\|_{L^2}^2\non\\
&\lesssim\|(u,b)\|_{H^2} \Big( \|\p_1 b\|_{H^2}^2 + \|\pa_1\cP\pa_2^2 u\|_{L^2}^2+\|\pa_1^2\pa_2 u\|_{L^2}^2+\|\pa_1\pa_2 u\|_{L^2}^2+\|\pa_2 u\|_{L^2}^2\Big).
\end{align}
{\bf Step 7.} Energy estimate of $u_t$.

Applying $\p_t$ to equation $(\ref{eq:MHDT})_1$ and taking the $L^2$ inner product of the resulting equation with $u_t$, and apply $\p_t$ to equation $(\ref{eq:MHDT})_2$ and taking the $L^2$ inner product of the resulting equation with $b_t$, we obtain
\begin{align}\label{H2u}
&\f12\f d {dt}(\|u_t\|_{L^2}^2+\|b_t\|_{L^2}^2)+\|\pa_2 u_t\|_{L^2}^2+\|\pa_1 b_t\|_{L^2}^2\non\\
&=-\lan u_t,\p_t(u\cdot\na u)\ran+\lan u_t,\p_t(b\cdot\na b)\ran-\lan b_t,\p_t(u\cdot\na b)\ran+\lan b_t,\p_t(b\cdot\na u)\ran\non\\
&=-\lan u_t,u_t\cdot\na u\ran+\lan u_t,b_t\cdot\na b\ran-\lan b_t,u_t\cdot\na b\ran+\lan b_t,b_t\cdot\na u\ran\non\\
&=-\lan u_{1,t}, u_{1,t} \pa_1 {u_1}\ran-\lan u_{1,t}, u_{2,t} \pa_2 {u_1}\ran-\lan u_{2,t}, u_{1,t} \pa_1 {u_2}\ran-\lan u_{2,t}, u_{2,t} \pa_2 {u_2}\ran\non\\
&\quad+\lan u_{1,t}, b_{1,t} \pa_1 {b_1}\ran+\lan u_{1,t}, b_{2,t} \pa_2 {b_1}\ran+\lan u_{2,t}, b_{1,t} \pa_1 {b_2}\ran+\lan u_{2,t}, b_{2,t} \pa_2 {b_2}\ran\non\\
&\quad-\lan b_{1,t}, u_{1,t} \pa_1 {b_1}\ran-\lan b_{1,t}, u_{2,t} \pa_2 {b_1}\ran-\lan b_{2,t}, u_{1,t} \pa_1 {b_2}\ran-\lan b_{2,t}, u_{2,t} \pa_2 {b_2}\ran\non\\
&\quad+\lan b_{1,t}, b_{1,t} \pa_1 {u_1}\ran+\lan b_{1,t}, b_{2,t} \pa_2 {u_1}\ran+\lan b_{2,t}, b_{1,t} \pa_1 {u_2}\ran+\lan b_{2,t}, b_{2,t} \pa_2 {u_2}\ran\non\\
&\lesssim \Big(\|u_t\|_{L^2}+\|b_t\|_{L^2}\Big)\Big(\| u_{1,t}\|_{L^\infty_{x_1}L^2_{x_2}}\|\pa_1u\|_{L^2_{x_1}L^\infty_{x_2}}+\| u_{2,t}\|_{L^2_{x_1}L^\infty_{x_2}}\|\pa_2u\|_{L^\infty_{x_1}L^2_{x_2}}\non\\
&\quad+\| b_{1,t}\|_{L^\infty_{x_1}L^2_{x_2}}\|\pa_1b\|_{L^2_{x_1}L^\infty_{x_2}}+\| b_{2,t}\|_{L^2_{x_1}L^\infty_{x_2}}\|\pa_2b\|_{L^\infty_{x_1}L^2_{x_2}}\Big)\non\\
&\lesssim \|(u_t,b_t)\|_{L^2}^\f32\|(\pa_2 u_t, \pa_1 b_t )\|_{L^2}^\f12\|(u,b)\|_{H^2}\non\\
&\lesssim \|(u, b)\|_{H^2}\Big(\|u_t\|_{L^2}^2+\|b_t\|_{L^2}^2+\|\pa_2u_t\|_{L^2}^2+\|\pa_1b_t\|_{L^2}^2\Big).
\end{align}

{\bf Step 8.} Dissipation estimate of  $\pa_2\cP\pa_2^2u$.

We apply $\pa_2\cP$ to equation  $(\ref{eq:MHDT})_1$, and take the $L^2$ inner product of the resulting equation with $-\pa_2\cP\pa_2^2u$  to obtain
\begin{align}\label{p2p2u}
&\|\pa_2\cP\p_2^2 u\|_{L^2}^2-\lan \pa_2\cP\p_2^2 u, \pa_2u_t\ran+\lan \pa_2\cP\p_2^2 u, \p_1\p_2 b\ran\nonumber\\
=&-\lan \p_2\cP\p_2^2 u,\pa_2(u\cdot\na u)\ran+\lan \p_2\cP\p_2^2 u,\pa_2(b\cdot\na b)\ran\non\\
\lesssim& \|\p_2\cP\p_2^2 u\|_{L^2}\Big(\|\pa_2 u_1\|_{L^\infty_{x_1}L^2_{x_2}}\|\pa_1 u\|_{L^2_{x_1}L^\infty_{x_2}}+\|\pa_2 u_2\|_{L^\infty_{x_1}L^2_{x_2}}\|\pa_2 u\|_{L^2_{x_1}L^\infty_{x_2}}+\| u_1\|_{L^2_{x_1}L^\infty_{x_2}}\|\pa_1\pa_2 u\|_{L^\infty_{x_1}L^2_{x_2}}\non\\
&+ \| u_2\|_{L^\infty}\|\pa_2^2 u\|_{L^2}+\|\pa_2 b_1\|_{L^\infty_{x_1}L^2_{x_2}}\|\pa_1 b\|_{L^2_{x_1}L^\infty_{x_2}}+\|\pa_2 b_2\|_{L^2_{x_1}L^\infty_{x_2}}\|\pa_2 b\|_{L^\infty_{x_1}L^2_{x_2}}\non\\
&+\| b_1\|_{L^\infty_{x_1}L^2_{x_2}}\|\pa_1\pa_2 b\|_{L^2_{x_1}L^\infty_{x_2}}+ \| b_2\|_{L^2_{x_1}L^\infty_{x_2}}\|\pa_2^2 b\|_{L^\infty_{x_1}L^2_{x_2}}\Big)\non\\
\lesssim& \|\p_2\cP\p_2^2 u\|_{L^2}\|(u,b)\|_{H^2}(\|\pa_1^2\pa_2u\|_{L^2}+\|\pa_2 u\|_{H^1}+\|\pa_1 b\|_{H^2})\non\\
\lesssim&\|(u,b)\|_{H^2} \Big( \|\p_1 b\|_{H^2}^2 +\|\pa_1^2\pa_2 u\|_{L^2}^2+\|\pa_2 u\|_{H^1}^2+\|\p_2\cP\p_2^2 u\|_{L^2}^2\Big).
\end{align}

{\bf Step 9.} Dissipation estimate of  $\p_1^2u$.

We apply $\pa_1$ to equation  $(\ref{eq:MHDT})_2$, and take the $L^2$ inner product of the resulting equation with $-\pa_1^2 u$, and apply $\pa_1$ to equation  $(\ref{eq:MHDT})_1$, and take the $L^2$ inner product of the resulting equation with $\pa_1^2 b$  to obtain
\begin{align}\label{p12u}
&\f d {dt}\lan \pa_1^2b,\p_1 u\ran+\|\p_1^2 u\|_{L^2}^2-\|\p_1^2 b\|_{L^2}^2+\lan \pa_1\pa_2 u, \p_1^2\pa_2 b\ran+\lan \pa_1^2 u, \p_1^3 b\ran\nonumber\\
=&-\lan \p_1^2 b,\pa_1(u\cdot\na u)\ran+\lan \p_1^2b,\pa_1(b\cdot\na b)\ran+\lan\p_1^2 u,\pa_1(u\cdot\na b)\ran-\lan\p_1^2u,\pa_1(b\cdot\na u)\ran\non\\
\lesssim& \|\pa_1^2 b\|_{L^2}\Big(\|\pa_1 u_1\|_{L^\infty_{x_1}L^2_{x_2}}\|\pa_1 u\|_{L^2_{x_1}L^\infty_{x_2}}+\|\pa_2 u\|_{L^\infty_{x_1}L^2_{x_2}}\|\pa_1 u_2\|_{L^2_{x_1}L^\infty_{x_2}}+\| u_1\|_{L^\infty_{x_1}L^2_{x_2}}\|\pa_1^2 u\|_{L^2_{x_1}L^\infty_{x_2}}\non\\
&+ \| u_2\|_{L^2_{x_1}L^\infty_{x_2}}\|\pa_1\pa_2 u\|_{L^\infty_{x_1}L^2_{x_2}}+\|\pa_1 b_1\|_{L^\infty_{x_1}L^2_{x_2}}\|\pa_1 b\|_{L^2_{x_1}L^\infty_{x_2}}+\|\pa_2 b\|_{L^\infty_{x_1}L^2_{x_2}}\|\pa_1 b_2\|_{L^2_{x_1}L^\infty_{x_2}}\non\\
&+\| b_1\|_{L^\infty_{x_1}L^2_{x_2}}\|\pa_1^2 b\|_{L^2_{x_1}L^\infty_{x_2}}+ \| b_2\|_{L^2_{x_1}L^\infty_{x_2}}\|\pa_1\pa_2 b\|_{L^\infty_{x_1}L^2_{x_2}}\Big)\non\\
+&\|\pa_1^2 u\|_{L^2}\Big(\|\pa_1 u_1\|_{L^\infty_{x_1}L^2_{x_2}}\|\pa_1 b\|_{L^2_{x_1}L^\infty_{x_2}}+\|\pa_2 b\|_{L^\infty_{x_1}L^2_{x_2}}\|\pa_1 u_2\|_{L^2_{x_1}L^\infty_{x_2}}+\| u_1\|_{L^\infty_{x_1}L^2_{x_2}}\|\pa_1^2 b\|_{L^2_{x_1}L^\infty_{x_2}}\non\\
&+ \| u_2\|_{L^2_{x_1}L^\infty_{x_2}}\|\pa_1\pa_2 b\|_{L^\infty_{x_1}L^2_{x_2}}+\|\pa_1 b_1\|_{L^\infty_{x_1}L^2_{x_2}}\|\pa_1 u\|_{L^2_{x_1}L^\infty_{x_2}}+\|\pa_2 u\|_{L^\infty_{x_1}L^2_{x_2}}\|\pa_1 b_2\|_{L^2_{x_1}L^\infty_{x_2}}\non\\
&+\| b_1\|_{L^\infty_{x_1}L^2_{x_2}}\|\pa_1^2 u\|_{L^2_{x_1}L^\infty_{x_2}}+ \| b_2\|_{L^2_{x_1}L^\infty_{x_2}}\|\pa_1\pa_2 u\|_{L^\infty_{x_1}L^2_{x_2}}\Big)\non\\
\lesssim& \Big(\|\pa_1 b\|_{H^2}+\|\pa_1^2 u\|_{L^2}\Big)\|(u,b)\|_{H^2}\Big(\|\pa_2u\|_{H^1}+\|\pa_1^2\pa_2 u\|_{L^2}+\|\pa_1 b\|_{H^2}\Big)\non\\
\lesssim&\|(u,b)\|_{H^2} \Big( \|\p_1 b\|_{H^2}^2 +\|\pa_1^2\pa_2 u\|_{L^2}^2+\|\pa_2 u\|_{H^1}^2+\|\pa_1^2 u\|_{L^2}^2\Big).
\end{align}
{\bf Step 10.} Energy estimate of   $\pa_2^2b$ and the dissipation estimate of  $\p_2u_t$ .

We apply $\pa_2^2$ to equation  $(\ref{eq:MHDT})_2$, and take the $L^2$ inner product of the resulting equation with $\pa_2^2 b$, and apply $\pa_2\cP$ to equation  $(\ref{eq:MHDT})_1$£¬and take the $L^2$ inner product of the equation  $(\ref{eq:MHDT})_1$ with $\pa_2 u_t$   to obtain
\begin{align}\label{p2b p2ut}
&\f12\f d {dt}\|\pa_2^2 b\|_{L^2}^2+\|\p_1\pa_2^2  b\|_{L^2}^2+\|\pa_2  u_t\|_{L^2}^2+\lan \pa_1\pa_2^2b,\pa_2^2u\ran-\lan \pa_1\pa_2b,\pa_2u_t\ran-\lan \pa_2\cP\pa_2^2u,\pa_2u_t\ran\non\\
=&\lan \p_2 u_t ,\pa_2\cP(-u\cdot\na u)\ran+\lan \p_2 u_t ,\p_2\cP(b\cdot\na b)\ran+\lan \p_2^2 b ,\pa_2^2(-u\cdot\na b)\ran+\lan \p_2^2 b ,\pa_2^2(b\cdot\na u)\ran\non\\
=&\lan \p_2 u_t ,\pa_2(-u\cdot\na u)-\p_2\na\cU\ran+\lan \p_2 u_t ,\p_2(b\cdot\na b)-\p_2\na\cB\ran+\lan \p_2^2 b ,\pa_2^2(-u\cdot\na b)\ran+\lan \p_2^2 b ,\pa_2^2(b\cdot\na u)\ran\non\\
=&\lan \p_2 u_t ,\pa_2(-u\cdot\na u)\ran+\lan \p_2 u_t ,\p_2(b\cdot\na b)\ran+\lan \p_2^2 b ,\pa_2^2(-u\cdot\na b)\ran+\lan \p_2^2 b ,\pa_2^2(b\cdot\na u)\ran\non\\
=&\lan \p_2 u_t ,\pa_2(-u\cdot\na u)\ran+\lan \p_2 u_t ,\p_2(b\cdot\na b)\ran-\lan \p_2^2 b ,\pa_2^2u\cdot\na b+2\pa_2u\cdot\na\pa_2 b\ran\non\\
&+\lan \p_2^2 b ,\pa_2^2b\cdot\na u+2\pa_2b\cdot\na\pa_2 u\ran+\lan \p_2^2 b ,b\cdot\na\pa_2^2u\ran\non\\
=&\lan \p_2 u_t ,\pa_2(-u\cdot\na u)\ran+\lan \p_2 u_t ,\p_2(b\cdot\na b)\ran-\lan \p_2^2 b ,\pa_2^2u\cdot\na b+2\pa_2u\cdot\na\pa_2 b\ran\non\\
&+\lan \p_2^2 b ,\pa_2^2b\cdot\na u+2\pa_2b\cdot\na\pa_2 u\ran-\lan \p_2^2 u ,b\cdot\na\pa_2^2b\ran\non\\
=&\lan \p_2 u_t ,\pa_2(-u\cdot\na u)\ran+\lan \p_2 u_t ,\p_2(b\cdot\na b)\ran-\lan \p_2^2 b ,\pa_2^2u\cdot\na b+2\pa_2u\cdot\na\pa_2 b\ran\non\\
&+\lan \p_2^2 b ,\pa_2^2b\cdot\na u+2\pa_2b\cdot\na\pa_2 u\ran-\lan \p_2^2 u ,\pa_2^2(b\cdot\na b)\ran+\lan \p_2^2 u ,\pa_2^2b\cdot\na b+2\pa_2b\cdot\na\pa_2 b\ran\non\\
=&\lan \p_2 u_t ,\pa_2(-u\cdot\na u)\ran+\lan \p_2 u_t ,\p_2(b\cdot\na b)\ran-\lan \p_2^2 b ,\pa_2^2u\cdot\na b+2\pa_2u\cdot\na\pa_2 b\ran\non\\
&+\lan \p_2^2 b ,\pa_2^2b\cdot\na u+2\pa_2b\cdot\na\pa_2 u\ran+\lan \p_2^2 u ,\pa_2^2b\cdot\na b+2\pa_2b\cdot\na\pa_2 b\ran\non\\
&-\lan u_t+\na p-\pa_1b+u\cdot\na u-b\cdot\na b  ,\pa_2^2(b\cdot\na b)\ran\non\\
=&\lan \p_2 u_t ,\pa_2(-u\cdot\na u)\ran+2\lan \p_2 u_t ,\p_2(b\cdot\na b)\ran-\lan \p_2^2 b ,\pa_2^2u\cdot\na b+2\pa_2u\cdot\na\pa_2 b\ran\non\\
&+\lan \p_2^2 b ,\pa_2^2b\cdot\na u+2\pa_2b\cdot\na\pa_2 u\ran+\lan \p_2^2 u ,\pa_2^2b\cdot\na b+2\pa_2b\cdot\na\pa_2 b\ran\non\\
&-\lan \na p-\pa_1b+u\cdot\na u-b\cdot\na b  ,\pa_2^2(b\cdot\na b)\ran\non\\
=&\mathcal{I}_{21}+\mathcal{I}_{22}+\mathcal{I}_{23}+\mathcal{I}_{24}+\mathcal{I}_{25}+\mathcal{I}_{26},
\end{align}
where we define
\begin{align*}
-u\cdot\na u=\cP(-u\cdot\na u)+\na\cU, \  -b\cdot\na b=\cP(-b\cdot\na b)+\na\cB,
\end{align*}
and use
\begin{align*}
&\lan \p_2 u_t ,\p_2\na(\cU+\cB)\ran=\lan \p_2 u_{1,t} ,\p_1\p_2(\cU+\cB)\ran+\lan \p_2 u_{2,t} ,\p_2^2(\cU+\cB)\ran=-\lan \p_2(\pa_1 u_{1,t}+\pa_2 u_{2,t}) ,\p_2(\cU+\cB)\ran=0.
\end{align*}

Next we compute the right term of the \eqref{p2b p2ut}, we have
\begin{align*}
\mathcal{I}_{21}+\mathcal{I}_{22}\lesssim &\|\p_2 u_t\|_{L^2}(\|u\|_{H^2}\|\pa_2 u\|_{H^1}+\|b\|_{H^2}\|\pa_1 b\|_{H^2})\\
\lesssim&\|(u,b)\|_{H^2} \Big( \|\pa_1 b\|_{H^2}^2+\|\pa_2 u\|_{H^1}^2+\|\p_2 u_t\|_{L^2}^2\Big),
\end{align*}
here we use
\begin{align}
\|\pa_2(u\cdot \na u)\|_{L^2}\lesssim&\|\pa_2u\|_{L^\infty_{x_1}L^2_{x_2}}(\|\pa_1u\|_{L^2_{x_1}L^\infty_{x_2}}+\|\pa_2u\|_{L^2_{x_1}L^\infty_{x_2}})+\|u\|_{L^\infty}(\|\pa_1\pa_2u\|_{L^2}+\|\pa_2^2u\|_{L^2})\non\\
\lesssim&\|u\|_{H^2}\|\pa_2u\|_{H^1},\label{u001}
\end{align}
and
\begin{align}
\|\pa_2(b\cdot \na b)\|_{L^2}\lesssim&\|\pa_2b_1\|_{L^\infty_{x_1}L^2_{x_2}}\|\pa_1b\|_{L^2_{x_1}L^\infty_{x_2}}+\|\pa_2b_2\|_{L^2_{x_1}L^\infty_{x_2}}\|\pa_2b\|_{L^\infty_{x_1}L^2_{x_2}}\non\\
&+\|b_1\|_{L^\infty}\|\pa_1\pa_2b\|_{L^2}+\|b_2\|_{L^2_{x_1}L^\infty_{x_2}}\|\pa_2^2b\|_{L^\infty_{x_1}L^2_{x_2}}\lesssim\|b\|_{H^2}\|\pa_1b\|_{H^2}.\label{b001}
\end{align}

And we have
\begin{align*}
\mathcal{I}_{23}+\mathcal{I}_{24}=&-\lan \p_2^2 b_1 ,\pa_2^2u_1\pa_1 b_1+\pa_2^2u_2\pa_2 b_1+2\pa_2u_1\pa_1\pa_2 b_1+2\pa_2u_2\pa_2^2 b_1\ran\non\\
&-\lan \p_2^2 b_2 ,\pa_2^2u_1\pa_1 b_2+\pa_2^2u_2\pa_2 b_2+2\pa_2u_1\pa_1\pa_2 b_2+2\pa_2u_2\pa_2^2 b_2\ran\non\\
&+\lan \p_2^2 b_1 ,\pa_2^2b_1\pa_1 u_1+\pa_2^2b_2\pa_2 u_1+2\pa_2b_1\pa_1\pa_2 u_1+2\pa_2b_2\pa_2^2 u_1\ran\non\\
&+\lan \p_2^2 b_2 ,\pa_2^2b_1\pa_1 u_2+\pa_2^2b_2\pa_2 u_2+2\pa_2b_1\pa_1\pa_2 u_2+2\pa_2b_2\pa_2^2 u_2\ran\non\\
=&-3\lan \p_2^2 b_1 ,\pa_2^2u_1\pa_1 b_1+\pa_2^2u_2\pa_2 b_1+\pa_2u_1\pa_1\pa_2 b_1+\pa_2u_2\pa_2^2 b_1\ran\non\\
&-\lan \p_2^2 b_2 ,\pa_2^2u_1\pa_1 b_2+2\pa_2u_1\pa_1\pa_2 b_2\ran+\lan \p_2^2 b_2 ,\pa_2^2b_1\pa_1 u_2+2\pa_2b_1\pa_1\pa_2 u_2\ran\non\\
&+\lan \p_2^2 b_2 ,\pa_2^2u_2\pa_2 b_2-\pa_2u_2\pa_2^2 b_2\ran\non\\
=&-3\lan \p_2^2 b_1 ,\pa_2^2u_1\pa_1 b_1-\pa_1\pa_2u_1\pa_2 b_1+\pa_2u_1\pa_1\pa_2 b_1+2u_1\pa_1\pa_2^2 b_1\ran\non\\
&-\lan \p_2^2 b_2 ,\pa_2^2u_1\pa_1 b_2+2\pa_2u_1\pa_1\pa_2 b_2\ran+\lan \p_2^2 b_2 ,\pa_2^2b_1\pa_1 u_2+2\pa_2b_1\pa_1\pa_2 u_2\ran\non\\
&+\lan \p_2^2 b_2 ,\pa_2^2u_2\pa_2 b_2-\pa_2u_2\pa_2^2 b_2\ran\non\\
=&-3\lan \p_2^2 b_1 ,\pa_2^2u_1\pa_1 b_1+2\pa_2u_1\pa_1\pa_2 b_1+2u_1\pa_1\pa_2^2 b_1\ran-3\lan \pa_1\p_2^2 b_1,\pa_2u_1\pa_2 b_1\ran\non\\
&-\lan \p_2^2 b_2 ,\pa_2^2u_1\pa_1 b_2+2\pa_2u_1\pa_1\pa_2 b_2\ran+\lan \p_2^2 b_2 ,\pa_2^2b_1\pa_1 u_2+2\pa_2b_1\pa_1\pa_2 u_2\ran\non\\
&+\lan \p_2^2 b_2 ,\pa_2^2u_2\pa_2 b_2-\pa_2u_2\pa_2^2 b_2\ran\non\\
\lesssim&\|\pa_2^2b_1\|_{L^\infty_{x_1}L^2_{x_2}}(\|\pa_2^2u_1\|_{L^2}\|\pa_1b_1\|_{L^2_{x_1}L^\infty_{x_2}}+\|\pa_2u_1\|_{L^2_{x_1}L^\infty_{x_2}}\|\pa_1\pa_2b_1\|_{L^2}+\|u_1\|_{L^2_{x_1}L^\infty_{x_2}}\|\pa_1\pa_2^2b_1\|_{L^2})\\
&+\|\pa_2^2b_2\|_{L^\infty_{x_1}L^2_{x_2}}(\|\pa_2^2u\|_{L^2}\|\pa_1b\|_{L^2_{x_1}L^\infty_{x_2}}+\|\pa_2u\|_{L^2_{x_1}L^\infty_{x_2}}\|\pa_1\pa_2b\|_{L^2})\\
&+\|\pa_2^2b_2\|_{L^2}\|\pa_2^2b_1\|_{L^\infty_{x_1}L^2_{x_2}}\|\pa_1u_2\|_{L^2_{x_1}L^\infty_{x_2}}+\|\pa_2^2b_2\|_{L^2_{x_1}L^\infty_{x_2}}\|\pa_2b_1\|_{L^\infty_{x_1}L^2_{x_2}}\|\pa_1\pa_2u_2\|_{L^2}\\
&+\|\pa_1\pa_2^2b_1\|_{L^2}\|\pa_2b_1\|_{L^\infty_{x_1}L^2_{x_2}}\|\pa_2u_1\|_{L^2_{x_1}L^\infty_{x_2}}\non\\
\lesssim&\|(u,b)\|_{H^2} \Big( \|\p_1 b\|_{H^2}^2 +\|\pa_2 u\|_{H^1}^2\Big),
\end{align*}
and
\begin{align*}
\mathcal{I}_{25}
\lesssim&\|\p_2^2 u\|_{L^2}(\|\pa_2^2b\|_{L^\infty_{x_1}L^2_{x_2}}\|\pa_1b\|_{L^2_{x_1}L^\infty_{x_2}}+\|\pa_1\pa_2b\|_{L^\infty_{x_1}L^2_{x_2}}\|\pa_2b\|_{L^2_{x_1}L^\infty_{x_2}})\\
\lesssim&\|\p_2^2 u\|_{L^2}\|b\|_{H^2}^\f12\|\pa_1b\|_{H^2}^\f12\|b\|_{H^2}^\f12\|\pa_1b\|_{H^2}^\f12+\|\p_2^2 u\|_{L^2}\|\pa_2b\|_{H^1}\|\pa_1\pa_2b\|_{H^1}\\
\lesssim&\|(u,b)\|_{H^2} \Big( \|\p_1 b\|_{H^2}^2 +\|\pa_2 u\|_{H^1}^2\Big),
\end{align*}
and by \eqref{u001} and \eqref{b001},
\begin{align*}
\mathcal{I}_{26}=&-\lan \na p-\pa_1b+u\cdot\na u-b\cdot\na b ,\p_2^2(b\cdot\na b)\ran\\
=&\lan  \pa_2 p,-\p_1\p_2(b\cdot\na b_1)\ran+\lan  \pa_2 p,-\p_2^2(b\cdot\na b_2)\ran\\
&+\lan  \pa_1 \pa_2^2b,b\cdot\na b\ran+\lan  \p_2(u\cdot\na u),\p_2(b\cdot\na b)\ran-\lan  \p_2(b\cdot\na b),\p_2(b\cdot\na b)\ran\\
=&-\lan  \pa_2 p,\p_1\p_2b_1\pa_1 b_1+\p_1b_1\pa_1\pa_2 b_1+\p_1\p_2b_2\pa_2 b_1+\p_1b_2\pa_2^2 b_1\ran\\
&-\lan  \pa_2 p,\p_2^2b_1\pa_1 b_2+\p_2b_1\pa_1\pa_2 b_2+\p_2^2b_2\pa_2 b_2+\p_2b_2\pa_2^2 b_2\ran\\
&-\lan  \pa_1 \pa_2b,\pa_2(b\cdot\na b)\ran+\lan  \p_2(u\cdot\na u),\p_2(b\cdot\na b)\ran-\lan  \p_2(b\cdot\na b),\p_2(b\cdot\na b)\ran\\
\lesssim&\|\na p\|_{L^2}\|\pa_1\pa_2b\|_{L^2_{x_1}L^\infty_{x_2}}\Big(\|\pa_1b\|_{L^\infty_{x_1}L^2_{x_2}}+\|\pa_2b\|_{L^\infty_{x_1}L^2_{x_2}}\Big)+\|\na p\|_{L^2}\|\pa_2^2b\|_{L^\infty_{x_1}L^2_{x_2}}\|\pa_1b\|_{L^2_{x_1}L^\infty_{x_2}}\\
&+\Big(\|\pa_1\pa_2b\|_{L^2}+\|\pa_2(u\cdot\na u)\|_{L^2}+\|\pa_2(b\cdot\na b)\|_{L^2}\Big)\|\pa_2(b\cdot\na b)\|_{L^2}\\
\lesssim&\Big(\|(u,b)\|_{H^2}+\|(u,b)\|_{H^2}^2\Big) \Big( \|\p_1 b\|_{H^2}^2 +\|\p_2u\|_{H^1}^2+ \|\na p\|_{L^2}^2\Big).
\end{align*}

Thus we proof the result
\begin{align}\label{p22b1}
&\f12\f d {dt}\|\pa_2^2 b\|_{L^2}^2+\|\p_1\pa_2^2  b\|_{L^2}^2+\|\pa_2  u_t\|_{L^2}^2+\lan \pa_1\pa_2^2b,\pa_2^2u\ran-\lan \pa_1\pa_2b,\pa_2u_t\ran-\lan \pa_2\cP\pa_2^2u,\pa_2u_t\ran\non\\
\lesssim&(\|(u,b)\|_{H^2}+\|(u,b)\|_{H^2}^2) \Big( \|\p_1 b\|_{H^2}^2 +\|\pa_1^2\pa_2 u\|_{L^2}^2+\|\pa_2 u\|_{H^1}^2+\|\pa_2 u_t\|_{L^2}^2+ \|\na p\|_{L^2}^2\Big).
\end{align}
{\bf Step 11.} Energy estimate and the dissipation estimate of $\pa_2^2u$.

 We rewrite equations  $\eqref{eq:MHDT}_{1,3,4}$  as
\ben\label{smhd1}
\left\{
\begin{array}{l}
-\Delta u+\na p=-\pa_1^2u+\p_1 b-\p_t u-u\cdot\na u+b\cdot \na b, \ \ \ \ x\in\Omega,\\
\div u=0,  \ \ \ \ x\in\Omega,\\
u=0, \ \ \ \ x\in\p\Omega.
\end{array}\right.
\een
By Stokes estimate, we have for the energy estimates,
\begin{align*}
&\|\pa_2^2u\|_{L^2}+\|\na p\|_{L^2}\non\\
\lesssim&\|\na^2u\|_{L^2}+\|\na p\|_{L^2}\non\\
\lesssim& \|\pa_1^2u\|_{L^2}+\|\pa_1b\|_{L^2}+\| u_t\|_{L^2}+\|u\|_{L^\infty}\|\na u\|_{L^2}+\|b\|_{L^\infty}\|\na b\|_{L^2}\non\\
\lesssim& \|\pa_1^2u\|_{L^2}+\|b\|_{H^1}+\| u_t\|_{L^2}+c_0\| u\|_{H^1}+c_0\| b\|_{H^1}\non\\
\lesssim & \|\pa_1^2u\|_{L^2}+\|(u, b)\|_{H^1}+\| u_t\|_{L^2},
\end{align*}
thus
\begin{align}\label{stokes1}
\|\pa_2^2u\|_{L^2}^2+\|\na p\|_{L^2}^2\lesssim  \|\pa_1^2u\|_{L^2}^2+\|(u, b)\|_{H^1}^2+\| u_t\|_{L^2}^2,
\end{align}
and for the dissipation estimates
\begin{align*}
&\|\pa_2^2u\|_{L^2}+\|\na p\|_{L^2}\non\\
\lesssim&\|\na^2u\|_{L^2}+\|\na p\|_{L^2}\non\\
\lesssim& \|\pa_1^2u\|_{L^2}+\|\pa_1b\|_{L^2}+\| u_t\|_{L^2}+\|u_1\|_{L^\infty_{x_1}L^2_{x_2}}\|\pa_1 u\|_{L^2_{x_1}L^\infty_{x_2}}\non\\
&+\|u_2\|_{L^2_{x_1}L^\infty_{x_2}}\|\pa_2 u\|_{L^\infty_{x_1}L^2_{x_2}}+\|b_1\|_{L^\infty_{x_1}L^2_{x_2}}\|\pa_1 b\|_{L^2_{x_1}L^\infty_{x_2}}\non\\
&+\|b_2\|_{L^2_{x_1}L^\infty_{x_2}}\|\pa_2 b\|_{L^\infty_{x_1}L^2_{x_2}}\non\\
\lesssim& \|\pa_1^2u\|_{L^2}+\|\pa_1b\|_{L^2}+\| u_t\|_{L^2}+c_0\| \pa_2u\|_{L^2}+c_0\| \pa_1\pa_2u\|_{L^2}+c_0\| \pa_1b\|_{H^1}\non\\
\lesssim & \|\pa_1^2u\|_{L^2}+\|\pa_1b\|_{H^1}+\|\pa_2u\|_{L^2}+\|\pa_1\pa_2u\|_{L^2}+\| u_t\|_{L^2},
\end{align*}
thus
\begin{align}\label{stokes2}
\int_0^t(\|\pa_2^2u\|_{L^2}^2+\|\na p\|_{L^2}^2)d\tau\lesssim  \int_0^t(\|\pa_1^2u\|_{L^2}^2+\|\pa_1b\|_{H^1}^2+\|\pa_1\pa_2u\|_{L^2}^2+\|\pa_2u\|_{L^2}^2+\| u_t\|_{L^2}^2)d\tau.
\end{align}
{\bf Step 12.} Closing of the {\it a priori} estimates.

Combining Step 1-11 together, for the  $0<\delta<\f14$, we calculate the following formula $\eqref{L2}+\eqref{H1}+\delta\big(\eqref{p1u}+\eqref{p12u,p12b}+\eqref{p1p2ub}+\eqref{H2u}\big)+\delta^2\big(\eqref{ut}+\eqref{p2p2u}+\eqref{p12u}\big)+
\delta^3\big(\eqref{stokes1}+\eqref{stokes2}\big)+\delta^4\eqref{p22b1}$, then we have
\begin{align*}
&\Big(\|(u,b)\|^2_{H^2}+\|\na p\|^2_{L^2}+\|(u_t,b_t)\|^2_{L^2}\Big)
+ \int_0^t  \Big( \|\na u\|_{H^1}^2+\|\p_1 b\|_{H^2}^2 +\| u_t\|_{L^2}^2+\|b_t\|_{L^2}^2+\|\pa_1^2\pa_2 u\|_{L^2}^2\non\\
&+\|\pa_2 u_t\|_{L^2}^2+\|\pa_1 b_t\|_{L^2}^2+ \|\na p\|_{L^2}^2+\|\cP\pa_2^2u\|_{H^1}^2\Big)ds\\
\lesssim& \|u_0\|^2_{H^2}+\|b_0\|^2_{H^2},
\end{align*}
for suitable $c_0$.
That is,
\begin{align*}
\cE^2(t)+\int_0^t\cF^2(s)ds
\lesssim \|u_0\|^2_{H^2}+\|b_0\|^2_{H^2} .
\end{align*}
 This completes the proof of Proposition \ref{high order}.
\end{proof}

\vspace{0.2cm}

\no{\bf Proof of  Theorem \ref{thm:main}.} We conclude Theorem \ref{thm:main} by a continuous argument.  \ef

\section{The resolvent estimate of linearized problem}
\subsection{Solution formula  of the linearized problem}
\setcounter{equation}{0}
 The linearized equation of \eqref{eq:MHDT} is
 \ben\label{eq:MHDL}
\left\{
\begin{array}{l}
\p_t u-\p_2^2 u-\p_1 b+\na p =f,  \ \ \ \ x\in\Omega,\ t>0,\\
\p_tb-\p_1^2 b-\p_1 u=g,  \ \ \ \ x\in\Omega,\ t>0,\\
\div u=\div b=0,  \ \ \ \ x\in\Omega,\ t>0,\\
u=0, \ \ \ b=0,   \ \ \ \ x\in\p\Omega,\ t>0,\\
u(x,0)=u_0(x),\ \ b(x,0)=b_0(x),\ \ \ \ x\in \Om,
\end{array}\right.
\een
where $f=-u\cdot\na u+b\cdot\na b$ and $g=-u \cdot\na b+b\cdot\na u$.

 Taking Laplace transform of $t$, we have
\begin{align}\label{eq:Lap trans}
\left\{
\begin{array}{l}
\lam u_{\lam}-\p_2^2 u_{\lam}-\p_1 b_{\lam}+\na p_{\lam}= u_0+f_{\lam}, \\
\lam b_{\lam}-\p_1^2 b_{\lam}-\p_1 u_{\lam}=b_0+g_{\lam}, \\
\div u_{\lam}=\div b_{\lam}=0,  \\
(u_{\lam}, b_{\lam})|_{x_2=0}=0,
\end{array}\right.
\end{align}
whereas we denote that
\beno
f_\lam=\cL\big(f(t)\big)=\int^{\oo}_0 e^{-\lam t}  f(t) dt.
\eeno

By taking horizontal Fourier transform $\wh{*}$ to \eqref{eq:Lap trans}, we obtain
 \ben\label{eq:MHDh}
\left\{
\begin{array}{l}
\lam \wh{u}_{1, \lam}-\pa_2^2  \wh{u}_{1, \lam}-i\xi_1 \wh{b}_{1, \lam}+i\xi_1 \wh{p}_{\lam} = \wh{u}_{1, 0}+\wh{f}_{1, \lam}, \\
\lam \wh{u}_{2, \lam}-\pa_2^2  \wh{u}_{2, \lam}-i\xi_1 \wh{b}_{2, \lam}+\p_2 \wh{p}_{\lam}= \wh{u}_{2, 0}+\wh{f}_{2, \lam}, \\
\lam \wh{b}_{\lam}+|\xi_1|^2 \wh{b}_{\lam}-i\xi_1 \wh{u}_\lam=\wh{b}_0+\wh{g}_{\lam},\\
i\xi_1 \wh{u}_{1, \lam}+\p_2 \wh{u}_{2, \lam}=i\xi_1\wh{b}_{1, \lam}+\p_2 \wh{b}_{2, \lam}=0,\\
\wh{u}_{\lam}|_{x_2=0}=\wh{b}_{\lam}|_{x_2=0}=0.
\end{array}\right.
\een

 By $\eqref{eq:MHDh}_{1, 2, 4}$, we have
\beno
(\p_2^2-|\xi_1|^2) \wh{p}_{\lam}=i\xi_1 \wh{f}_{1, \lam}+\p_2 \wh{f}_{2, \lam},
\eeno
 the solution is
\begin{align}\label{pressure}
\wh{p}_{\lam}=C(\xi_1) e^{-|\xi_1|x_2}-E_{|\xi_1|}[i\xi_1  \wh{f}_{1, \lam}+\p_2\wh{f}_{2, \lam}],
\end{align}
where
\begin{align}\label{Eom}
E_{|\xi_1|}[f]:=\f{1}{2|\xi_1|} \int^{\infty}_0 e^{-|\xi_1||x_2-y_2|} f(y_2) dy_2,
\end{align}
 then
\begin{align}\label{pressure2}
\p_2 E_{|\xi_1|}[i\xi_1  \wh{f}_{1, \lam}+\p_2\wh{f}_{2, \lam}]=&-\f12  e^{-|\xi_1|x_2}\int^{x_2}_0 e^{|\xi_1| y_2} (i\xi_1  \wh{f}_{1, \lam}+\p_2\wh{f}_{2, \lam}) dy_2\non\\
&+\f12  e^{|\xi_1|x_2}\int^\infty_{x_2} e^{-|\xi_1| y_2}( i\xi_1  \wh{f}_{1, \lam}+\p_2\wh{f}_{2, \lam} ) dy_2.\non\\
\end{align}

 By $\eqref{eq:MHDh}_{3}$, we have
 \begin{align}\label{ulamb}
\wh{b}_\lam=\f{i\xi_1\wh{u}_\lam+\wh{b}_0+\wh{g}_\lam}{\lam+|\xi_1|^2}.
\end{align}

 Taking \eqref{ulamb}  into $\eqref{eq:MHDh}_{1,2}$, we have
 \ben\label{ulam1b}
\left\{
\begin{array}{l}
(\lam+\f{|\xi_1|^2}{\lam+|\xi_1|^2}-\p_2^2) \wh{u}_{1, \lam}+i\xi_1 \wh{p}_{\lam} = (\wh{u}_{1, 0}+\wh{f}_{1, \lam})+\f{i\xi_1}{\lam+|\xi_1|^2}(\wh{b}_{1, 0}+\wh{g}_{1, \lam}), \\
(\lam+\f{|\xi_1|^2}{\lam+|\xi_1|^2}-\p_2^2) \wh{u}_{2, \lam}+\pa_2 \wh{p}_{\lam} = (\wh{u}_{2, 0}+\wh{f}_{2, \lam})+\f{i\xi_1}{\lam+|\xi_1|^2}(\wh{b}_{2, 0}+\wh{g}_{2, \lam}).
\end{array}\right.
\een

Taking \eqref{pressure}, \eqref{pressure2} into \eqref{ulam1b} and defining
\begin{align*}
\om(\lam, \xi_1)= \sqrt{\lam+\f{|\xi_1|^2}{\lam+|\xi_1|^2}},
\end{align*}
then $\wh{u}_\lam$ satisfies the following system
 \ben\label{eq:MHDh2}
\left\{
\begin{array}{l}
(\om^2-\p_2^2) \wh{u}_{1, \lam} = (\wh{u}_{1, 0}+\wh{f}_{1, \lam})+\f{i\xi_1}{\lam+|\xi_1|^2}(\wh{b}_{1, 0}+\wh{g}_{1, \lam})\\
\qquad \qquad \qquad \quad-i\xi_1C(\xi_1) e^{-|\xi_1|x_2}+i\xi_1E_{|\xi_1|}[i\xi_1  \wh{f}_{1, \lam}+\p_2\wh{f}_{2, \lam}], \\
(\om^2-\p_2^2) \wh{u}_{2, \lam} = (\wh{u}_{2, 0}+\wh{f}_{2, \lam})+\f{i\xi_1}{\lam+|\xi_1|^2}(\wh{b}_{2, 0}+\wh{g}_{2, \lam})\\
\qquad \qquad \qquad \quad+|\xi_1|C(\xi_1) e^{-|\xi_1|x_2}+\pa_2E_{|\xi_1|}[i\xi_1  \wh{f}_{1, \lam}+\p_2\wh{f}_{2, \lam}], \\
\wh{u}_{\lam}|_{x_2=0}=\p_2\wh{u}_{2, \lam}|_{x_2=0}=0,
\end{array}\right.
\een
and  the  unique solution is
$$
\left\{
\begin{array}{l}
\wh{u}_{1, \lam} = A_1(\xi_1)e^{-\om x_2}+E_\om[\wh{u}_{1, 0}+\wh{f}_{1, \lam}]+\f{i\xi_1}{\lam+|\xi_1|^2}E_\om[\wh{b}_{1, 0}+\wh{g}_{1, \lam}]\\
\qquad  \quad -i\xi_1C(\xi_1) E_\om[e^{-|\xi_1|x_2}]+i\xi_1E_\om[E_{|\xi_1|}[i\xi_1  \wh{f}_{1, \lam}+\p_2\wh{f}_{2, \lam}]], \\
\wh{u}_{2, \lam} = A_2(\xi_1)e^{-\om x_2}+E_\om[\wh{u}_{2, 0}+\wh{f}_{2, \lam}]+\f{i\xi_1}{\lam+|\xi_1|^2}E_\om[\wh{b}_{2, 0}+\wh{g}_{2, \lam}]\\
\qquad  \quad +|\xi_1|C(\xi_1) E_\om[e^{-|\xi_1|x_2}]+E_\om[\pa_2E_{|\xi_1|}[i\xi_1  \wh{f}_{1, \lam}+\p_2\wh{f}_{2, \lam}]], \\
\end{array}\right.
$$
 where $A_1(\xi_1), A_2(\xi_1)$ and $C(\xi_1)$ depend only on $\xi_1$, Re $\om\geq0$.
Using the boundary condition $(\ref{eq:MHDh2})_3$, and by \eqref{pressure2}, noticing that
\begin{align*}
&\p_2 E_{\om}[f]\Big|_{x_2=0}=\f{1}{2} \int^{\infty}_0 e^{-\om y_3} f( y_2) dy_2=\om E_{\om}[f]_0,
\end{align*}
(where we define $E_\om[f]_0:=E_\om[f]\Big|_{x_2=0}$), we have
$$\left\{
\begin{array}{l}
 A_1(\xi_1)=-\big(E_\om[\wh{u}_{1, 0}+\wh{f}_{1, \lam}]_0+\f{i\xi_1}{\lam+|\xi_1|^2}E_\om[\wh{b}_{1, 0}+\wh{g}_{1, \lam}]_0+i\xi_1E_\om[E_{|\xi_1|}[i\xi_1  \wh{f}_{1, \lam}+\p_2\wh{f}_{2, \lam}]]_0\big)\\
\qquad  \qquad -\f{i\xi_1}{|\xi_1|}\{ E_\om[\wh{u}_{2, 0}+\wh{f}_{2, \lam}]_0+\f{i\xi_1}{\lam+|\xi_1|^2}E_\om[\wh{b}_{2, 0}+\wh{g}_{2, \lam}]_0+E_\om[\pa_2E_{|\xi_1|}[i\xi_1  \wh{f}_{1, \lam}+\p_2\wh{f}_{2, \lam}]]_0\} , \\
A_2(\xi_1)=0,\\
C(\xi_1)=-\f{1}{|\xi_1|E_{\om}[e^{-|\xi_1|x_2}]_0}\{ E_\om[\wh{u}_{2, 0}+\wh{f}_{2, \lam}]_0+\f{i\xi_1}{\lam+|\xi_1|^2}E_\om[\wh{b}_{2, 0}+\wh{g}_{2, \lam}]_0+E_\om[\pa_2E_{|\xi_1|}[i\xi_1  \wh{f}_{1, \lam}+\p_2\wh{f}_{2, \lam}]]_0\}.
\end{array}\right.
$$

 Define
 \begin{align}\label{Nom}
 N_\om[f]=\f{1}{2\om}\int_0^\infty \big(e^{-\om|x_2-y_2|}-e^{-\om(x_2+y_2)}\big) f(y_2) dy_2,
 \end{align}
 noticing that
\begin{align*}
E_{\om}[e^{-|\xi_1|x_2}]&=\f{1}{2\om} \Big(\int^{x_2}_0 e^{-\om(x_2-y_2)} e^{-|\xi_1| y_2} dy_2+\int_{x_2}^\infty e^{\om(x_2-y_2)} e^{-|\xi_1| y_2} dy_2\Big)\\
&=\f{1}{2\om} \Big(e^{-\om x_2} \f{e^{(\om-|\xi_1|) x_2}-1}{\om-|\xi_1|} +e^{\om x_2} \f{e^{-(\om+|\xi_1|) x_2}}{\om+|\xi_1|} \Big)\\
&=\f{1}{2\om} \Big(\f{ e^{-|\xi_1|x_2}- e^{-\om x_2}}{\om-|\xi_1|} +\f{e^{-|\xi_1|x_2}}{\om+|\xi_1|}\Big),\\
\end{align*}
and
\begin{align*}
N_{\om}[e^{-|\xi_1|x_2}]&=E_{\om}[e^{-|\xi_1|x_2}]-\f{1}{2\om} \int^{\infty}_0 e^{-\om(x_2+y_2)} e^{-|\xi_1| y_2} dy_2\\
&=\f{1}{2\om} \Big(\f{ e^{-|\xi_1|x_2}- e^{-\om x_2}}{\om-|\xi_1|} +\f{e^{-|\xi_1|x_2}}{\om+|\xi_1|}-\f{e^{-\om x_2}}{\om+|\xi_1|}\Big)\\
&= \f{ e^{-|\xi_1|x_2}- e^{-\om x_2}}{\om^2-|\xi_1|^2},
\end{align*}
then
\begin{align}\label{0005}
\f{E_{\om}[e^{-|\xi_1|x_2}]}{E_{\om}[e^{-|\xi_1|x_2}]_0}&=\f{N_{\om}[e^{-|\xi_1|x_2}]}{E_{\om}[e^{-|\xi_1|x_2}]_0}+e^{-\om x_2},
\end{align}
and
\begin{align}\label{0006}
\f{N_{\om}[e^{-|\xi_1|x_2}]}{E_{\om}[e^{-|\xi_1|x_2}]_0}=\f{2\om }{\om-|\xi_1|}\Big(e^{-|\xi_1|x_2}- e^{-\om x_2}\Big).
\end{align}

  By \eqref{0005}, we obtain the solution of system \eqref{eq:MHDh2}
\begin{align}\label{ulam}
\left\{
\begin{array}{l}
\wh{u}_{1, \lam} =N_\om[\wh{u}_{1, 0}+\wh{f}_{1, \lam}]+\f{i\xi_1}{\lam+|\xi_1|^2}N_\om[\wh{b}_{1, 0}+\wh{g}_{1, \lam}]+i\xi_1N_\om[E_{|\xi_1|}[i\xi_1  \wh{f}_{1, \lam}+\p_2\wh{f}_{2, \lam}]]\\
\qquad  \quad +\f{i\xi_1N_\om[e^{-|\xi_1|x_2}]}{|\xi_1|E_{\om}[e^{-|\xi_1|x_2}]_0}\{ E_\om[\wh{u}_{2, 0}+\wh{f}_{2, \lam}]_0+\f{i\xi_1}{\lam+|\xi_1|^2}E_\om[\wh{b}_{2, 0}+\wh{g}_{2, \lam}]_0+E_\om[\pa_2E_{|\xi_1|}[i\xi_1  \wh{f}_{1, \lam}+\p_2\wh{f}_{2, \lam}]]_0\} , \\
\wh{u}_{2, \lam} =E_\om[\wh{u}_{2, 0}+\wh{f}_{2, \lam}]+\f{i\xi_1}{\lam+|\xi_1|^2}E_\om[\wh{b}_{2, 0}+\wh{g}_{2, \lam}]+E_\om[\pa_2E_{|\xi_1|}[i\xi_1  \wh{f}_{1, \lam}+\p_2\wh{f}_{2, \lam}]]\\
\qquad  \quad -\f{E_\om[e^{-|\xi_1|x_2}]}{E_{\om}[e^{-|\xi_1|x_2}]_0}\{ E_\om[\wh{u}_{2, 0}+\wh{f}_{2, \lam}]_0+\f{i\xi_1}{\lam+|\xi_1|^2}E_\om[\wh{b}_{2, 0}+\wh{g}_{2, \lam}]_0+E_\om[\pa_2E_{|\xi_1|}[i\xi_1  \wh{f}_{1, \lam}+\p_2\wh{f}_{2, \lam}]]_0\}\\
\qquad =N_\om[\wh{u}_{2, 0}+\wh{f}_{2, \lam}]+\f{i\xi_1}{\lam+|\xi_1|^2}N_\om[\wh{b}_{2, 0}+\wh{g}_{2, \lam}]+N_\om[\pa_2E_{|\xi_1|}[i\xi_1  \wh{f}_{1, \lam}+\p_2\wh{f}_{2, \lam}]]\\
\qquad  \quad -\f{N_\om[e^{-|\xi_1|x_2}]}{E_{\om}[e^{-|\xi_1|x_2}]_0}\{ E_\om[\wh{u}_{2, 0}+\wh{f}_{2, \lam}]_0+\f{i\xi_1}{\lam+|\xi_1|^2}E_\om[\wh{b}_{2, 0}+\wh{g}_{2, \lam}]_0+E_\om[\pa_2E_{|\xi_1|}[i\xi_1  \wh{f}_{1, \lam}+\p_2\wh{f}_{2, \lam}]]_0\}.
\end{array}\right.
\end{align}

By $\eqref{ulamb}$, we obtain
\begin{align}\label{blam}
\left\{
\begin{array}{l}
\wh{b}_{1, \lam} =\f{i\xi_1}{\lam+|\xi_1|^2}\big(N_\om[\wh{u}_{1, 0}+\wh{f}_{1, \lam}]+\f{i\xi_1}{\lam+|\xi_1|^2}N_\om[\wh{b}_{1, 0}+\wh{g}_{1, \lam}]+i\xi_1N_\om[E_{|\xi_1|}[i\xi_1  \wh{f}_{1, \lam}+\p_2\wh{f}_{2, \lam}]]\\
\qquad \quad  +\f{i\xi_1N_\om[e^{-|\xi_1|x_2}]}{|\xi_1|E_{\om}[e^{-|\xi_1|x_2}]_0}\{ E_\om[\wh{u}_{2, 0}+\wh{f}_{2, \lam}]_0+\f{i\xi_1}{\lam+|\xi_1|^2}E_\om[\wh{b}_{2, 0}+\wh{g}_{2, \lam}]_0+E_\om[\pa_2E_{|\xi_1|}[i\xi_1  \wh{f}_{1, \lam}+\p_2\wh{f}_{2, \lam}]]_0\} \big)\\
\qquad \quad+\f{\wh{b}_{1,0}+\wh{g}_{1,\lam}}{\lam+|\xi_1|^2} , \\
\wh{b}_{2, \lam} =\f{i\xi_1}{\lam+|\xi_1|^2}\big(N_\om[\wh{u}_{2, 0}+\wh{f}_{2, \lam}]+\f{i\xi_1}{\lam+|\xi_1|^2}N_\om[\wh{b}_{2, 0}+\wh{g}_{2, \lam}]+N_\om[\pa_2E_{|\xi_1|}[i\xi_1  \wh{f}_{1, \lam}+\p_2\wh{f}_{2, \lam}]]\\
\qquad \quad  -\f{N_\om[e^{-|\xi_1|x_2}]}{E_{\om}[e^{-|\xi_1|x_2}]_0}\{ E_\om[\wh{u}_{2, 0}+\wh{f}_{2, \lam}]_0+\f{i\xi_1}{\lam+|\xi_1|^2}E_\om[\wh{b}_{2, 0}+\wh{g}_{2, \lam}]_0+E_\om[\pa_2E_{|\xi_1|}[i\xi_1  \wh{f}_{1, \lam}+\p_2\wh{f}_{2, \lam}]]_0\}\big)\\
\qquad\quad+\f{\wh{b}_{2,0}+\wh{g}_{2,\lam}}{\lam+|\xi_1|^2}.
\end{array}\right.
\end{align}

Taking  inverse Laplace transform in $t$:
\beno
\cL^{-1}(F(\lam))=\f{1}{2\pi i} \int^{\beta+i\infty}_{\beta-i\infty} e^{\lam t} F(\lam) d\lam
\eeno
to \eqref{ulam} and  \eqref{blam} ($F(\lam)$ has no singularity  on the right hand side of $\beta=Re\,\lam$) , we  get the horizontal Fourier transform of the solution of \eqref{eq:MHDL}
\begin{align}\label{eq:u sol2}
\left\{
\begin{array}{l}
\wh{u}_1(\xi_1,x_2, t)=\f{1}{2\pi i}\int_{\Ga}e^{\lam t}
\Big\{ N_\om[\wh{u}_{1, 0}+\wh{f}_{1, \lam}]+\f{i\xi_1}{\lam+|\xi_1|^2}N_\om[\wh{b}_{1, 0}+\wh{g}_{1, \lam}]+i\xi_1N_\om[E_{|\xi_1|}[i\xi_1  \wh{f}_{1, \lam}+\p_2\wh{f}_{2, \lam}]]\\
\qquad \qquad \quad  \quad+\f{i\xi_1N_\om[e^{-|\xi_1|x_2}]}{|\xi_1|E_{\om}[e^{-|\xi_1|x_2}]_0}\{ E_\om[\wh{u}_{2, 0}+\wh{f}_{2, \lam}]_0+\f{i\xi_1}{\lam+|\xi_1|^2}E_\om[\wh{b}_{2, 0}+\wh{g}_{2, \lam}]_0\\
\qquad \qquad \quad \quad+E_\om[\pa_2E_{|\xi_1|}[i\xi_1  \wh{f}_{1, \lam}+\p_2\wh{f}_{2, \lam}]]_0\}\Big\}d\lam,\\
\wh{u}_2(\xi_1,x_2, t)=\f{1}{2\pi i}\int_{\Ga}e^{\lam t}
\Big\{N_\om[\wh{u}_{2, 0}+\wh{f}_{2, \lam}]+\f{i\xi_1}{\lam+|\xi_1|^2}N_\om[\wh{b}_{2, 0}+\wh{g}_{2, \lam}]+N_\om[\pa_2E_{|\xi_1|}[i\xi_1  \wh{f}_{1, \lam}+\p_2\wh{f}_{2, \lam}]]\\
\qquad \qquad \quad \quad  -\f{N_\om[e^{-|\xi_1|x_2}]}{E_{\om}[e^{-|\xi_1|x_2}]_0}\{ E_\om[\wh{u}_{2, 0}+\wh{f}_{2, \lam}]_0+\f{i\xi_1}{\lam+|\xi_1|^2}E_\om[\wh{b}_{2, 0}+\wh{g}_{2, \lam}]_0\\
\qquad \qquad \quad \quad+E_\om[\pa_2E_{|\xi_1|}[i\xi_1  \wh{f}_{1, \lam}+\p_2\wh{f}_{2, \lam}]]_0\}  \Big\}d\lam,\\
\wh{b}_1(\xi_1,x_2, t)=\f{1}{2\pi i}\int_{\Ga}e^{\lam t}
\Big\{\f{i\xi_1}{\lam+|\xi_1|^2}\big(N_\om[\wh{u}_{1, 0}+\wh{f}_{1, \lam}]+\f{i\xi_1}{\lam+|\xi_1|^2}N_\om[\wh{b}_{1, 0}+\wh{g}_{1, \lam}]+i\xi_1N_\om[E_{|\xi_1|}[i\xi_1  \wh{f}_{1, \lam}+\p_2\wh{f}_{2, \lam}]]\\
\qquad \qquad \quad \quad  +\f{i\xi_1N_\om[e^{-|\xi_1|x_2}]}{|\xi_1|E_{\om}[e^{-|\xi_1|x_2}]_0}\{ E_\om[\wh{u}_{2, 0}+\wh{f}_{2, \lam}]_0+\f{i\xi_1}{\lam+|\xi_1|^2}E_\om[\wh{b}_{2, 0}+\wh{g}_{2, \lam}]_0+E_\om[\pa_2E_{|\xi_1|}[i\xi_1  \wh{f}_{1, \lam}+\p_2\wh{f}_{2, \lam}]]_0\} \big)\\
\qquad \qquad \quad \quad+\f{\wh{b}_{1,0}+\wh{g}_{1,\lam}}{\lam+|\xi_1|^2}  \Big\}d\lam,\\
\wh{b}_2(\xi_1,x_2, t)=\f{1}{2\pi i}\int_{\Ga}e^{\lam t}
\Big\{ \f{i\xi_1}{\lam+|\xi_1|^2}\big(N_\om[\wh{u}_{2, 0}+\wh{f}_{2, \lam}]+\f{i\xi_1}{\lam+|\xi_1|^2}N_\om[\wh{b}_{2, 0}+\wh{g}_{2, \lam}]+N_\om[\pa_2E_{|\xi_1|}[i\xi_1  \wh{f}_{1, \lam}+\p_2\wh{f}_{2, \lam}]]\\
\qquad \qquad \quad  \quad -\f{N_\om[e^{-|\xi_1|x_2}]}{E_{\om}[e^{-|\xi_1|x_2}]_0}\{ E_\om[\wh{u}_{2, 0}+\wh{f}_{2, \lam}]_0+\f{i\xi_1}{\lam+|\xi_1|^2}E_\om[\wh{b}_{2, 0}+\wh{g}_{2, \lam}]_0+E_\om[\pa_2E_{|\xi_1|}[i\xi_1  \wh{f}_{1, \lam}+\p_2\wh{f}_{2, \lam}]]_0\}\big)\\
\qquad \qquad \quad \quad+\f{\wh{b}_{2,0}+\wh{g}_{2,\lam}}{\lam+|\xi_1|^2} \Big\}d\lam,
\end{array}\right.
\end{align}
here $\Ga=\{\lam=R+\eta e^{2\pi i/3}, \eta\geq 0\}\cup \{\lam=R+\eta e^{-2\pi i/3}, \eta\geq 0\}$,  $R>0$ is sufficiently large number taken in such a way that Re $  \om (\lam;\xi_1,\xi_2)>0$ for all $\lam\in \Ga$.

\subsection{The contour integration of the  linear part}
The linear part of \eqref{eq:u sol2}  is
\begin{align}\label{linear}
\left\{
\begin{array}{l}
\wh{u}_{1,L}(\xi_1,x_2, t)=\f{1}{2\pi i}\int_{\Ga}e^{\lam t}
\Big\{ N_\om[\wh{u}_{1, 0}]+\f{i\xi_1}{\lam+|\xi_1|^2}N_\om[\wh{b}_{1, 0}] +\f{i\xi_1N_\om[e^{-|\xi_1|x_2}]}{|\xi_1|E_{\om}[e^{-|\xi_1|x_2}]_0}\{ E_\om[\wh{u}_{2, 0}]_0+\f{i\xi_1}{\lam+|\xi_1|^2}E_\om[\wh{b}_{2, 0}]_0\}\Big\}d\lam,\\
\wh{u}_{2,L}(\xi_1,x_2, t)=\f{1}{2\pi i}\int_{\Ga}e^{\lam t}
\Big\{N_\om[\wh{u}_{2, 0}]+\f{i\xi_1}{\lam+|\xi_1|^2}N_\om[\wh{b}_{2, 0}]-\f{N_\om[e^{-|\xi_1|x_2}]}{E_{\om}[e^{-|\xi_1|x_2}]_0}\{ E_\om[\wh{u}_{2, 0}]_0+\f{i\xi_1}{\lam+|\xi_1|^2}E_\om[\wh{b}_{2, 0}]_0\}  \Big\}d\lam,\\
\wh{b}_{1,L}(\xi_1,x_2, t)=\f{1}{2\pi i}\int_{\Ga}e^{\lam t}
\Big\{\f{i\xi_1}{\lam+|\xi_1|^2}\big(N_\om[\wh{u}_{1, 0}]+\f{i\xi_1}{\lam+|\xi_1|^2}N_\om[\wh{b}_{1,0}]\\
\qquad \qquad \qquad \quad+\f{i\xi_1N_\om[e^{-|\xi_1|x_2}]}{|\xi_1|E_{\om}[e^{-|\xi_1|x_2}]_0}\{ E_\om[\wh{u}_{2, 0}]_0+\f{i\xi_1}{\lam+|\xi_1|^2}E_\om[\wh{b}_{2, 0}]_0\} \big) +\f{\wh{b}_{1,0}}{\lam+|\xi_1|^2}  \Big\}d\lam,\\
\wh{b}_{2,L}(\xi_1,x_2, t)=\f{1}{2\pi i}\int_{\Ga}e^{\lam t}
\Big\{ \f{i\xi_1}{\lam+|\xi_1|^2}\big(N_\om[\wh{u}_{2, 0}]+\f{i\xi_1}{\lam+|\xi_1|^2}N_\om[\wh{b}_{2, 0}]\\
\qquad \qquad \qquad \quad-\f{N_\om[e^{-|\xi_1|x_2}]}{E_{\om}[e^{-|\xi_1|x_2}]_0}\{ E_\om[\wh{u}_{2, 0}]_0+\f{i\xi_1}{\lam+|\xi_1|^2}E_\om[\wh{b}_{2, 0}]_0\}\big)+\f{\wh{b}_{2,0}}{\lam+|\xi_1|^2} \Big\}d\lam.\\
\end{array}\right.
\end{align}

By \eqref{0006}, $\wh{u}_{1,L}$ can be rewritten as
\begin{align}\label{u1}
\wh{u}_{1,L}(\xi_1,x_2, t)&=\f{1}{2\pi i}\int_{\Ga}e^{\lam t}
\Big\{ N_\om[\wh{u}_{1, 0}]+\f{i\xi_1}{\lam+|\xi_1|^2}N_\om[\wh{b}_{1, 0}]\non\\
 &\quad+\f{2i\xi_1 \om}{|\xi_1|(\om-|\xi_1|)}\{  E_\om[\wh{u}_{2, 0}]_0+\f{i\xi_1}{\lam+|\xi_1|^2}E_\om[\wh{b}_{2, 0}]_0\} \Big(e^{-|\xi_1|x_2}- e^{-\om x_2}\Big)\Big\}d\lam\non\\
 &=\f{1}{2\pi i}\int_{\Ga}e^{\lam t}
\Big\{ N_\om[\wh{u}_{1, 0}]+\f{i\xi_1}{\lam+|\xi_1|^2}N_\om[\wh{b}_{1, 0}]\Big\}d\lam\non\\
 &\quad+\f{1}{2\pi i}\int_{\Ga}e^{\lam t}
\Big\{\f{2i\xi_1 \om}{|\xi_1|(\om-|\xi_1|)}\{  E_\om[\wh{u}_{2, 0}]_0+\f{i\xi_1}{\lam+|\xi_1|^2}E_\om[\wh{b}_{2, 0}]_0\} \Big(e^{-|\xi_1|x_2}- e^{-\om x_2}\Big)\Big\}d\lam,\non\\
&:=I_1+I_2,
\end{align}
defining
\begin{align*}
\lam^{'}_\pm=\left\{
\begin{array}{l}
\f{-|\xi_1|^2\pm i\sqrt{4|\xi_1|^2-|\xi_1|^4}}{2}, \quad |\xi_1|\leq 2,\\
\f{-|\xi_1|^2\pm\sqrt{|\xi_1|^4-4|\xi_1|^2}}{2} ,\quad |\xi_1|> 2.
\end{array}\right.
\end{align*}

Notice the following results:
\begin{align*}
\lim\limits_{\om\to|\xi_1|}\f{\om(e^{-|\xi_1|x_2}- e^{-\om x_2})}{\om-|\xi_1|}=|\xi_1|x_2e^{-|\xi_1|x_2}< 1,
\end{align*}
we only need to consider the branch point $\lam=\lam_+^{'}$, $\lam=\lam_-^{'}$ and $\lam=-|\xi_1|^2$.\\

Similarly,
\begin{align}\label{u2}
\wh{u}_{2,L}(\xi_1,x_2, t)&=\f{1}{2\pi i}\int_{\Ga}e^{\lam t}
\Big\{ N_\om[\wh{u}_{2, 0}]+\f{i\xi_1}{\lam+|\xi_1|^2}N_\om[\wh{b}_{2, 0}]\non\\
 &\quad-\f{2 \om}{\om-|\xi_1|}\{  E_\om[\wh{u}_{2, 0}]_0+\f{i\xi_1}{\lam+|\xi_1|^2}E_\om[\wh{b}_{2, 0}]_0\} \Big(e^{-|\xi_1|x_2}- e^{-\om x_2}\Big)\Big\}d\lam,
\end{align}
which can be estimated as $\wh{u}_{1, L}$.

When $|\xi_1|\leq 2$, we use the branch specified by the requirement
\beno
\arg(\lam-\lam_\pm')=\mp\f{\pi}{2} \ \text{at} \ \lam=\text{Re}\  \lam_\pm' \ \text{and} \ \arg \lam=0 \ \text{at}\  \lam=0,
\eeno
and take the branch cut
\beno
\{ \lam; \ \text{Re} \ \lam \leq 0,\  \text{Im} \ \lam =0\}\cup \{\lam \in \Pi; \ \text{Re} \ \lam \leq \text{Re}\ \lam'_\pm\},
\eeno
where $\Pi$ is the circle defined by
\beno
\Pi=\{ \lam=\eta+i\sigma; \ \eta^2+\sigma^2= |\xi_1|^2\}.
\eeno

When $|\xi_1|> 2$, we use the branch specified by
\beno
\arg(\lam-\lam'_\pm)=\arg  \lam=0 \ \text{at} \ \lam=0,
\eeno
and take the branch cut
\beno
\{ \lam;\  \text{Re} \ \lam \leq \lam'_-,\  \text{Im} \ \lam =0\} \cup \{  \lam'_+ \leq \text{Re} \ \lam \leq 0,\  \text{Im} \ \lam =0\}.
\eeno

Now we consider the contour in the following two cases:

Case 1: $ |\xi_1|\leq 2$.

We deform the contour $\Ga$ into  $\cup_{m=1}^5 \Ga_m$, where $\Ga_1=\Ga_1^{(+)}\cup\Ga_1^{(-)}$  wraps around the portion $\{ \lam=-\eta-|\xi_1|^2; \ \eta:0\to |\xi_1|\}$ in the branch cut with
\begin{align*}
{\Ga^{(+)}_1}&=\Big\{ \lam=-\eta-|\xi_1|^2; \ \eta:0\to |\xi_1|\Big\},\\
{\Ga^{(-)}_1}&=\Big\{ \lam=-\eta-|\xi_1|^2; \ \eta:|\xi_1| \to 0\Big\}.
\end{align*}

$\Ga_2=\Ga_2^{(+)}\cup\Ga_2^{(-)}$  wraps around the portion $\{ \lam\in \Pi; -|\xi_1|^2-|\xi_1|\leq \text{Re} \lam \leq \text{Re} \lam'_+,\  \text{Im} \lam > 0\}$ in the branch cut with
\begin{align*}
{\Ga^{(+)}_2}=\Big\{ \lam=\lam'_+-\eta+i\big(-\text{Im} \lam'_++D(\eta, \xi_1)\big); \ \eta:0\to d_0\Big\},\\
{\Ga^{(-)}_2}=\Big\{ \lam=\lam'_+-\eta+i\big(-\text{Im} \lam'_++D(\eta, \xi_1)\big); \ \eta:d_0\to 0\Big\},
\end{align*}
where we define
 $$D(\eta, \xi_1):=\sqrt{(\text{Im} \lam'_+)^2+2(\text{Re} \lam'_++|\xi_1|^2) \eta-\eta^2}, \quad d_0= \text{Re} \lam'_++|\xi_1|+|\xi_1|^2.$$

$\Ga_3=\Ga_3^{(+)}\cup\Ga_3^{(-)}$ wraps around the portion $\{ \lam\in \Pi; -|\xi_1|^2-|\xi_1| \leq \text{Re} \lam \leq \text{Re} \lam'_+, \ \text{Im} \lam < 0\}$ in the branch cut with
\begin{align*}
{\Ga^{(+)}_3}&=\Big\{ \lam=\lam'_--\eta-i\big(-\text{Im} \lam'_++D(\eta, \xi_1)\big); \ \eta: 0\to d_0\Big\},\\
{\Ga^{(-)}_3}&=\Big\{ \lam=\lam'_--\eta-i\big(-\text{Im} \lam'_++D(\eta, \xi_1)\big); \ \eta:d_0\to 0\Big\}.
\end{align*}

$\Ga_4=\Ga_4^{(+)}\cup\Ga_4^{(-)}$  defined as
 \begin{align*}
{\Ga^{(+)}_4}&=\Big\{ \lam=-\eta-|\xi_1|^2 ; \ \eta:|\xi_1|\to \infty\Big\},\\
{\Ga^{(-)}_4}&=\Big\{ \lam=-\eta-|\xi_1|^2 ; \ \eta: \infty\to |\xi_1|\Big\},
\end{align*}
and
\beno
{\Ga_5}=\Big\{\lam=-|\xi_1|^2+\e e^{i\gamma}; \ \gamma: -\pi \to \pi\Big\}
\eeno
with $ \e\to 0$ (But not equal 0).\\

Case 2: $|\xi_1|> 2$.

We deform the contour $\Ga$   into $\widetilde{\Ga}^{(+)}_1\cup\widetilde{\Ga}^{(-)}_1$, where $\widetilde{\Ga}_1=\widetilde{\Ga}_1^{(+)}\cup \widetilde{\Ga}_1^{(-)}$ wraps around the portion $ \{ \lam=-\eta-|\xi_1|^2; \ \eta:\lam'_-\to \lam'_+\}$ in the branch cut with
\begin{align*}
{\widetilde{\Ga}_1^{(+)}}=\Big\{ \lam=-\eta-|\xi_1|^2; \ \eta:\lam'_-\to \lam'_+\Big\},\\
{\widetilde{\Ga}_1^{(-)}}=\Big\{ \lam=-\eta-|\xi_1|^2; \ \eta:\lam'_+ \to \lam'_-\Big\},
 \end{align*}
and $\widetilde{\Ga}_2=\widetilde{\Ga}_2^{(+)}\cup \widetilde{\Ga}_2^{(-)}$ wraps around the portion $\{ \lam=-\eta-|\xi_1|^2; \ \eta:0\to \infty \}$ in the branch cut with
\begin{align*}
{\widetilde{\Ga}_2^{(+)}}=\Big\{ \lam=-\eta-|\xi_1|^2; \ \eta:0\to \infty\Big\},\\
{\widetilde{\Ga}_2^{(-)}}=\Big\{ \lam=-\eta-|\xi_1|^2; \ \eta:\infty \to 0\Big\},
 \end{align*}
and
\beno
{\widetilde{\Ga}_3}=\Big\{\lam=-|\xi_1|^2+\e e^{i\gamma}; \ \gamma: -\pi \to \pi\Big\}
\eeno
with $ \e\to 0$ (But not equal 0).\\

 \subsection{The resolvent estimate of the linear part}

 We first give the  resolvent estimate of velocity field in  \eqref{linear}.
\begin{proposition}\label{linear u}
We have the following $L^2$ estimates
 \begin{align}\label{uL2}
 \|u_L\|_{L^2}&\lesssim \lan t\ran^{-\f12}\|(u_0, b_0)\|_{L^1\cap L^2 }
 \end{align}
 for the velocity  field in \eqref{linear}.
 \end{proposition}

 \begin{proof}  We only need to consider $I_1$ and $I_2$ in  \eqref{u1}.

Consider the  odd extension of $\wh{u}_{1, 0}+\f{i\xi_1}{\lam} \wh{b}_{1,0}$,
\begin{align}\label{u10'}
\wh{u}_{1, 0}^{o}(y_2)=\left\{
\begin{array}{l}
\wh{u}_{1, 0}(y_2), \ \  \text{for}  \quad y_2>0, \\
-\wh{u}_{1, 0}(-y_2), \  \  \text{for}  \quad y_2\leq0,
\end{array}\right.
\end{align}
\begin{align}\label{b10'}
\wh{b}_{1, 0}^{o}(y_2)=\left\{
\begin{array}{l}
\wh{b}_{1, 0}(y_2), \ \  \text{for}  \quad y_2>0, \\
-\wh{b}_{1, 0}(-y_2), \  \  \text{for}  \quad y_2\leq0,
\end{array}\right.
\end{align}
\begin{align}\label{vphi,chi}
\wh{U}_{1, 0}(y_2)=
(\wh{u}^{o}_{1, 0}+\f{i\xi_1}{\lam+|\xi_1|^2} \wh{b}^{o}_{1,0})(y_2),
\end{align}
then  $I_1$ can be rewritten as
\begin{align*}
I_1&=\f{1}{2\pi i}\int_{\Ga}e^{\lam t} \f{1}{2\om} \int^\infty_{-\infty} e^{-\om|x_2-y_2|} \wh{U}_{1, 0}(y_2) dy_2 d\lam\non\\
&=\f{1}{2\pi i}\int_{\Ga}e^{\lam t} \f{1}{2\om} e^{-\om|x_2|} * \wh{U}_{1, 0}  d\lam:=\f{1}{2\pi i}\int_{\Ga}e^{\lam t} I_{1, \lam}  d\lam.
\end{align*}

Consider the odd extension  $\wt{I}_{1, \lam}$ of $I_{1, \lam}$, then
\begin{align*}
\cF_{x_2}(\wt{I}_1)&=\f{1}{2\pi i}\int_{\Ga}e^{\lam t}\cF_{x_2} (\wt{I}_{1, \lam})d\lam\non\\
&=\f{1}{2\pi i}\int_{\Ga}e^{\lam t}\cF_{x_2}(\f{1}{2\om} e^{-\om|x_2|} ) \cF_{x_2}(\wh{U}_{1, 0}  ) d\lam\non\\
&=\f{1}{2\pi i}\int_{\Ga}e^{\lam t}\f{1}{\om^2+\xi_2^2} \cF_{x_2}(\wh{U}_{1, 0}  ) d\lam\non\\
&=\f{1}{2\pi i}\int_{\Ga}\f{(\lam+|\xi_1|^2) e^{\lam t}}{(\lam-\lam_+)(\lam-\lam_-)}\cF_{x_2}(\wh{U}_{1, 0}  ) d\lam\non\\
&=\f{1}{2\pi i}\int_{\Ga}\f{(\lam+|\xi_1|^2) e^{\lam t}}{(\lam-\lam_+)(\lam-\lam_-)}\cF_{x_2}(\wh{u}^{o}_{1, 0})+\f{i\xi_1 e^{\lam t}}{(\lam-\lam_+)(\lam-\lam_-)}\cF_{x_2}(\wh{b}^{o}_{1, 0}) d\lam\non\\
&=\Big(\f{\lam_+ e^{\lam_+ t}-\lam_- e^{\lam_- t}}{\lam_+-\lam_-}+|\xi_1|^2\f{ e^{\lam_+ t}-e^{\lam_- t}}{\lam_+-\lam_-}\Big)\cF_{x_2}(\wh{u}^{o}_{1, 0})+\Big(i\xi_1\f{ e^{\lam_+ t}-e^{\lam_- t}}{\lam_+-\lam_-}\Big)\cF_{x_2}(\wh{b}^{o}_{1, 0}),
\end{align*}
the definition of $\lam_\pm$ can be see \eqref{lampm}.

Define
\begin{align*}
M_1&:=\big|\f{\lam_+ e^{\lam_+ t}-\lam_- e^{\lam_- t}}{\lam_+-\lam_-}\big|, \\
M_2&:=|\xi_1|^2\big|\f{ e^{\lam_+ t}-e^{\lam_- t}}{\lam_+-\lam_-}\big|,\\
M_3&:=|\xi_1|\big|\f{ e^{\lam_+ t}-e^{\lam_- t}}{\lam_+-\lam_-}\big|,
\end{align*}
 we consider the following four cases:\\

\textbf{Case a}: $\big||\xi_1|^2-|\xi_2|^2\big|\leq |\xi_1|$.\\

We have
\begin{align}
&|\lam_+-\lam_-|= \sqrt{4|\xi_1|^2-(|\xi_1|^2-|\xi_2|^2)^2}\geq \sqrt{3}|\xi_1|,\non\\
&\Big|\f{ e^{\lam_+ t}-e^{\lam_- t}}{\lam_+-\lam_-}\Big|= e^{-\f{|\xi_1|^2+|\xi_2|^2}{2}t}\Big|\f{sin(\f{ \sqrt{4|\xi_1|^2-(|\xi_1|^2-|\xi_2|^2)^2}}{2}t)}{\f{ \sqrt{4|\xi_1|^2-(|\xi_1|^2-|\xi_2|^2)^2}}{2}}\Big|\lesssim \f{1}{|\xi_1|}e^{-\f{|\xi_1|^2+|\xi_2|^2}{2}t}.\label{206}
\end{align}

By $|\f{\sin \theta}{\theta}|\leq 1$, we have
\begin{align}
&M_1=\Big|\f{\lam_+ e^{\lam_+ t}-\lam_- e^{\lam_- t}}{\lam_+-\lam_-}\Big|\non\\
&\qquad =e^{-\f{|\xi_1|^2+|\xi_2|^2}{2}t}\Big|\f{|\xi_1|^2+|\xi_2|^2}{2}t\f{sin(\f{ \sqrt{4|\xi_1|^2-(|\xi_1|^2-|\xi_2|^2)^2}}{2}t)}{\f{ \sqrt{4|\xi_1|^2-(|\xi_1|^2-|\xi_2|^2)^2}}{2}t}-cos(\f{ \sqrt{4|\xi_1|^2-(|\xi_1|^2-|\xi_2|^2)^2}}{2}t)\Big|\non\\
&\qquad\lesssim e^{-\f{|\xi_1|^2+|\xi_2|^2}{2}t},\label{203}\\
&M_2=|\xi_1|^2\Big|\f{ e^{\lam_+ t}-e^{\lam_- t}}{\lam_+-\lam_-}\Big|= |\xi_1|^2te^{-\f{|\xi_1|^2+|\xi_2|^2}{2}t}\Big|\f{sin(\f{ \sqrt{4|\xi_1|^2-(|\xi_1|^2-|\xi_2|^2)^2}}{2}t)}{\f{ \sqrt{4|\xi_1|^2-(|\xi_1|^2-|\xi_2|^2)^2}}{2}t}\Big|\lesssim e^{-\f{|\xi_1|^2+|\xi_2|^2}{2}t},\non
\end{align}
by \eqref{206},
\begin{align*}
&M_3=|\xi_1|\Big|\f{ e^{\lam_+ t}-e^{\lam_- t}}{\lam_+-\lam_-}\Big|\lesssim e^{-\f{|\xi_1|^2+|\xi_2|^2}{2}t}.
\end{align*}

\textbf{Case b}: $|\xi_1|<\big||\xi_1|^2-|\xi_2|^2\big|\leq 2|\xi_1|$.\\

By the $\big||\xi_1|^2-|\xi_2|^2\big|=\pm(|\xi_1|^2-|\xi_2|^2)$, we consider the following two cases:\\

\textbf{Case b.1}: $|\xi_1|<|\xi_1|^2-|\xi_2|^2\leq 2|\xi_1|$.\\

For $ |\xi_1|<|\xi_1|^2-|\xi_2|^2\leq 2|\xi_1|$, it means $|\xi_1|> 1$, we have
\begin{align}
&\Big|\f{ e^{\lam_+ t}-e^{\lam_- t}}{\lam_+-\lam_-}\Big|= e^{-\f{|\xi_1|^2+|\xi_2|^2}{2}t}t\Big|\f{sin(\f{ \sqrt{4|\xi_1|^2-(|\xi_1|^2-|\xi_2|^2)^2}}{2}t)}{\f{ \sqrt{4|\xi_1|^2-(|\xi_1|^2-|\xi_2|^2)^2}}{2}t}\Big|\lesssim te^{-\f{|\xi_1|^2+|\xi_2|^2}{2}t}, \label{201}
\end{align}
by
\begin{align}\label{301}
te^{-\f{|\xi_1|^2+|\xi_2|^2}{2}t}\lesssim \f{1}{|\xi_1|^2+|\xi_2|^2}e^{-\f{|\xi_1|^2+|\xi_2|^2}{2}t},
\end{align}
we have
\begin{align}
\Big|\f{ e^{\lam_+ t}-e^{\lam_- t}}{\lam_+-\lam_-}\Big|\lesssim \f{1}{|\xi_1|^2+|\xi_2|^2}e^{-\f{|\xi_1|^2+|\xi_2|^2}{2}t}\lesssim \f{1}{|\xi_1|^2}e^{-\f{|\xi_1|^2+|\xi_2|^2}{2}t}\lesssim \f{1}{|\xi_1|}e^{-\f{|\xi_1|^2+|\xi_2|^2}{2}t},\label{202}
\end{align}
by \eqref{203},
\begin{align}\label{223}
&M_1=\Big|\f{\lam_+ e^{\lam_+ t}-\lam_- e^{\lam_- t}}{\lam_+-\lam_-}\Big|\lesssim e^{-\f{|\xi_1|^2+|\xi_2|^2}{2}t},
\end{align}
by \eqref{201},
\begin{align}\label{M2}
&M_2=|\xi_1|^2\Big|\f{ e^{\lam_+ t}-e^{\lam_- t}}{\lam_+-\lam_-}\Big|\lesssim |\xi_1|^2te^{-\f{|\xi_1|^2+|\xi_2|^2}{2}t}\lesssim e^{-\f{|\xi_1|^2+|\xi_2|^2}{2}t},
\end{align}
by \eqref{202},
\begin{align*}
&M_3=|\xi_1|\Big|\f{ e^{\lam_+ t}-e^{\lam_- t}}{\lam_+-\lam_-}\Big|\lesssim e^{-\f{|\xi_1|^2+|\xi_2|^2}{2}t}.
\end{align*}

\textbf{Case b.2}: $|\xi_1|<|\xi_2|^2-|\xi_1|^2\leq 2|\xi_1|$.\\

We have
\begin{align}
&\Big|\f{ e^{\lam_+ t}-e^{\lam_- t}}{\lam_+-\lam_-}\Big|= e^{-\f{|\xi_1|^2+|\xi_2|^2}{2}t}t\Big|\f{sin(\f{ \sqrt{4|\xi_1|^2-(|\xi_1|^2-|\xi_2|^2)^2}}{2}t)}{\f{ \sqrt{4|\xi_1|^2-(|\xi_1|^2-|\xi_2|^2)^2}}{2}t}\Big|\lesssim te^{-\f{|\xi_1|^2+|\xi_2|^2}{2}t},\label{204}
\end{align}
by \eqref{301},
\begin{align}
\Big|\f{ e^{\lam_+ t}-e^{\lam_- t}}{\lam_+-\lam_-}\Big|\lesssim \f{1}{|\xi_1|^2+|\xi_2|^2}e^{-\f{|\xi_1|^2+|\xi_2|^2}{2}t}\lesssim \f{1}{2|\xi_1|^2+|\xi_1|}e^{-\f{|\xi_1|^2+|\xi_2|^2}{2}t}\lesssim \f{1}{|\xi_1|}e^{-\f{|\xi_1|^2+|\xi_2|^2}{2}t},\label{205}
\end{align}
by \eqref{203},
\begin{align}\label{224}
&M_1=\Big|\f{\lam_+ e^{\lam_+ t}-\lam_- e^{\lam_- t}}{\lam_+-\lam_-}\Big| \lesssim e^{-\f{|\xi_1|^2+|\xi_2|^2}{2}t},
\end{align}
by \eqref{204},
\begin{align*}
&M_2=|\xi_1|^2\Big|\f{ e^{\lam_+ t}-e^{\lam_- t}}{\lam_+-\lam_-}\Big|\lesssim |\xi_1|^2te^{-\f{|\xi_1|^2+|\xi_2|^2}{2}t}\lesssim e^{-\f{|\xi_1|^2+|\xi_2|^2}{2}t},
\end{align*}
by \eqref{205},
\begin{align*}
&M_3=|\xi_1|\Big|\f{ e^{\lam_+ t}-e^{\lam_- t}}{\lam_+-\lam_-}\Big|\lesssim e^{-\f{|\xi_1|^2+|\xi_2|^2}{2}t}.
\end{align*}

\textbf{Case c}: $2|\xi_1|<\big||\xi_1|^2-|\xi_2|^2\big|\leq 4|\xi_1|$.\\

For $ \big||\xi_1|^2-|\xi_2|^2\big|> 2|\xi_1|$, we have $\lam_-< \lam_+<0$, we only need to consider the trouble decay item $\lam_+$.
We have
\begin{align*}
&\lam_+ =-\f{|\xi_1|^2+|\xi_2|^2}{2}+\f{\sqrt{(|\xi_1|^2-|\xi_2|^2)^2-4|\xi_1|^2}}{2}\\
 & \qquad \leq-\f{|\xi_1|^2+|\xi_2|^2}{2}+\f{\sqrt{15}}{4}\f{\big||\xi_1|^2-|\xi_2|^2\big|}{2}\leq-c_0(|\xi_1|^2+|\xi_2|^2),\\
 &\lam_+ \geq-\f{|\xi_1|^2+|\xi_2|^2}{2},
 \end{align*}
 for $c_0\in(\f12-\f{\sqrt{15}}{8},\f12)$, thus
 \begin{align}\label{209}
 |\lam_+|\lesssim |\xi_1|^2+|\xi_2|^2.
 \end{align}

 By $e^{-x}>1-x$ for $x>0$,
 \begin{align}
\Big|\f{ e^{\lam_+ t}-e^{\lam_- t}}{\lam_+-\lam_-}\Big|&=e^{\lam_+ t}\f{1-e^{-(\lam_+-\lam_-)t}}{\lam_+-\lam_-}\lesssim t e^{\lam_+ t} \lesssim \f{1}{|\xi_1|^2+|\xi_2|^2} e^{-c_0(|\xi_1|^2+|\xi_2|^2) t},\label{207}
\end{align}
by $|\xi_1|^2+|\xi_2|^2>|\xi_1|$,
\begin{align}
\Big|\f{ e^{\lam_+ t}-e^{\lam_- t}}{\lam_+-\lam_-}\Big|\lesssim \f{1}{|\xi_1|} e^{-c_0(|\xi_1|^2+|\xi_2|^2) t},\label{208}
\end{align}
by \eqref{209} and \eqref{207},
\begin{align*}
&M_1=\Big|\f{ \lam_+e^{\lam_+ t}-\lam_-e^{\lam_- t}}{\lam_+-\lam_-}\Big|=\Big|e^{\lam_- t}+\lam_+\f{ e^{\lam_+ t}-e^{\lam_- t}}{\lam_+-\lam_-}\Big|\lesssim   e^{-c_0(|\xi_1|^2+|\xi_2|^2) t},
\end{align*}
 by \eqref{207},
\begin{align*}
&M_2=|\xi_1|^2\Big|\f{ e^{\lam_+ t}-e^{\lam_- t}}{\lam_+-\lam_-}\Big| \lesssim \f{|\xi_1|^2}{|\xi_1|^2+|\xi_2|^2} e^{-c_0(|\xi_1|^2+|\xi_2|^2) t}\lesssim e^{-c_0(|\xi_1|^2+|\xi_2|^2) t},
\end{align*}
by \eqref{208},
\begin{align*}
&M_3=|\xi_1|\Big|\f{ e^{\lam_+ t}-e^{\lam_- t}}{\lam_+-\lam_-}\Big|\lesssim   e^{-c_0(|\xi_1|^2+|\xi_2|^2) t}.
\end{align*}

\textbf{Case d}: $\big||\xi_1|^2-|\xi_2|^2\big|> 4|\xi_1|$.\\

We only need to consider the trouble decay item $\lam_+$. By the $\big||\xi_1|^2-|\xi_2|^2\big|=\pm(|\xi_1|^2-|\xi_2|^2)$, we consider the following two cases:\\

\textbf{Case d.1}: $|\xi_1|^2-|\xi_2|^2> 4|\xi_1|$.\\

We have $|\xi_1|>4$, and
\begin{align*}
\lam_+ &=-\f{|\xi_1|^2+|\xi_2|^2}{2}+\f{\sqrt{(|\xi_1|^2-|\xi_2|^2)^2-4|\xi_1|^2}}{2}=\f{-|\xi_1|^2-|\xi_1|^2|\xi_2|^2}{\f{|\xi_1|^2+|\xi_2|^2}{2}+\f{\sqrt{(|\xi_1|^2-|\xi_2|^2)^2-4|\xi_1|^2}}{2}}\\
&\leq-\f{|\xi_1|^2+|\xi_1|^2|\xi_2|^2}{|\xi_1|^2}\leq-1,
\end{align*}
it has an exponential decay.\\

\textbf{Case d.2}: $|\xi_2|^2-|\xi_1|^2> 4|\xi_1|$.\\

We have
\begin{align}
&\lam_+ =-\f{|\xi_1|^2+|\xi_2|^2}{2}+\f{\sqrt{(|\xi_1|^2-|\xi_2|^2)^2-4|\xi_1|^2}}{2}=\f{-|\xi_1|^2-|\xi_1|^2|\xi_2|^2}{\f{|\xi_1|^2+|\xi_2|^2}{2}+\f{\sqrt{(|\xi_1|^2-|\xi_2|^2)^2-4|\xi_1|^2}}{2}}\non\\
&\qquad \leq\f{-|\xi_1|^2-|\xi_1|^2|\xi_2|^2}{\f{|\xi_1|^2+|\xi_2|^2}{2}+\f{|\xi_2|^2-|\xi_1|^2}{2}}=-\f{|\xi_1|^2+|\xi_1|^2|\xi_2|^2}{|\xi_2|^2},\label{210}\\
&\lam_+ =\f{-|\xi_1|^2-|\xi_1|^2|\xi_2|^2}{\f{|\xi_1|^2+|\xi_2|^2}{2}+\f{\sqrt{(|\xi_1|^2-|\xi_2|^2)^2-4|\xi_1|^2}}{2}} \geq-\f{|\xi_1|^2+|\xi_1|^2|\xi_2|^2}{\f12(|\xi_2|^2-|\xi_1|^2)}=-\f{2|\xi_1|^2+2|\xi_1|^2|\xi_2|^2}{|\xi_2|^2-|\xi_1|^2},\label{213}
\end{align}
by \eqref{213},
\begin{align}
|\lam_+| \lesssim\f{|\xi_1|^2+|\xi_1|^2|\xi_2|^2}{|\xi_2|^2-|\xi_1|^2}.\label{211}
\end{align}

And we have
\begin{align}
&\f{1}{\lam_+-\lam_-}=\f{1}{\sqrt{(|\xi_1|^2-|\xi_2|^2)^2-4|\xi_1|^2}}\lesssim\f{1}{|\xi_2|^2-|\xi_1|^2}.\label{212}
\end{align}
By \eqref{210}, \eqref{211} and \eqref{212},
\begin{align*}
M_1=&\Big|\f{ \lam_+e^{\lam_+ t}-\lam_-e^{\lam_- t}}{\lam_+-\lam_-}\Big|\\
\lesssim&   \f{|\xi_1|^2+|\xi_1|^2|\xi_2|^2}{(|\xi_2|^2-|\xi_1|^2)^2}e^{-\f{|\xi_1|^2}{|\xi_2|^2}t-|\xi_1|^2t}\\
=&\f{|\xi_1|^2+|\xi_1|^2(|\xi_2|^2-|\xi_1|^2)+|\xi_1|^4}{(|\xi_2|^2-|\xi_1|^2)^2}e^{-\f{|\xi_1|^2}{|\xi_2|^2}t-|\xi_1|^2t}\\
=&\f{|\xi_1|^2+|\xi_1|^4}{(|\xi_2|^2-|\xi_1|^2)^2}e^{-\f{|\xi_1|^2}{|\xi_2|^2}t-|\xi_1|^2t}+\f{|\xi_1|^2}{|\xi_2|^2-|\xi_1|^2}e^{-\f{|\xi_1|^2}{|\xi_2|^2}t-|\xi_1|^2t}\\
\lesssim&\f{|\xi_1|+|\xi_1|^2+|\xi_1|^3}{|\xi_2|^2-|\xi_1|^2}e^{-\f{|\xi_1|^2}{|\xi_2|^2}t-|\xi_1|^2t},
\end{align*}
and by \eqref{210} and \eqref{212},
\begin{align*}
&M_2=|\xi_1|^2\Big|\f{ e^{\lam_+ t}-e^{\lam_- t}}{\lam_+-\lam_-}\Big| \lesssim\f{|\xi_1|^2}{|\xi_2|^2-|\xi_1|^2}e^{-\f{|\xi_1|^2}{|\xi_2|^2}t-|\xi_1|^2t},\\
&M_3=|\xi_1|\Big|\f{ e^{\lam_+ t}-e^{\lam_- t}}{\lam_+-\lam_-}\Big|\lesssim   \f{|\xi_1|}{|\xi_2|^2-|\xi_1|^2}e^{-\f{|\xi_1|^2}{|\xi_2|^2}t-|\xi_1|^2t}.
\end{align*}

In Case d, taking $M_1$, $M_2$ and $M_3$ into consideration, we know the most trouble decay is the case $\f{|\xi_1|}{|\xi_2|^2-|\xi_1|^2}e^{-\f{|\xi_1|^2}{|\xi_2|^2}t-|\xi_1|^2t}$ by $|\xi_1|e^{-|\xi_1|^2t}\leq t^{-\f12}$.

For $t>1$, by $\textbf{1}_{|\xi_2|^2>|\xi_1|^2+4|\xi_1|}\f{|\xi_1|}{|\xi_2|^2-|\xi_1|^2}\lesssim\textbf{1}_{|\xi_2|^2>|\xi_1|^2+4|\xi_1|}\f{|\xi_1|+|\xi_1|^2}{|\xi_2|^2}$, and consider the most trouble decay, we have the $L^2$ estimates,
\begin{align}
&\|\textbf{1}_{|\xi_2|^2>|\xi_1|^2+4|\xi_1|}\f{|\xi_1|}{|\xi_2|^2}e^{-\f{|\xi_1|^2}{|\xi_2|^2}t-|\xi_1|^2t}\big(\cF(u_{0}),\cF(b_{0})\big)\|_{L^2_{\xi_1}L^2_{\xi_2}}\non\\
\lesssim&\|\textbf{1}_{|\xi_2|^2>|\xi_1|^2+4|\xi_1|}\f{|\xi_1|}{|\xi_2|^2}e^{-\f{|\xi_1|^2}{|\xi_2|^2}t-|\xi_1|^2t}\|_{L^2_{\xi_1}L^2_{\xi_2}}\|\big(\cF(u_{0}),\cF(b_{0})\big)\|_{L^\infty_{\xi_1}L^\infty_{\xi_2}}\non\\
\lesssim&\|\textbf{1}_{|\xi_1|\geq1}\f{1}{|\xi_1|^{\f12+\delta}}\|_{L^2_{\xi_1}}\|\textbf{1}_{|\xi_2|\geq1}\f{1}{|\xi_2|^{\f12+\delta}}\|_{L^2_{\xi_2}}\|\textbf{1}_{|\xi_2|^2>|\xi_1|^2+4|\xi_1|}\f{|\xi_1|^{\f32+\delta}}{|\xi_2|^{\f32-\delta}}e^{-\f{|\xi_1|^2}{|\xi_2|^2}t-|\xi_1|^2t}\|_{L^\infty_{\xi_1}L^\infty_{\xi_2}}\non\\ &\quad\|\big(\cF(u_{0}),\cF(b_{0})\big)\|_{L^\infty_{\xi_1}L^\infty_{\xi_2}}\non\\
&+\|\textbf{1}_{|\xi_1|<1}\f{1}{|\xi_1|^{\f12-\delta}}\|_{L^2_{\xi_1}}\|\textbf{1}_{|\xi_2|>1}\f{1}{|\xi_2|^{\f12+\delta}}\|_{L^2_{\xi_2}}\|\textbf{1}_{|\xi_2|^2>|\xi_1|^2+4|\xi_1|}\f{|\xi_1|^{\f32-\delta}}{|\xi_2|^{\f32-\delta}}e^{-\f{|\xi_1|^2}{|\xi_2|^2}t-|\xi_1|^2t}\|_{L^\infty_{\xi_1}L^\infty_{\xi_2}}\non\\
&\quad\|\big(\cF(u_{0}),\cF(b_{0})\big)\|_{L^\infty_{\xi_1}L^\infty_{\xi_2}}\non\\
&+\|\textbf{1}_{|\xi_1|<1}\f{1}{|\xi_1|^{\f12-\delta}}\|_{L^2_{\xi_1}}\|\textbf{1}_{|\xi_2|<1}\f{1}{|\xi_2|^{\f12-\delta}}\|_{L^2_{\xi_2}}\|\textbf{1}_{|\xi_2|^2>|\xi_1|^2+4|\xi_1|}\f{|\xi_1|^{\f32-3\delta}}{|\xi_2|^{\f32-3\delta}}\f{|\xi_1|^{2\delta}}{|\xi_2|^{4\delta}}e^{-\f{|\xi_1|^2}{|\xi_2|^2}t-|\xi_1|^2t}\|_{L^\infty_{\xi_1}L^\infty_{\xi_2}}\non\\
&\quad\|\big(\cF(u_{0}),\cF(b_{0})\big)\|_{L^\infty_{\xi_1}L^\infty_{\xi_2}}\non\\
\lesssim& t^{-(\f34-\f{3\delta}{2})}\|(u_{0},b_{0})\|_{L^1},\label{101}
\end{align}
and the heat kernel has the decay
\begin{align*}
&\|e^{-c_0(\xi_1^2t+\xi_2^2t)}\big(\cF(u_{0}),\cF(b_{0})\big)\|_{L^2_{\xi_1}L^2_{\xi_2}}\non\\
\lesssim&\|e^{-c_0(\xi_1^2t+\xi_2^2t)}\|_{L^2_{\xi_1}L^2_{\xi_2}}\|\big(\cF(u_{0}),\cF(b_{0})\big)\|_{L^\infty_{\xi_1}L^\infty_{\xi_2}}\non\\
\lesssim& t^{-\f12}\|(u_{0},b_{0})\|_{L^1}.
\end{align*}

 For $0<t<1$, we have the $L^2$ estimates,
 \begin{align*}
&\|\textbf{1}_{|\xi_2|^2>|\xi_1|^2+4|\xi_1|}\f{|\xi_1|}{|\xi_2|^2-|\xi_1|^2}e^{-\f{|\xi_1|^2}{|\xi_2|^2}t-|\xi_1|^2t}\cF(u_{0})\|_{L^2_{\xi_1}L^2_{\xi_2}}\lesssim\|\cF(u_{0})\|_{L^2_{\xi_1}L^2_{\xi_2}}\lesssim \|u_{0}\|_{L^2},
\end{align*}
and
 \begin{align*}
&\|e^{-c_0(\xi_1^2t+\xi_2^2t)}\cF(u_{0})\|_{L^2_{\xi_1}L^2_{\xi_2}}\lesssim\|\cF(u_{0})\|_{L^2_{\xi_1}L^2_{\xi_2}}\lesssim \|u_{0}\|_{L^2}.
\end{align*}

 $I_2$  can be rewritten as
 \begin{align*}
I_2=&\f{i\xi_1}{|\xi_1|} \f{1}{2\pi i}\int_{\Ga}e^{\lam t} \int^\infty_0  e^{-\om y_2} (\wh{u}_{2, 0}+\f{i\xi_1}{\lam+|\xi_1|^2} \wh{b}_{2, 0})(y_2) dy_2 \f{1}{\om-|\xi_1|} e^{-|\xi_1|x_2}  d\lam\\
&-\f{i\xi_1}{|\xi_1|} \f{1}{2\pi i}\int_{\Ga}e^{\lam t} \int^\infty_0  e^{-\om (x_2+y_2)} (\wh{u}_{2, 0}+\f{i\xi_1}{\lam+|\xi_1|^2} \wh{b}_{2, 0})(y_2) dy_2 \f{1}{\om-|\xi_1|} d\lam\\
:=&J_{1}+J_{2}.
\end{align*}

 {\bf Case 1.  $|\xi_1|\leq 2$}.

We first  divide $J_1$ and $J_2$ as
\beno
J_1=\f{i\xi_1}{|\xi_1|} \f{1}{2\pi i}\sum_{i=1}^{5}\int_{\Ga_i}e^{\lam t} \int^\infty_0  e^{-\om y_2} (\wh{u}_{2, 0}+\f{i\xi_1}{\lam+|\xi_1|^2} \wh{b}_{2, 0})(y_2) dy_2 \f{1}{\om-|\xi_1|} e^{-|\xi_1|x_2} d\lam:= \sum_{i=1}^{5} J_1^i.\\
J_2=-\f{i\xi_1}{|\xi_1|} \f{1}{2\pi i}\sum_{i=1}^{5}\int_{\Ga_i}e^{\lam t} \int^\infty_0  e^{-\om (x_2+y_2)} (\wh{u}_{2, 0}+\f{i\xi_1}{\lam+|\xi_1|^2} \wh{b}_{2, 0})(y_2) dy_2 \f{1}{\om-|\xi_1|}  d\lam:= \sum_{i=1}^5 J_2^i.
\eeno

Then we consider the following case:\\

 {\bf  $\bf{\Gamma_1}$ and $\bf{\Gamma_4}$.}

We first consider the combination of $J_2^1$ and $J_2^4$, for the $\wh{u}_{2,0}$ part, by the definition of $\Ga_1: \lam= -\eta-|\xi_1|^2 \ (\eta: 0 \to |\xi_1|)$ and $\Ga_4: \lam= -\eta-|\xi_1|^2 \ (\eta: |\xi_1| \to \infty)$, we have
\begin{align}\label{I11}
&-\f{1}{2\pi i} \int_{\Ga_1^{(\pm)}\cup \Ga_4 ^{(\pm)} } \f{ e^{\lam t} }{\om-|\xi_1|} \int^\infty_0 e^{-\om(x_2+y_2)}  \wh{u}_{2, 0} (y_2) dy_2 d\lam\non\\
&=-\f{1}{2\pi i} \Big(\int^\infty_0\f{ e^{(-\eta-|\xi_1|^2) t} }{-i|\om|-|\xi_1|} \int^\infty_0 e^{i|\om|(x_2+y_2)}   \wh{u}_{2, 0} (y_2)  dy_2 d\eta\non\\
&\qquad-\int^{\infty}_0\f{ e^{(-\eta-|\xi_1|^2) t} }{i|\om|-|\xi_1|} \int^\infty_0 e^{-i|\om|(x_2+y_2)}   \wh{u}_{2, 0} (y_2) dy_2 d\eta\Big),
\end{align}
with
\begin{align*}
\om=\sqrt{\lam+\f{|\xi_1|^2}{\lam+|\xi_1|^2}}=\mp i \sqrt{\eta+|\xi_1|^2+\f{|\xi_1|^2}{\eta}}=\mp i|\om|.
\end{align*}
By changing of variables $\xi_2=|\om|$, we have
\begin{align*}
& \eta+|\xi_1|^2+\f{|\xi_1|^2}{\eta} =|\xi_2|^2
\Leftrightarrow  \quad\eta^2+(|\xi_1|^2-|\xi_2|^2) \eta+|\xi_1|^2=0,  \non\\
&\Rightarrow\quad \eta_\pm=\f{|\xi_2|^2-|\xi_1|^2}{2} \pm \f{\sqrt{(|\xi_2|^2-|\xi_1|^2)^2-4|\xi_1|^2}}{2}=-\lam_\mp-|\xi_1|^2.
\end{align*}

For the term $e^{(-\eta-|\xi_1|^2) t}$, the least decaying term is the case $\eta=\eta_{-}$. Define $|\xi^{'}|^2=|\xi_2|^2-|\xi_1|^2\geq 2|\xi_1|$, we have
\begin{align*}
\eta_-&=\f{|\xi_2|^2-|\xi_1|^2}{2} - \f{\sqrt{(|\xi_2|^2-|\xi_1|^2)^2-4|\xi_1|^2}}{2}\\
&=\f{|\xi^{'}|^2}{2} (1- \f{\sqrt{|\xi^{'}|^4-4|\xi_1|^2}}{|\xi^{'}|^2})\\
&=\f{|\xi^{'}|^2}{2} (1- \sqrt{1-\f{4|\xi_1|^2}{|\xi^{'}|^4}}),\\
d\eta_-&=\xi_2\Big(1-\f{|\xi^{'}|^2}{\sqrt{|\xi^{'}|^4-4|\xi_1|^2}}\Big) d\xi_2 =2 \xi_2 \f{\lam_++|\xi_1|^2}{\lam_+-\lam_-} d\xi_2.
\end{align*}

We consider the following cases:

\begin{itemize}
\item
 $\f{4|\xi_1|^2}{|\xi^{'}|^4}\leq \f12$, since $(1+s)^\f12 =1+\f{s}{2}\int_0^1(1+\theta s)^{-\f12} d\theta$, we have
 \begin{align*}
 \f{|\xi^{'}|^2}{2}\sqrt{1-\f{4|\xi_1|^2}{|\xi^{'}|^4}}=\f{|\xi^{'}|^2}{2}\Big(1-\f{2|\xi_1|^2}{|\xi^{'}|^4}\int_0^1(1- \f{4|\xi_1|^2}{|\xi^{'}|^4}\theta)^{-\f12} d\theta\Big),
 \end{align*}
   we obtain $\f{|\xi_1|^2}{|\xi^{'}|^2} \leq \eta_{-} \leq \f{\sqrt{2}|\xi_1|^2}{|\xi^{'}|^2}$.\\
 \item
 $\f12 \leq\f{4|\xi_1|^2}{|\xi^{'}|^4}\leq1$, we have $\f{|\xi^{'}|^2}{2}\sqrt{1-\f{4|\xi_1|^2}{|\xi^{'}|^4}}\leq\f{|\xi^{'}|^2}{2\sqrt{2}}$. It then follows that
 \begin{align*}
\eta_{-}&=\f{|\xi^{'}|^2}{2}-\f{|\xi^{'}|^2}{2}\sqrt{1-\f{4|\xi_1|^2}{|\xi^{'}|^4}}\geq\f{2-\sqrt{2}}{4}|\xi^{'}|^2.\\
 -\eta_{-}-|\xi_1|^2&\leq-\f{2-\sqrt{2}}{4}|\xi^{'}|^2-|\xi_1|^2\leq-c_1(|\xi_1|^2+|\xi_2|^2), \ \text{here} \ c_1 \in (0,\f{2-\sqrt{2}}{4}).
  \end{align*}
\end{itemize}

We mention that the most trouble case is
\beno
\eta_{-}\sim\f{|\xi_1|^2}{|\xi^{'}|^2} \quad (\text{when} \quad \f{4|\xi_1|^2}{|\xi^{'}|^4}\leq \f12),
\eeno
(and the other cases can be controlled by heat kernel). For this case,
\beno
d \eta_{-}= -\f{2|\xi_1|^2}{|\xi^{'}|^4} \xi_2 d\xi_2.
\eeno

By the boundary condition $u_{2,0}|_{x_2=0}=0$, the first term on the right hand side of \eqref{I11} can be rewritten as
\begin{align}\label{I11-1}
&-\f{1}{2\pi i} \int^\infty_{\wt{d}} \f{e^{\lam_+ t} }{i\xi_2+|\xi_1|} \int^\infty_0 e^{i\xi_2(x_2+y_2)}   \wh{u}_{2, 0} (y_2)  dy_2 2 \xi_2 \f{\lam_++|\xi_1|^2}{\lam_+-\lam_-} d\xi_2\non\\
=&-\f{1}{2\pi i} \int_{\R}\vphi(\xi_2) e^{i\xi_2 x_2}\f{e^{\lam_+ t} }{i\xi_2+|\xi_1|} \int^\infty_0 e^{i\xi_2y_2}   \wh{u}_{2, 0} (y_2)  dy_2 2 \xi_2 \f{\lam_++|\xi_1|^2}{\lam_+-\lam_-} d\xi_2,
\end{align}
it can be seen as Fourier transform,
\begin{align}
&-\f{1}{2\pi i} \int^\infty_{\wt{d}} \f{e^{\lam_+ t} }{i\xi_2+|\xi_1|} \int^\infty_0 e^{i\xi_2(x_2+y_2)}   \wh{u}_{2, 0} (y_2)  dy_2 2 \xi_2 \f{\lam_++|\xi_1|^2}{\lam_+-\lam_-} d\xi_2\non\\
=&-2i\cF_{x_2}^{-1}\Big(\vphi(\xi_2) \f{\lam_++|\xi_1|^2}{\lam_+-\lam_-} e^{\lam_+ t}  \f{\xi_2 }{i\xi_2+|\xi_1|}\cF_{y_2}\big(\chi  \wh{u}_{2, 0} \big) (-\xi_2) \Big)(x_2),\label{I11-2}
\end{align}
where   $\vphi, \chi$ are   cut-off functions defined as
\begin{align}\label{vphi,chi}
\vphi(\xi_2)=\left\{
\begin{array}{l}
1, \  \text{for} \ \xi_2\geq \wt{d}, \\
0, \  \text{for} \ \xi_2<\wt{d},
\end{array}\right.
\quad \text{and}\quad
\chi(y_2)=\left\{
\begin{array}{l}
1, \  \text{for} \ y_2\geq 0, \\
0, \  \text{for} \ y_2<0,
\end{array}\right.
\end{align}
and $\wt{d}=\sqrt{2\sqrt{2}|\xi_1|+|\xi_1|^2}$.

By  Plancherel theorem, and $|\xi^{'}|^2\geq2|\xi_1|$,  we have the $L^2$ estimate
\begin{align} \label{010}
&\Big\|\textbf{1}_{|\xi_1|\leq 2}\cF_{x_2}^{-1}\Big(\vphi(\xi_2) \f{|\xi_1|^2}{|\xi^{'}|^4} e^{(-\f{|\xi_1|^2}{|\xi^{'}|^2}-|\xi_1|^2) t}  \f{\xi_2 }{i\xi_2+|\xi_1|}\cF_{y_2}\big(\chi  \wh{u}_{2, 0} \big) (-\xi_2) \Big)(x_2)\Big\|_{L^2_{\xi_1}L^2_{x_2}}\non\\
&\lesssim\Big\|\vphi(\xi_2)\f{|\xi_1|^2}{|\xi^{'}|^4} e^{(-\f{|\xi_1|^2}{|\xi^{'}|^2}-|\xi_1|^2) t} \f{\xi_2 }{i\xi_2+|\xi_1|}\cF_{y_2}\big(\chi  \wh{u}_{2, 0} \big) (-\xi_2)    \Big\|_{L^2_{\xi_1}L^2_{\xi_2}}\non\\
&\lesssim\Big\|\vphi(\xi_2)\f{|\xi_1|}{|\xi^{'}|^2}e^{(-\f{|\xi_1|^2}{|\xi^{'}|^2}-|\xi_1|^2) t} \cF_{y_2}\big(\chi  \wh{u}_{2, 0} \big) (-\xi_2)   \Big\|_{L^2_{\xi_1}L^2_{\xi_2}}\non\\
&\lesssim\Big\|\vphi(\xi_2)\f{|\xi_1|+|\xi_1|^2}{|\xi^{'}|^2+|\xi_1|^2}e^{(-\f{|\xi_1|^2}{|\xi_2|^2}-|\xi_1|^2) t} \cF_{y_2}\big(\chi  \wh{u}_{2, 0} \big) (-\xi_2)   \Big\|_{L^2_{\xi_1}L^2_{\xi_2}}\non\\
&\lesssim\|\textbf{1}_{|\xi_1|\geq1}\f{1}{|\xi_1|^{\f12+\delta}}\|_{L^2_{\xi_1}}\|\textbf{1}_{|\xi_2|\geq1}\f{1}{|\xi_2|^{\f12+\delta}}\|_{L^2_{\xi_2}}\|(1+|\xi_1|)\f{|\xi_1|^{\f32+\delta}}{|\xi_2|^{\f32-\delta}}e^{-\f{|\xi_1|^2}{|\xi_2|^2}t-|\xi_1|^2t}\cF_{y_2}\big(\chi  \wh{u}_{2, 0} \big) (-\xi_2)\|_{L^\infty_{\xi_1}L^\infty_{\xi_2}}\non\\
&\quad+\|\textbf{1}_{|\xi_1|<1}\f{1}{|\xi_1|^{\f12-\delta}}\|_{L^2_{\xi_1}}\|\textbf{1}_{|\xi_2|>1}\f{1}{|\xi_2|^{\f12+\delta}}\|_{L^2_{\xi_2}}\|(1+|\xi_1|)\f{|\xi_1|^{\f32-\delta}}{|\xi_2|^{\f32-\delta}}e^{-\f{|\xi_1|^2}{|\xi_2|^2}t-|\xi_1|^2t}\cF_{y_2}\big(\chi  \wh{u}_{2, 0} \big) (-\xi_2)\|_{L^\infty_{\xi_1}L^\infty_{\xi_2}}\non\\
&\quad+\|\textbf{1}_{|\xi_1|<1}\f{1}{|\xi_1|^{\f12-\delta}}\|_{L^2_{\xi_1}}\|\textbf{1}_{|\xi_2|<1}\f{1}{|\xi_2|^{\f12-\delta}}\|_{L^2_{\xi_2}}\|\varphi(\xi_2)(1+|\xi_1|)\f{|\xi_1|^{\f32-3\delta}}{|\xi_2|^{\f32-3\delta}}\f{|\xi_1|^{2\delta}}{|\xi_2|^{4\delta}}e^{-\f{|\xi_1|^2}{|\xi_2|^2}t-|\xi_1|^2t}\cF_{y_2}\big(\chi  \wh{u}_{2, 0} \big) (-\xi_2)\|_{L^\infty_{\xi_1}L^\infty_{\xi_2}}\non\\
&\lesssim t^{-(\f34-\f32\delta)}    \| u_{2, 0} \|_{L^1}.
\end{align}

The $\wh{b}_{2,0}$ part of $J_2^1$ and $J_2^4$ can be estimated similarly since
\beno
\f{|\xi_1|^2}{|\xi^{'}|^4} \f{|\xi_1|}{\big|\lam_++|\xi_1|^2\big|}\sim\f{|\xi_1|^2}{|\xi^{'}|^4} \f{|\xi^{'}|^2}{|\xi_1|}=\f{|\xi_1|}{|\xi^{'}|^2}.
\eeno

The only terms that $J_1^{1,4}$ differs from $J_2^{1,4}$ are $e^{-|\xi_1|x_2}$ in $J_1^{1,4}$ and $e^{-\om x_2}$ in $J_2^{1,4}$, then change \eqref{I11}   $e^{-\om x_2}$ into $e^{-|\xi_1|x_2}$, use $\|e^{-|\xi_1|x_2}\|_{L_{x_2}^2}\lesssim |\xi_1|^{-\f12}$, and similarly \eqref{I11-1}, by the boundary condition $\wh{u}_{2,0}|_{x_2=0}=0$, we have the $L^2$ estimate
\begin{align}\label{009}
&\Big\|\textbf{1}_{|\xi_1|\leq 2}\int_{\R}\vphi(\xi_2)\f{|\xi_1|^2}{|\xi^{'}|^4} e^{(-\f{|\xi_1|^2}{|\xi^{'}|^2}-|\xi_1|^2) t} \f{\xi_2 }{i\xi_2+|\xi_1|}|\xi_1|^{-\f12}\cF_{y_2}\big(\chi  \wh{u}_{2, 0} \big) (-\xi_2)d\xi_2   \Big\|_{L^2_{\xi_1}}\non\\
&=\Big\|\textbf{1}_{|\xi_1|\leq 2}\int_{\R}\vphi(\xi_2)\f{|\xi_1|^2}{|\xi^{'}|^4} e^{(-\f{|\xi_1|^2}{|\xi^{'}|^2}-|\xi_1|^2) t} \f{\xi_2 }{i\xi_2+|\xi_1|}|\xi_1|^{-\f12}\int_0^{\infty}e^{i\xi_2y_2}  \wh{u}_{2, 0} dy_2d\xi_2   \Big\|_{L^2_{\xi_1}}\non\\
&=\Big\|\textbf{1}_{|\xi_1|\leq 2}\int_{\R}\vphi(\xi_2)\f{|\xi_1|^2}{|\xi^{'}|^4} e^{(-\f{|\xi_1|^2}{|\xi^{'}|^2}-|\xi_1|^2) t} \f{\xi_2 }{i\xi_2+|\xi_1|}|\xi_1|^{-\f12}\int_0^{\infty}e^{i\xi_2y_2}  \f{1}{i\xi_2}\pa_2\wh{u}_{2, 0} dy_2d\xi_2   \Big\|_{L^2_{\xi_1}}\non\\
&=\Big\|\textbf{1}_{|\xi_1|\leq 2}\int_{\R}\vphi(\xi_2)\f{|\xi_1|^2}{|\xi^{'}|^4} e^{(-\f{|\xi_1|^2}{|\xi^{'}|^2}-|\xi_1|^2) t} \f{\xi_2 }{i\xi_2+|\xi_1|}|\xi_1|^{-\f12}\int_0^{\infty}e^{i\xi_2y_2}  \f{|\xi_1|}{i\xi_2}\wh{u}_{1, 0} dy_2d\xi_2   \Big\|_{L^2_{\xi_1}}\non\\
&\lesssim\Big\|\textbf{1}_{|\xi_1|\leq 2}\vphi(\xi_2)\f{|\xi_1|^2}{|\xi^{'}|^4} e^{(-\f{|\xi_1|^2}{|\xi^{'}|^2}-|\xi_1|^2) t} \f{|\xi_2|}{(|\xi_1|^2+|\xi_2|^2)^\f12}\f{|\xi_1|^\f12}{|\xi_2|}\cF_{y_2}\big(\chi  \wh{u}_{1, 0} \big) (-\xi_2)   \Big\|_{L^2_{\xi_1}L^1_{\xi_2}}\non\\
&\lesssim\Big\|\vphi(\xi_2)\f{|\xi_1|}{|\xi^{'}|^2}\f{|\xi_1|^\f12}{(|\xi_1|^2+|\xi_2|^2)^\f12} e^{(-\f{|\xi_1|^2}{|\xi^{'}|^2}-|\xi_1|^2) t}\cF_{y_2}\big(\chi  \wh{u}_{1, 0} \big)(-\xi_2)   \Big\|_{L^2_{\xi_1}L^1_{\xi_2}}\non\\
&\lesssim\Big\|\vphi(\xi_2)\f{|\xi_1|}{|\xi^{'}|^2}e^{(-\f{|\xi_1|^2}{|\xi^{'}|^2}-|\xi_1|^2) t}\cF_{y_2}\big(\chi  \wh{u}_{1, 0} \big)(-\xi_2)   \Big\|_{L^2_{\xi_1}L^2_{\xi_2}}\Big\|\f{|\xi_1|^\f12}{(|\xi_1|^2+|\xi_2|^2)^\f12}    \Big\|_{L^\infty_{\xi_1}L^2_{\xi_2}}\non\\
&\lesssim t^{-(\f34-\delta)}    \| u_{1, 0} \|_{L^1},
\end{align}
where we use
\begin{align}
\Big\|\f{|\xi_1|^\f12}{(|\xi_1|^2+|\xi_2|^2)^\f12}    \Big\|_{L^\infty_{\xi_1}L^2_{\xi_2}}&=\sup\limits_{\xi_1}(\int_0^\infty \f{|\xi_1|}{|\xi_1|^2+|\xi_2|^2} d\xi_2)^\f12=\sup\limits_{\xi_1}(\int_0^\infty \f{1}{1+j^2} d j )^\f12\non\\
&=\sup\limits_{\xi_1}(arctan j\big|_0^\infty)^\f12\leq \sqrt{\f{\pi}{2}}.\label{2000}
\end{align}

Consider the case  $0<t<1$, we have for the $L^2$ estimate of $J^{1,4}_2$
\begin{align*}
\Big\|\vphi(\xi_2)\f{|\xi_1|}{|\xi^{'}|^2}\f{|\xi_1|}{(|\xi_1|^2+|\xi_2|^2)^\f12} e^{(-\f{|\xi_1|^2}{|\xi^{'}|^2}-|\xi_1|^2) t} \cF_{y_2}\big(\chi  \wh{u}_{1, 0} \big) (-\xi_2)   \Big\|_{L^2_{\xi_1}L^2_{\xi_2}}
\lesssim \Big\|\f{|\xi_1|}{|\xi^{'}|^2}  \cF_{y_2}\big(\chi  \wh{u}_{1, 0} \big) (-\xi_2)   \Big\|_{L^2_{\xi_1}L^2_{\xi_2}}\lesssim \|u_{1, 0}\|_{L^2},
\end{align*}
the $L^2$ estimate of $J^{1,4}_1$
\begin{align*}
\Big\|\vphi(\xi_2)\f{|\xi_1|^\f12}{|\xi^{'}|^2}\f{|\xi_1|}{(|\xi_1|^2+|\xi_2|^2)^\f12} e^{(-\f{|\xi_1|^2}{|\xi^{'}|^2}-|\xi_1|^2) t} \cF_{y_2}\big(\chi  \wh{u}_{1, 0} \big) (-\xi_2)   \Big\|_{L^2_{\xi_1}L^1_{\xi_2}}
\lesssim \Big\|\f{|\xi_1|}{|\xi^{'}|^2}  \cF_{y_2}\big(\chi  \wh{u}_{1, 0} \big) (-\xi_2)   \Big\|_{L^2_{\xi_1}L^2_{\xi_2}}\lesssim \|u_{1, 0}\|_{L^2}.\\
\end{align*}

 {\bf   $\bf{\Gamma_2}$ and $\bf{\Gamma_3}$}.\\

For the $\wh{u}_{2,0}$ part of $J_2^2$, by the definition of $\Ga_2$: $ \lam=\lam'_+-\eta+i\big(-\text{Im} \lam'_++D(\eta, \xi_1)\big) \ (0< \eta <d_0)$, we have
\begin{align}\label{I12}
&-\f{1}{2\pi i} \int_{\Ga_2^{(+)}\cup \Ga_2 ^{(-)} } e^{\lam t}  \f{1}{ \om-|\xi_1|} \int^\infty_0 e^{-\om(x_2+y_2)}  \wh{u}_{2, 0} (y_2) dy_2 d\lam\non\\
&=-\f{1}{2\pi i} \Big(\int^{d_0}_0\f{ e^{\lam t}}{i|\om|-|\xi_1|} \int^\infty_0 e^{-i|\om|(x_2+y_2)}   \wh{u}_{2, 0} (y_2) dy_2 (-1+i \f{\text{Re} \lam'_++|\xi_1|^2-\eta}{D(\eta, \xi_1)} ) d\eta\non\\
&\quad\qquad-\int^{d_0}_0\f{ e^{\lam t}}{-i|\om|-|\xi_1|} \int^\infty_0 e^{i|\om|(x_2+y_2)}   \wh{u}_{2, 0} (y_2) dy_2 (-1+i \f{\text{Re} \lam'_++|\xi_1|^2-\eta}{D(\eta, \xi_1)})  d\eta\Big),\end{align}
with
\begin{align*}
\om&=\sqrt{\f{(\lam-\lam'_+)(\lam-\lam'_-)}{\lam+|\xi_1|^2}}=\sqrt{-2\eta}=\pm i \sqrt{2\eta}=\pm i|w|.
\end{align*}

By changing of variables $\xi_2=|\om|$, we have
\begin{align*}
\lam&= \text{Re} \lam'_+-\eta+i D(\eta, \xi_1)\\
&=-\f{|\xi_1|^2+|\xi_2|^2}{2}+i \sqrt{|\xi_1|^2-\f{|\xi_1|^4}{4}+\f{|\xi_1|^2|\xi_2|^2}{2}-\f{|\xi_2|^4}{4}}=\lam_+,
\end{align*}
and
\beno
-1+i \f{\text{Re} \lam'_++|\xi_1|^2-\eta}{D(\eta, \xi_1)} =-\f{2(\lam_++|\xi_1|^2)}{\lam_+-\lam_-}.
\eeno

Thus \eqref{I12} can be rewritten as
\begin{align}\label{J22}
&\f{1}{\pi i } \Big(\int^{\wt{d}}_0 \f{(\lam_++|\xi_1|^2) e^{\lam_+t}}{\lam_+-\lam_-}
\f{1}{i\xi_2-|\xi_1|} \int^\infty_0 e^{-i\xi_2(x_2+y_2)}   \wh{u}_{2, 0} (y_2) dy_2  \xi_2d\xi_2\non\\
&\quad\qquad+\int^{\wt{d}}_0 \f{(\lam_++|\xi_1|^2) e^{\lam_+ t}}{\lam_+-\lam_-}
\f{1}{i\xi_2+|\xi_1|} \int^\infty_0 e^{i\xi_2(x_2+y_2)}   \wh{u}_{2, 0} (y_2) dy_2  \xi_2 d\xi_2\Big).
\end{align}

 Similarly, the $\wh{u}_{2,0}$ part of $J_2^3$ can be rewritten as
\begin{align}\label{J23}
&-\f{1}{\pi i } \Big(\int^{\wt{d}}_0 \f{(\lam_-+|\xi_1|^2)e^{\lam_- t}}{\lam_+-\lam_-}
\f{1}{i\xi_2-|\xi_1|} \int^\infty_0 e^{-i\xi_2(x_2+y_2)}   \wh{u}_{2, 0} (y_2) dy_2  \xi_2 d\xi_2\non\\
&\quad\qquad+\int^{\wt{d}}_0 \f{(\lam_-+|\xi_1|^2) e^{\lam_- t}}{\lam_+-\lam_-}
\f{1}{i\xi_2+|\xi_1|} \int^\infty_0 e^{i\xi_2(x_2+y_2)}   \wh{u}_{2, 0} (y_2) dy_2  \xi_2 d\xi_2\Big).
\end{align}

By the definition of $\wt{d}$, we have $|\xi_1|\geq \f{|\xi^{'}|^2}{2}$£¬ and similar \eqref{203} and \eqref{M2},
\begin{align*}
&\Big|\f{(\lam_-+|\xi_1|^2)e^{\lam_- t}-(\lam_++|\xi_1|^2) e^{\lam_+ t}}{\lam_+-\lam_-}\Big|
\lesssim e^{-\f{|\xi_1|^2+|\xi_2|^2}{2} t}.
\end{align*}

Thus, the $\wh{u}_{2,0}$ part  of $J_2^2+J_2^3$ can be controlled by heat kernel, we have the $L^2$ estimate
\begin{align*}
&\Big\| \int^{\wt{d}}_0\f{(\lam_-+|\xi_1|^2)e^{\lam_- t}-(\lam_++|\xi_1|^2) e^{\lam_+ t}}{\lam_+-\lam_-}\f{\xi_2}{i\xi_2-|\xi_1|} \int^\infty_0 e^{-i\xi_2(x_2+y_2)}   \wh{u}_{2, 0} (y_2) dy_2  d\xi_2\Big\|_{L^2_{\xi_1}L^2_{x_2}}\\
\lesssim&\Big\| \int_{\R} \big(1-\vphi(\xi_2)\big)e^{-\f{|\xi_1|^2+|\xi_2|^2}{2} t}\f{|\xi_2|}{(|\xi_1|^2+|\xi_2|^2)^\f12}  e^{-i\xi_2x_2} \cF_{y_2}\Big(\chi(y_2)  \wh{u}_{2, 0} (y_2)  \Big)(\xi_2) d\xi_2\Big\|_{L^2_{\xi_1}L^2_{x_2}}\\
\lesssim&\Big\| \cF^{-1}_{x_2}\Big(\big(1-\vphi(\xi_2)\big)e^{-\f{|\xi_1|^2+|\xi_2|^2}{2} t}\f{|\xi_2|}{(|\xi_1|^2+|\xi_2|^2)^\f12} \cF_{y_2}\Big(\chi(y_2)  \wh{u}_{2, 0} (y_2)  \Big)(\xi_2) \Big)(-x_2) \Big\|_{L^2_{\xi_1}L^2_{x_2}}\\
\lesssim&\Big\|\big(1-\vphi(\xi_2)\big)e^{-\f{|\xi_1|^2+|\xi_2|^2}{2} t}
\f{|\xi_2|}{(|\xi_1|^2+|\xi_2|^2)^\f12}\Big\|_{L^2_{\xi_1}L^2_{\xi_2}}\Big\|\cF_{y_2}\Big(\chi(y_2)  \wh{u}_{2, 0} (y_2)  \Big)\Big\|_{L^\infty_{\xi_1}L^\infty_{\xi_2}}\\
\lesssim&\|e^{-\f{|\xi_1|^2}{2} t}e^{-\f{|\xi_2|^2}{2}t} \|_{L^2_{\xi_1}L^2_{\xi_2}}\|u_{2,0}\|_{L^1_{x_1}L^1_{x_2}}\\
\lesssim&  t^{-\f12}\|u_{2,0}\|_{L^1}.
\end{align*}

Similarly, we have the estimate for the $\wh{b}_{2,0}$ part  of $J_2^2+J_2^3$, since
\begin{align}\label{011}
 \f{|i\xi_1|}{|\lam+|\xi_1|^2|}=\f{|\xi_1|}{|\lam_\pm+|\xi_1|^2|}= 1.
\end{align}

For the $J_1^2+J_1^3$, by the case $\|e^{-|\xi_1|x_2}\|_{L^2_{x_2}}\lesssim |\xi_1|^{-\f12}$, we have the $L^2$ estimate:
\begin{align*}
&\Big\|\textbf{1}_{|\xi_1|\leq 2}\int_{\R}\vphi(\xi_2)\f{|\xi_1|^2}{|\xi^{'}|^4} e^{(-\f{|\xi_1|^2}{|\xi^{'}|^2}-|\xi_1|^2) t} \f{\xi_2 }{i\xi_2+|\xi_1|}|\xi_1|^{-\f12}\cF_{y_2}\big(\chi  \wh{u}_{2, 0} \big) (-\xi_2)d\xi_2   \Big\|_{L^2_{\xi_1}}\non\\
\lesssim&\Big\| \big(1-\vphi(\xi_2)\big)\f{(\lam_-+|\xi_1|^2)e^{\lam_- t}-(\lam_++|\xi_1|^2) e^{\lam_+ t}}{\lam_+-\lam_-}\f{-|\xi_1|}{i\xi_2-|\xi_1|}|\xi_1|^{-\f12} \cF_{y_2}\Big(\chi(y_2)  \wh{u}_{1, 0} (y_2)  \Big)(\xi_2)\Big\|_{L^2_{\xi_1}L^1_{\xi_2}}\\
\lesssim&\Big\|\big(1-\vphi(\xi_2)\big)\f{(\lam_-+|\xi_1|^2)e^{\lam_- t}-(\lam_++|\xi_1|^2) e^{\lam_+ t}}{\lam_+-\lam_-}
\f{|\xi_1|^\f12}{(|\xi_1|^2+|\xi_2|^2)^\f12}\Big\|_{L^2_{\xi_1}L^1_{\xi_2}}\Big\|\cF_{y_2}\Big(\chi(y_2)  \wh{u}_{1, 0} (y_2)  \Big)\Big\|_{L^\infty_{\xi_1}L^\infty_{\xi_2}}\\
\lesssim&\|e^{-\f{|\xi_1|^2}{2} t}e^{-\f{|\xi_2|^2}{2}t} \|_{L^2_{\xi_1}L^2_{\xi_2}}\|u_{1,0}\|_{L^1_{x_1}L^1_{x_2}}\\
\lesssim&  t^{-\f12}\|u_{1,0}\|_{L^1},
\end{align*}
here we use \eqref{2000}.\\

Consider the case $0<t<1$,
\begin{align*}
\Big|\f{(\lam_-+|\xi_1|^2)e^{\lam_- t}-(\lam_++|\xi_1|^2) e^{\lam_+ t}}{\lam_+-\lam_-}\Big|
\lesssim 1,
\end{align*}
the $L^2$ estimate for $J^2_2+J^3_2$,
\begin{align*}
&\Big\|\big(1-\vphi(\xi_2)\big)\f{(\lam_-+|\xi_1|^2)e^{\lam_- t}-(\lam_++|\xi_1|^2) e^{\lam_+ t}}{\lam_+-\lam_-}
\f{\xi_2}{i\xi_2-|\xi_1|}\Big\|_{L^\infty_{\xi_1}L^\infty_{\xi_2}}\Big\|\cF_{y_2}\Big(\chi(y_2)  \wh{u}_{2, 0} (y_2)  \Big)\Big\|_{L^2_{\xi_1}L^2_{\xi_2}}\\
\lesssim&  \|u_{2,0}\|_{L^2},
\end{align*}
the $L^2$ estimate for $J^2_1+J^3_1$,
\begin{align*}
&\Big\|\big(1-\vphi(\xi_2)\big)\f{(\lam_-+|\xi_1|^2)e^{\lam_- t}-(\lam_++|\xi_1|^2) e^{\lam_+ t}}{\lam_+-\lam_-}
\f{|\xi_1|^\f12}{(|\xi_1|^2+|\xi_2|^2)^\f12}\Big\|_{L^\infty_{\xi_1}L^2_{\xi_2}}\Big\|\cF_{y_2}\Big(\chi(y_2)  \wh{u}_{1, 0} (y_2)  \Big)\Big\|_{L^2_{\xi_1}L^2_{\xi_2}}\\
\lesssim&  \|u_{1,0}\|_{L^2}.
\end{align*}

 {   $\bf{\Gamma_5}$.}

For the  $\wh{u}_{2,0}$ part and $\wh{b}_{2,0}$ part of $I_2^5=J_1^5+J_2^5$,  define $L_1$ is the $\wh{u}_{2,0}$ part of the $I_2^5$, $L_2$ is the $\wh{b}_{2,0}$ part of the $I_2^5$, then by $\lam=-|\xi_1|^2+\e e^{i\theta}$, and the boundary condition $b_{2,0}|_{x_2=0}=0$, we have
\begin{align*}
|\textbf{1}_{|\xi_1|>\delta} L_1|
=&\Big|\textbf{1}_{|\xi_1|>\delta}\f{1}{2\pi i} \int_{\Ga_5 } e^{\lam t}  \f{e^{-|\xi_1|x_2}-e^{-\om x_2}}{ \om-|\xi_1|} \int^\infty_0 e^{-\om y_2}  \wh{u}_{2,0}(y_2)dy_2 d\lam\Big|\non\\
\lesssim& \textbf{1}_{|\xi_1|>\delta}e^{\e t}\int_{\Ga_5 } \Big|  \f{e^{-|\xi_1|x_2}-e^{-\om x_2}}{ \om-|\xi_1|} \int^\infty_0 e^{-\om y_2}  \wh{u}_{2,0}(y_2)dy_2 \Big|d\lam\non\\
\lesssim& \textbf{1}_{|\xi_1|>\delta}e^{\e t}\int_{\Ga_5 } \Big|  \f{e^{-|\xi_1|x_2}-e^{-\om x_2}}{ \om-|\xi_1|} \Big|d\lam \sup_{\lam \in \Gamma_5}\Big|\int^\infty_0 e^{-\om y_2}  \wh{u}_{2,0}dy_2 \Big|\non\\
\lesssim& \textbf{1}_{|\xi_1|>\delta}e^{\e t}\int_{\Ga_5 } \Big|  \f{e^{-|\xi_1|x_2}-e^{-\om x_2}}{ \om-|\xi_1|} \Big|d\lam\|\wh{u}_{2, 0}\|_{L^1_{x_2}},
\end{align*}
and
\begin{align}\label{J25-1}
|\textbf{1}_{|\xi_1|>\delta} L_2|
=&\Big|\textbf{1}_{|\xi_1|>\delta}\f{1}{2\pi i} \int_{\Ga_5 } e^{\lam t}  \f{e^{-|\xi_1|x_2}-e^{-\om x_2}}{ \om-|\xi_1|} \f{i\xi_1}{\lam+|\xi_1|^2}\int^\infty_0 e^{-\om y_2}  \wh{b}_{2,0}(y_2)dy_2 d\lam\Big|\non\\
=&\Big|\textbf{1}_{|\xi_1|>\delta}\f{1}{2\pi i} \int_{\Ga_5 } e^{\lam t}  \f{e^{-|\xi_1|x_2}-e^{-\om x_2}}{ \om-|\xi_1|} \f{i\xi_1}{\lam+|\xi_1|^2}\f{1}{\om}\int^\infty_0 e^{-\om y_2}  \pa_2\wh{b}_{2,0}(y_2)dy_2 d\lam\Big|\non\\
\lesssim& \textbf{1}_{|\xi_1|>\delta}e^{\e t}\int_{\Ga_5 } \Big|  \f{e^{-|\xi_1|x_2}-e^{-\om x_2}}{ \om-|\xi_1|} \f{i\xi_1}{\lam+|\xi_1|^2}\f{1}{\om}\int^\infty_0 e^{-\om y_2}  \pa_2\wh{b}_{2,0}(y_2)dy_2 \Big|d\lam\non\\
\lesssim& \textbf{1}_{|\xi_1|>\delta}e^{\e t}\int_{\Ga_5 } \Big|  \f{e^{-|\xi_1|x_2}-e^{-\om x_2}}{ \om-|\xi_1|} \Big|d\lam \sup_{\lam \in \Gamma_5}\Big|\f{i\xi_1}{\lam+|\xi_1|^2}\f{1}{\om}\int^\infty_0 e^{-\om y_2}  \pa_2\wh{b}_{2,0}dy_2 \Big|\non\\
\lesssim&\e^{-\f12} \textbf{1}_{|\xi_1|>\delta}e^{\e t}\int_{\Ga_5 } \Big|  \f{e^{-|\xi_1|x_2}-e^{-\om x_2}}{ \om-|\xi_1|} \Big|d\lam\|\pa_2\wh{b}_{2, 0}\|_{L^1_{x_2}},
\end{align}
here we use when $\e \to 0$,
\begin{align}
&\Big|\f{i\xi_1}{\lam+|\xi_1|^2}\f{1}{\om}\Big|=\f{|\xi_1|}{\e}\Big|\f{\sqrt{\lam+|\xi_1|^2}}{\sqrt{\lam^2+\lam|\xi_1|^2+|\xi_1|^2}}\Big|=\f{|\xi_1|}{\e}\f{\sqrt{\e }}{\big|\sqrt{\e e^{i\gamma}(-|\xi_1|^2+\e e^{i\gamma})+|\xi_1|^2}\big|}\lesssim\e^{-\f12}.\label{002}
\end{align}

Then
\begin{align*}
\|\textbf{1}_{|\xi_1|>\delta}L_1\|_{L^2_{\xi_1}L^2_{x_2}}
\lesssim&e^{\e t} \Big\|\textbf{1}_{|\xi_1|>\delta} \int_{\Ga_5 }  \Big| \f{1}{ \om-|\xi_1|} \Big|  \|e^{-|\xi_1|x_2}-e^{-\om x_2}\|_{L^2_{x_2}}d\lam\Big\|_{L^2_{\xi_1}}\| u_{2, 0}\|_{L^1}\non\\
\lesssim&e^{\e t} \Big\|\textbf{1}_{|\xi_1|>\delta} \int_{\Ga_5 }  \f{\e^\f12}{|\xi_1|}\Big(\f{1}{|\xi_1|^\f12}+\f{1}{(Re\om)^\f12}\Big)  d\lam\Big\|_{L^2_{\xi_1}}\| u_{2, 0}\|_{L^1}\non\\
\lesssim&e^{\e t} \Big\|\textbf{1}_{|\xi_1|>\delta}  \f{\e^\f12}{|\xi_1|}\Big(\f{\e}{|\xi_1|^\f12}+\f{\e^\f54}{|\xi_1|^\f12}\Big)  \Big\|_{L^2_{\xi_1}}\| u_{2, 0}\|_{L^1}\non\\
\lesssim&\e^\f32 e^{\e t}\|\textbf{1}_{|\xi_1|>\delta}\f{1}{|\xi_1|^\f32}\|_{L^2_{\xi_1}}\|u_{2, 0}\|_{L^1}\non\\
\lesssim&\e^\f32 e^{\e t}\|u_{2, 0}\|_{L^1} \to 0, \quad\text{for}\quad \e \to 0, \non
\end{align*}
and
\begin{align*}
\|\textbf{1}_{|\xi_1|>\delta}L_2\|_{L^2_{\xi_1}L^2_{x_2}}
\lesssim&\e^{-\f12}e^{\e t} \Big\|\textbf{1}_{|\xi_1|>\delta} \int_{\Ga_5 }  \Big| \f{1}{ \om-|\xi_1|} \Big|  \|e^{-|\xi_1|x_2}-e^{-\om x_2}\|_{L^2_{x_2}}d\lam\Big\|_{L^2_{\xi_1}}\|\pa_2 b_{2, 0}\|_{L^1}\non\\
\lesssim&\e^{-\f12}e^{\e t} \Big\|\textbf{1}_{|\xi_1|>\delta} \int_{\Ga_5 }  \f{\e^\f12}{|\xi_1|}\Big(\f{1}{|\xi_1|^\f12}+\f{1}{(Re\om)^\f12}\Big)  d\lam\Big\|_{L^2_{\xi_1}}\|\pa_2 b_{2, 0}\|_{L^1}\non\\
\lesssim&\e^{-\f12}e^{\e t} \Big\|\textbf{1}_{|\xi_1|>\delta}  \f{\e^\f12}{|\xi_1|}\Big(\f{\e}{|\xi_1|^\f12}+\f{\e^\f54}{|\xi_1|^\f12}\Big)  \Big\|_{L^2_{\xi_1}}\|\pa_2 b_{2, 0}\|_{L^1}\non\\
\lesssim&\e e^{\e t}\|\textbf{1}_{|\xi_1|>\delta}\f{1}{|\xi_1|^\f32}\|_{L^2_{\xi_1}}\|\pa_2 b_{2, 0}\|_{L^1}\non\\
\lesssim&\e e^{\e t}\|\pa_2 b_{2, 0}\|_{L^1} \to 0, \quad\text{for}\quad \e \to 0, \non
\end{align*}
here we use
\begin{align*}
&\om^2=\lam+\f{|\xi_1|^2}{\lam+|\xi_1|^2}=-|\xi_1|^2+\e e^{i\gamma}+\f{|\xi_1|^2}{\e}e^{-i\gamma},
\end{align*}
and by \eqref{lemma1}, we have
\begin{align}
 (Re\om)^2&=\f{\sqrt{|\xi_1|^4+\e^2+\f{|\xi_1|^4}{\e^2}-2|\xi_1|^2\e \cos\gamma-\f{2|\xi_1|^4}{\e}\cos\gamma+2|\xi_1|^2 \cos 2\gamma}-|\xi_1|^2+\e \cos\gamma+\f{|\xi_1|^2}{\e}\cos\gamma}{2}\non\\
 & \sim\f{|\xi_1|^2}{\e}\f{1+\cos{\gamma}}{2}=\f{|\xi_1|^2}{\e}(\cos{\f{\gamma}{2}})^2, \label{reom}
 \end{align}
and when $\e \to 0$,
 \begin{align}
&\Big|\f{1}{\om-|\xi_1|}\Big|=\Big|\f{\sqrt{\lam+|\xi_1|^2}}{ \sqrt{\lam^2+|\xi_1|^2\lam+|\xi_1|^2}-|\xi_1|\sqrt{\lam+|\xi_1|^2}}\Big|=\Big|\f{\sqrt{\e e^{i\gamma}}}{ \sqrt{(-|\xi_1|^2+\e e^{i\gamma})\e e^{i\gamma}+|\xi_1|^2}-|\xi_1|\sqrt{\e e^{i\gamma}}}\Big|\lesssim \f{\e^\f12}{|\xi_1|}. \label{L2-01}
\end{align}

As the $\textbf{1}_{|\xi_1|\leq\delta}$ part, we have
\begin{align*}
|\textbf{1}_{|\xi_1|\leq\delta} L_1|
=&\Big|\textbf{1}_{|\xi_1|\leq\delta}\f{1}{2\pi i} \int_{\Ga_5 } e^{\lam t}  \f{e^{-|\xi_1|x_2}-e^{-\om x_2}}{ \om-|\xi_1|} \int^\infty_0 e^{-\om y_2}  \wh{u}_{2,0}(y_2)dy_2 d\lam\Big|\non\\
\lesssim& \textbf{1}_{|\xi_1|\leq\delta}e^{\e t}\int_{\Ga_5 } \Big|  \f{e^{-|\xi_1|x_2}-e^{-\om x_2}}{ \om-|\xi_1|} \Big|d\lam\|\wh{u}_{2, 0}\|_{L^1_{x_2}},
\end{align*}
and
\begin{align*}
|\textbf{1}_{|\xi_1|\leq\delta} L_2|
=&\Big|\textbf{1}_{|\xi_1|\leq\delta}\f{1}{2\pi i} \int_{\Ga_5 } e^{\lam t}  \f{e^{-|\xi_1|x_2}-e^{-\om x_2}}{ \om-|\xi_1|} \f{i\xi_1}{\lam+|\xi_1|^2}\int^\infty_0 e^{-\om y_2}  \wh{b}_{2,0}(y_2)dy_2 d\lam\Big|\non\\
\lesssim&\f{|\xi_1|}{\e} \textbf{1}_{|\xi_1|\leq\delta}e^{\e t}\int_{\Ga_5 } \Big|  \f{e^{-|\xi_1|x_2}-e^{-\om x_2}}{ \om-|\xi_1|} \Big|d\lam\|\wh{b}_{2, 0}\|_{L^1_{x_2}},
\end{align*}
where we use
\begin{align}\label{ixi1}
&\Big|\f{i\xi_1}{\lam+|\xi_1|^2}\Big|=\f{|\xi_1|}{\e}.
\end{align}

Since the definition of contour, we have $\e\leq|\xi_1|$, then by \eqref{reom} and \eqref{L2-01},  we have
\begin{align*}
\|\textbf{1}_{|\xi_1|\leq\delta}L_1\|_{L^2_{\xi_1}L^2_{x_2}}
\lesssim&e^{\e t} \Big\|\textbf{1}_{|\xi_1|\leq\delta} \int_{\Ga_5 }  \Big| \f{1}{ \om-|\xi_1|} \Big|  \|e^{-|\xi_1|x_2}-e^{-\om x_2}\|_{L^2_{x_2}}d\lam\Big\|_{L^2_{\xi_1}}\| u_{2, 0}\|_{L^1}\non\\
\lesssim&e^{\e t} \Big\|\textbf{1}_{|\xi_1|\leq\delta} \int_{\Ga_5 }  \f{\e^\f12}{|\xi_1|}\Big(\f{1}{|\xi_1|^\f12}+\f{1}{(Re\om)^\f12}\Big)  d\lam\Big\|_{L^2_{\xi_1}}\| u_{2, 0}\|_{L^1}\non\\
\lesssim&\e^{\f12-\zeta} e^{\e t}\|\textbf{1}_{|\xi_1|\leq\delta}\e^{\zeta}\f{1}{|\xi_1|^\f12}\|_{L^2_{\xi_1}}\| u_{2, 0}\|_{L^1}\non\\
\lesssim&\e^{\f12-\zeta} e^{\e t}\|\textbf{1}_{|\xi_1|\leq\delta}\f{1}{|\xi_1|^{\f12-\zeta}}\|_{L^2_{\xi_1}}\| u_{2, 0}\|_{L^1}\non\\
\lesssim&\e^{\f12-\zeta} e^{\e t}\| u_{2, 0}\|_{L^1} \to 0, \quad\text{for}\quad \e \to 0, \non
\end{align*}
and
\begin{align*}
\|\textbf{1}_{|\xi_1|\leq\delta}L_2\|_{L^2_{\xi_1}L^2_{x_2}}
\lesssim&\e^{-1}e^{\e t} \Big\|\textbf{1}_{|\xi_1|\leq\delta} |\xi_1|\int_{\Ga_5 }  \Big| \f{1}{ \om-|\xi_1|} \Big|  \|e^{-|\xi_1|x_2}-e^{-\om x_2}\|_{L^2_{x_2}}d\lam\Big\|_{L^2_{\xi_1}}\| b_{2, 0}\|_{L^1}\non\\
\lesssim&\e^{-1}e^{\e t} \Big\|\textbf{1}_{|\xi_1|\leq\delta} |\xi_1|\int_{\Ga_5 }  \f{\e^\f12}{|\xi_1|}\Big(\f{1}{|\xi_1|^\f12}+\f{1}{(Re\om)^\f12}\Big)  d\lam\Big\|_{L^2_{\xi_1}}\| b_{2, 0}\|_{L^1}\non\\
\lesssim&\e^{-1}e^{\e t} \Big\|\textbf{1}_{|\xi_1|\leq\delta} |\xi_1| \f{\e^\f12}{|\xi_1|}\Big(\f{\e}{|\xi_1|^\f12}+\f{\e^\f54}{|\xi_1|^\f12}\Big)  \Big\|_{L^2_{\xi_1}}\| b_{2, 0}\|_{L^1}\non\\
\lesssim&\e^{\f12-\zeta} e^{\e t}\|\textbf{1}_{|\xi_1|\leq\delta}\e^{\zeta}\f{1}{|\xi_1|^\f12}\|_{L^2_{\xi_1}}\| b_{2, 0}\|_{L^1}\non\\
\lesssim&\e^{\f12-\zeta} e^{\e t}\|\textbf{1}_{|\xi_1|\leq\delta}\f{1}{|\xi_1|^{\f12-\zeta}}\|_{L^2_{\xi_1}}\| b_{2, 0}\|_{L^1}\non\\
\lesssim&\e^{\f12-\zeta} e^{\e t}\| b_{2, 0}\|_{L^1} \to 0, \quad\text{for}\quad \e \to 0. \non
\end{align*}

\smallskip

{\bf Case 2.  $|\xi_1|> 2$.}

We first  divide $J_1$ and $J_2$ as
\begin{align*}
J_1&=\f{i\xi_1}{|\xi_1|} \f{1}{2\pi i}\sum_{i=1}^{3}\int_{\wt{\Ga}_i}e^{\lam t} \int^\infty_0  e^{-\om y_2} (\wh{u}_{2, 0}+\f{i\xi_1}{\lam+|\xi_1|^2} \wh{b}_{2, 0})(y_2) dy_2 \f{1}{\om-|\xi_1|} e^{-|\xi_1|x_2} d\lam:= \sum_{i=1}^3 \wt{J}_1^i.\\
J_2&=-\f{i\xi_1}{|\xi_1|} \f{1}{2\pi i}\sum_{i=1}^{3}\int_{\wt{\Ga}_i}e^{\lam t} \int^\infty_0  e^{-\om (x_2+y_2)} (\wh{u}_{2, 0}+\f{i\xi_1}{\lam} \wh{b}_{2, 0})(y_2) dy_2 \f{1}{\om-|\xi_1|}  d\lam:= \sum_{i=1}^{3} \wt{J}_2^i.
\end{align*}

For the proof of the $J_i^{1,4,5}(i=1,2)$ on Case 1, since we did not use the property $|\xi_1|\leq 2$, this conclusion is also true for $\wt{J}_i^{2,3}(i=1,2)$ on Case 2, and the most trouble kernel  is
$\varphi(\xi_2)\f{|\xi_1|}{|\xi^{'}|^2}e^{-|\xi_1|^2t-\f{|\xi_1|^2}{|\xi^{'}|^2}t}$.\\

For the $\wh{u}_{2,0}$ part of $\wt{J}_2^1$, by the definition of $\wt{\Ga}_1: \lam= -\eta-|\xi_1|^2 \ (\eta: \lam'_- \to \lam'_+)$, we have
\begin{align}\label{007}
&-\f{1}{2\pi i} \int_{\wt{\Ga}_1 } \f{ e^{\lam t} }{\om-|\xi_1|} \int^\infty_0 e^{-\om(x_2+y_2)}  \wh{u}_{2, 0} (y_2) dy_2 d\lam\non\\
&=-\f{1}{2\pi i} \Big(\int^{\lam'_+}_{\lam'_-}\f{ e^{(-\eta-|\xi_1|^2) t} }{-i|\om|-|\xi_1|} \int^\infty_0 e^{i|\om|(x_2+y_2)}   \wh{u}_{2, 0} (y_2)  dy_2 d\eta\non\\
&\qquad-\int^{\lam'_+}_{\lam'_-}\f{ e^{(-\eta-|\xi_1|^2) t} }{i|\om|-|\xi_1|} \int^\infty_0 e^{-i|\om|(x_2+y_2)}   \wh{u}_{2, 0} (y_2) dy_2 d\eta\Big),
\end{align}
with
\begin{align*}
\om=\sqrt{\lam+\f{|\xi_1|^2}{\lam+|\xi_1|^2}}=\mp i \sqrt{\eta+|\xi_1|^2+\f{|\xi_1|^2}{\eta}}=\mp i|\om|.
\end{align*}

By changing of variables $\xi_2=|\om|$, we have
\begin{align}\label{lamda+}
& \eta+|\xi_1|^2+\f{|\xi_1|^2}{\eta} =|\xi_2|^2
\Rightarrow\quad \eta_\pm=\f{|\xi_2|^2-|\xi_1|^2}{2} \pm \f{\sqrt{(|\xi_1|^2-|\xi_2|^2)^2-4|\xi_1|^2}}{2}=-\lam_\mp-|\xi_1|^2,
\end{align}
 the least decaying term is the case $\eta=\eta_{-}$. Define $|\xi^{''}|^2=|\xi_1|^2-|\xi_2|^2\geq 2|\xi_1|$, we have
\begin{align*}
\eta_-
&=-\f{|\xi^{''}|^2}{2} (1+ \f{\sqrt{|\xi^{''}|^4-4|\xi_1|^2}}{|\xi^{''}|^2}),\\
d\eta_-&=\xi_2\Big(1+\f{|\xi^{''}|^2}{\sqrt{|\xi^{''}|^4-4|\xi_1|^2}}\Big) d\xi_2 =2 \xi_2 \f{\lam_++|\xi_1|^2}{\lam_+-\lam_-} d\xi_2.
\end{align*}

We consider the following cases:
\begin{itemize}
\item
 $\f{4|\xi_1|^2}{|\xi^{''}|^4}\leq \f12$,
 we obtain $  -|\xi^{''}|^2+\f{|\xi_1|^2}{|\xi^{''}|^2}\leq\eta_{-} \leq -|\xi^{''}|^2+\f{\sqrt{2}|\xi_1|^2}{|\xi^{''}|^2}$.\\
 \item
 $\f12 \leq\f{4|\xi_1|^2}{|\xi^{''}|^4}\leq1$, we have $\f{|\xi^{''}|^2}{2}\sqrt{1-\f{4|\xi_1|^2}{|\xi^{''}|^4}}\leq\f{|\xi^{''}|^2}{2\sqrt{2}}$. It then follows that
 \begin{align*}
\eta_-&=-\f{|\xi^{''}|^2}{2}-\f{|\xi^{''}|^2}{2}\sqrt{1-\f{4|\xi_1|^2}{|\xi^{''}|^4}}> -\f{(2+\sqrt{2})|\xi^{''}|^2}{4}.\\
-\eta_--|\xi_1|^2&\leq -\f{|\xi_1|^2+|\xi_2|^2}{2}+\f{\sqrt{2}(|\xi_1|^2-|\xi_2|^2)}{4}\leq -c_1(|\xi_1|^2+|\xi_2|^2).
  \end{align*}
\end{itemize}

The most trouble decay is  the case $\f{4|\xi_1|^2}{|\xi^{''}|^4}\leq \f12$,  then we have
\beno
\eta_{-}+|\xi^{''}|^2\sim \f{|\xi_1|^2}{|\xi^{''}|^2},
\eeno

 For this case,
\begin{align*}
&-\eta_--|\xi_1|^2= -|\xi_2|^2-c_2\f{|\xi_1|^2}{|\xi^{''}|^2},\\
&d \eta_{-}=2 \xi_2 \f{\lam_++|\xi_1|^2}{\lam_+-\lam_-} d\xi_2 =  -2\xi_2-2c_2\xi_2\f{|\xi_1|^2}{|\xi^{''}|^4}  d\xi_2,\\
&-2\xi_2-2c_2\xi_2\f{|\xi_1|^2}{|\xi^{''}|^4}\in\big(-2\xi_2-\f{c_2}{4}\xi_2,-2\xi_2\big),  \ \text{here}\  c_2\in(1,\sqrt{2}),
\end{align*}
 and by \eqref{lamda+}, the first term on the right hand side of \eqref{007} can be rewritten as
\begin{align*}
&-\f{1}{2\pi i} \int^{\wt{d^{'}}}_{0} \f{e^{\lam_+ t} }{i\xi_2+|\xi_1|} \int^\infty_0 e^{i\xi_2(x_2+y_2)}   \wh{u}_{2, 0} (y_2)  dy_2 2 \xi_2 \f{\lam_++|\xi_1|^2}{\lam_+-\lam_-} d\xi_2\non\\
=&-2i\cF_{x_2}^{-1}\Big(\phi(\xi_2)  \f{\lam_++|\xi_1|^2}{\lam_+-\lam_-}e^{\lam_+ t}\f{\xi_2 }{i\xi_2+|\xi_1|}\cF_{y_2}\big(\chi  \wh{u}_{2, 0} \big) (-\xi_2) \Big)(x_2),
\end{align*}
where   $\phi$ are   cut-off functions defined as
\begin{align}\label{vphi,chi}
\phi(\xi_2)=\left\{
\begin{array}{l}
1, \  \text{for} \ \xi_2\in [0,\wt{d^{'}}], \\
0, \  \text{for} \ \xi_2\notin [0,\wt{d^{'}}],
\end{array}\right.
\end{align}
and $\wt{d^{'}}=\sqrt{|\xi_1|^2-2|\xi_1|}$.

By  Plancherel theorem, we have the $L^2$ estimate
\begin{align*}
&\Big\|\textbf{1}_{|\xi_1|> 2}\cF_{x_2}^{-1}\Big(\phi(\xi_2)   e^{(-|\xi_2|^2-\f{|\xi_1|^2}{|\xi^{''}|^2}) t}\f{\xi_2 }{i\xi_2+|\xi_1|}\cF_{y_2}\big(\chi  \wh{u}_{2, 0} \big) (-\xi_2)  \Big)(x_2)\Big\|_{L^2_{\xi_1}L^2_{x_2}}\\
&\lesssim\Big\|\textbf{1}_{|\xi_1|> 2}\phi(\xi_2)  e^{(-|\xi_2|^2-\f{|\xi_1|^2}{|\xi^{''}|^2}) t}\f{\xi_2 }{i\xi_2+|\xi_1|}\cF\big(\chi \wh{u}_{2, 0} \big) (-\xi_2)   \Big\|_{L^2_{\xi_1}L^2_{\xi_2}}\non\\
&\lesssim e^{-t}    \| u_{2, 0} \|_{L^2}.
\end{align*}

Following \eqref{009}, we have the $L^2$ estimate for $\wt{J}_1^1$:
\begin{align*}
&\Big\|\textbf{1}_{|\xi_1|> 2}\int_{\mathbb{R}}\phi(\xi_2)   e^{(-|\xi_2|^2-\f{|\xi_1|^2}{|\xi^{''}|^2}) t}\f{\xi_2 }{i\xi_2+|\xi_1|}|\xi_1|^{-\f12}\cF_{y_2}\big(\chi  \wh{u}_{2, 0} \big) (-\xi_2)  d\xi_2\Big\|_{L^2_{\xi_1}}\\
&=\Big\|\textbf{1}_{|\xi_1|> 2}\int_{\mathbb{R}}\phi(\xi_2)   e^{(-|\xi_2|^2-\f{|\xi_1|^2}{|\xi^{''}|^2}) t}\f{|\xi_1|^\f12 }{i\xi_2+|\xi_1|}\cF_{y_2}\big(\chi  \wh{u}_{1, 0} \big) (-\xi_2)  d\xi_2\Big\|_{L^2_{\xi_1}}\\
&\lesssim\Big\|\textbf{1}_{|\xi_1|> 2}\phi(\xi_2)  e^{(-|\xi_2|^2-\f{|\xi_1|^2}{|\xi^{''}|^2}) t}\f{|\xi_1|^\f12 }{(|\xi_1|^2+|\xi_2|^2)^\f12}\cF\big(\chi \wh{u}_{1, 0} \big) (-\xi_2)   \Big\|_{L^2_{\xi_1}L^1_{\xi_2}}\non\\
&\lesssim e^{-t}\Big\|\f{|\xi_1|^\f12 }{(|\xi_1|^2+|\xi_2|^2)^\f12}\Big\|_{L^\infty_{\xi_1}L^2_{\xi_2}}\Big\|\cF\big(\chi \wh{u}_{1, 0} \big) (-\xi_2)   \Big\|_{L^2_{\xi_1}L^2_{\xi_2}}\non\\
&\lesssim e^{-t}    \| u_{1, 0} \|_{L^2}.
\end{align*}

The $\wh{b}_{2,0}$ part can be estimated similarly since
\begin{align}\label{013}
 \f{|\xi_1|}{\big|\lam_++|\xi_1|^2\big|}=\f{|\xi_1|}{|\xi^{''}|^2-\f{|\xi_1|^2}{|\xi^{''}|^2}} \lesssim 1.
\end{align}

Thus we complete the proof of the Proposition 2.1.
\end{proof}

\smallskip

Now we  give the  resolvent estimate of magnetic  field in  \eqref{linear}.

\begin{proposition}\label{linear b}
We have the following $L^2$ estimates
 \begin{align*}
&\|b_{1,L}\|_{L^2}\lesssim \lan t \ran ^{-\f14}\|(u_{0},b_{0})\|_{L^1\cap L^2\cap L^1_{x_1}L^2_{x_2}},\\
&\|b_{2,L}\|_{L^2}\lesssim \lan t \ran ^{-\f12}\|(u_{0},b_{0})\|_{L^1\cap L^2\cap L^1_{x_1}L^2_{x_2}},
 \end{align*}
 for the magnetic  field in \eqref{linear}.
\end{proposition}

 \begin{proof}by \eqref{0006} and \eqref{linear}, $\wh{b}_{1,L}$ can be rewritten as
 \begin{align}
\wh{b}_{1,L}=&K_1+K_2\non\\
=&\f{1}{2\pi i}\int_{\Ga}e^{\lam t}
\Big\{ \f{i\xi_1}{\lam+|\xi_1|^2}\big(N_\om[\wh{u}_{1, 0}]+\f{i\xi_1}{\lam+|\xi_1|^2}N_\om[\wh{b}_{1, 0}]\big)+\f{\wh{b}_{1,0}}{\lam+|\xi_1|^2} \Big\}d\lam\non\\
&+\f{1}{2\pi i}\int_{\Ga}e^{\lam t}
\Big\{\f{i\xi_1}{\lam+|\xi_1|^2}\f{i\xi_1}{|\xi_1|}\f{N_\om[e^{-|\xi_1|x_2}]}{E_{\om}[e^{-|\xi_1|x_2}]_0}\{ E_\om[\wh{u}_{2, 0}]_0+\f{i\xi_1}{\lam+|\xi_1|^2}E_\om[\wh{b}_{2, 0}]_0\}\big) \Big\}d\lam\non\\
=&\f{1}{2\pi i}\int_{\Ga}e^{\lam t}
\Big\{ \f{i\xi_1}{\lam+|\xi_1|^2}\big(N_\om[\wh{u}_{1, 0}]+\f{i\xi_1}{\lam+|\xi_1|^2}N_\om[\wh{b}_{1, 0}]\big)+\f{\wh{b}_{1,0}}{\lam+|\xi_1|^2} \Big\}d\lam\non\\
&+\f{1}{2\pi i}\int_{\Ga}e^{\lam t}
\Big\{\f{i\xi_1}{\lam+|\xi_1|^2}\f{2i\xi_1 \om}{|\xi_1|(\om-|\xi_1|)}\{  E_\om[\wh{u}_{2, 0}]_0+\f{i\xi_1}{\lam+|\xi_1|^2}E_\om[\wh{b}_{2, 0}]_0\} \Big(e^{-|\xi_1|x_2}- e^{-\om x_2}\Big)\Big\}d\lam.\label{k2}
\end{align}
$\wh{b}_{2,L}$ can be solved as $\wh{b}_{1,L}$, so we only need to consider $\wh{b}_{1,L}$.\\

For the $K_1$ part, consider the odd extension, $K_1$ can be rewritten as
\begin{align*}
K_1=\f{1}{2\pi i}\int_{\Ga}\f{i\xi_1e^{\lam t}}{\lam+|\xi_1|^2} \f{1}{2\om} e^{-\om|x_2|} * \wh{U}_{1, 0}d\lam+\f{1}{2\pi i}\int_{\Ga}\f{e^{\lam t}\wh{b}_{1,0}}{\lam+|\xi_1|^2}  d\lam:=\f{1}{2\pi i}\int_{\Ga}e^{\lam t} K_{1, \lam}d\lam +e^{-|\xi_1|^2t}\wh{b}_{1,0}.
\end{align*}
Make the odd extension of the $ K_{1, \lam}$ and $\wh{b}_{1,0}$, then
\begin{align*}
\cF_{x_2}(\wt{K}_1)&=\f{1}{2\pi i}\int_{\Ga}e^{\lam t}\cF_{x_2} (\wt{K}_{1, \lam})d\lam+e^{-|\xi_1|^2t}\cF_{x_2}(\wh{b}^{o}_{1,0})\non\\
&=\f{1}{2\pi i}\int_{\Ga}\f{i\xi_1e^{\lam t}}{(\lam-\lam_+)(\lam-\lam_-)}\cF_{x_2}(\wh{U}_{1, 0}) d\lam+e^{-|\xi_1|^2t}\cF_{x_2}(\wh{b}^{o}_{1,0})\non\\
&=\f{1}{2\pi i}\int_{\Ga}\f{i\xi_1e^{\lam t}}{(\lam-\lam_+)(\lam-\lam_-)}\cF_{x_2}(\wh{u}^{o}_{1, 0})+\f{(i\xi_1)^2e^{\lam t}}{(\lam-\lam_+)(\lam-\lam_-)(\lam+|\xi_1|^2)}\cF_{x_2}(\wh{b}^{o}_{1, 0})d\lam\\
&\quad+e^{-|\xi_1|^2t}\cF_{x_2}(\wh{b}^{o}_{1,0}) \non\\
&=i\xi_1\f{ e^{\lam_+t} -e^{\lam_- t}}{\lam_+-\lam_-}\cF_{x_2}(\wh{u}^{o}_{1, 0})-\f{(\lam_-+|\xi_1|^2) e^{\lam_+t} -(\lam_++|\xi_1|^2)e^{\lam_- t}}{\lam_+-\lam_-}\cF_{x_2}(\wh{b}^{o}_{1, 0})\\
&\quad+(i\xi_1)^2\f{ e^{-|\xi_1|^2t}}{(|\xi_1|^2+\lam_+)(|\xi_1|^2+\lam_-)}\cF_{x_2}(\wh{b}^{o}_{1, 0})+e^{-|\xi_1|^2 t}\cF_{x_2}(\wh{b}^{o}_{1, 0})\\
&=i\xi_1\f{ e^{\lam_+t} -e^{\lam_- t}}{\lam_+-\lam_-}\cF_{x_2}(\wh{u}^{o}_{1, 0})-\f{(\lam_-+|\xi_1|^2) e^{\lam_+t} -(\lam_++|\xi_1|^2)e^{\lam_- t}}{\lam_+-\lam_-}\cF_{x_2}(\wh{b}^{o}_{1, 0}).
\end{align*}

In fact, we have discussed   $i\xi_1\f{ e^{\lam_+t} -e^{\lam_- t}}{\lam_+-\lam_-}$ and $|\xi_1|^2\f{ e^{\lam_+t} -e^{\lam_- t}}{\lam_+-\lam_-}$ thoroughly in $M_2$ and $M_3$ of $u_L$, thus we only need to consider
\begin{align*}
M_4:=\Big|\f{\lam_- e^{\lam_+t} -\lam_+e^{\lam_- t}}{\lam_+-\lam_-}\Big|.
\end{align*}

\textbf{Case a}: $\big||\xi_1|^2-|\xi_2|^2\big|\leq |\xi_1|$.\\

We have
\begin{align}\label{M4a}
M_4&=\Big|\f{\lam_- e^{\lam_+ t}-\lam_+ e^{\lam_- t}}{\lam_+-\lam_-}\Big|\non\\
&\quad=e^{-\f{|\xi_1|^2+|\xi_2|^2}{2}t}\Big|\f{|\xi_1|^2+|\xi_2|^2}{2}t\f{sin(\f{ \sqrt{4|\xi_1|^2-(|\xi_1|^2-|\xi_2|^2)^2}}{2}t)}{\f{ \sqrt{4|\xi_1|^2-(|\xi_1|^2-|\xi_2|^2)^2}}{2}t}+cos(\f{ \sqrt{4|\xi_1|^2-(|\xi_1|^2-|\xi_2|^2)^2}}{2}t)\Big|\non\\
&\lesssim e^{-\f{|\xi_1|^2+|\xi_2|^2}{2}t}.
\end{align}

\textbf{Case b}: $|\xi_1|<\big||\xi_1|^2-|\xi_2|^2\big|\leq 2|\xi_1|$.\\

We can use the same way in \eqref{M4a} to deal with $M_4$:
\begin{align*}
M_4&=\Big|\f{\lam_- e^{\lam_+ t}-\lam_+ e^{\lam_- t}}{\lam_+-\lam_-}\Big|\lesssim e^{-\f{|\xi_1|^2+|\xi_2|^2}{2}t}.
\end{align*}

\textbf{Case c}: $2|\xi_1|<\big||\xi_1|^2-|\xi_2|^2\big|\leq 4|\xi_1|$.\\

For $ \big||\xi_1|^2-|\xi_2|^2\big|> 2|\xi_1|$, we have $\lam_-< \lam_+<0$, we only need to consider the trouble decay item $\lam_+$.
We have
\begin{align}\label{bc}
&\lam_+ =-\f{|\xi_1|^2+|\xi_2|^2}{2}+\f{\sqrt{(|\xi_1|^2-|\xi_2|^2)^2-4|\xi_1|^2}}{2}\non\\
 & \qquad \leq-\f{|\xi_1|^2+|\xi_2|^2}{2}+\f{\sqrt{15}}{4}\f{\big||\xi_1|^2-|\xi_2|^2\big|}{2}\leq-c_0(|\xi_1|^2+|\xi_2|^2),\non\\
 &\lam_+=-\f{|\xi_1|^2+|\xi_2|^2}{2}+\f{\sqrt{(|\xi_1|^2-|\xi_2|^2)^2-4|\xi_1|^2}}{2}\geq -\f{|\xi_1|^2+|\xi_2|^2}{2}.
 \end{align}
 By \eqref{207} and \eqref{bc}, we have
 \begin{align*}
&\Big|\f{ e^{\lam_+ t}-e^{\lam_- t}}{\lam_+-\lam_-}\Big|\lesssim t e^{\lam_+ t} \lesssim \f{1}{|\xi_1|^2+|\xi_2|^2} e^{-c_0(|\xi_1|^2+|\xi_2|^2) t},\\
&M_4=\Big|\f{ \lam_-e^{\lam_+ t}-\lam_+e^{\lam_- t}}{\lam_+-\lam_-}\Big|=\Big|-e^{\lam_+ t}+\lam_+\f{ e^{\lam_+ t}-e^{\lam_- t}}{\lam_+-\lam_-}\Big|\lesssim   e^{-c_0(|\xi_1|^2+|\xi_2|^2) t}.
\end{align*}

\textbf{Case d}: $\big||\xi_1|^2-|\xi_2|^2\big|> 4|\xi_1|$.\\

 We only need to consider the trouble decay item $\lam_+$, and consider the following two cases:\\

\textbf{Case d.1}: $|\xi_1|^2-|\xi_2|^2> 4|\xi_1|$.\\

For $ |\xi_1|^2-|\xi_2|^2> 4|\xi_1|$, we have $|\xi_1|>4$, and
\begin{align*}
\lam_+ &=-\f{|\xi_1|^2+|\xi_2|^2}{2}+\f{\sqrt{(|\xi_1|^2-|\xi_2|^2)^2-4|\xi_1|^2}}{2}=\f{-|\xi_1|^2-|\xi_1|^2|\xi_2|^2}{\f{|\xi_1|^2+|\xi_2|^2}{2}+\f{\sqrt{(|\xi_1|^2-|\xi_2|^2)^2-4|\xi_1|^2}}{2}}\\
&\leq-\f{|\xi_1|^2+|\xi_1|^2|\xi_2|^2}{|\xi_1|^2}\leq-1,
\end{align*}
it has an exponential decay.\\

\textbf{Case d.2}: $|\xi_2|^2-|\xi_1|^2> 4|\xi_1|$.\\

For $ |\xi_2|^2-|\xi_1|^2> 4|\xi_1|$, we have
\begin{align*}
&\lam_+ =-\f{|\xi_1|^2+|\xi_2|^2}{2}+\f{\sqrt{(|\xi_1|^2-|\xi_2|^2)^2-4|\xi_1|^2}}{2}=\f{-|\xi_1|^2-|\xi_1|^2|\xi_2|^2}{\f{|\xi_1|^2+|\xi_2|^2}{2}+\f{\sqrt{(|\xi_1|^2-|\xi_2|^2)^2-4|\xi_1|^2}}{2}}\\
&\qquad \leq-\f{|\xi_1|^2+|\xi_1|^2|\xi_2|^2}{|\xi_2|^2},\\
&M_4=\Big|\f{ \lam_-e^{\lam_+ t}-\lam_+e^{\lam_- t}}{\lam_+-\lam_-}\Big|\lesssim  |e^{\lam_+t}|\lesssim e^{-\f{|\xi_1|^2}{|\xi_2|^2}t-|\xi_1|^2t}.
\end{align*}

Then the $M_4$ means the most trouble kernel $\textbf{1}_{|\xi_2|^2>|\xi_1|^2+4|\xi_1|}e^{-\f{|\xi_1|^2}{|\xi_2|^2}t-|\xi_1|^2t}$.\\

For $t>1$, we have the $L^2$ estimates,

\begin{align*}
&\|e^{-\f{|\xi_1|^2}{|\xi_2|^2}t-|\xi_1|^2t}\cF(b_{1,0})\|_{L^2_{\xi_1}L^2_{\xi_2}}
\lesssim\|e^{-|\xi_1|^2t}\cF(b_{1,0})\|_{L^2_{\xi_1}L^2_{\xi_2}}
\lesssim t^{-\f14}\|b_{1,0}\|_{L^1_{x_1}L^2_{x_2}},\\
&\|e^{-\f{|\xi_1|^2}{|\xi_2|^2}t-|\xi_1|^2t}\cF(b_{2,0})\|_{L^2_{\xi_1}L^2_{\xi_2}}
\lesssim\|\f{|\xi_1|}{|\xi_2|}e^{-\f{|\xi_1|^2}{|\xi_2|^2}t}e^{-|\xi_1|^2t}\cF(b_{1,0})\|_{L^2_{\xi_1}L^2_{\xi_2}}
\lesssim t^{-\f34}\|b_{1,0}\|_{L^1_{x_1}L^2_{x_2}}.
\end{align*}

 For $0<t<1$, we have the $L^2$ estimates,
 \begin{align*}
&\|e^{-\f{|\xi_1|^2}{|\xi_2|^2}t-|\xi_1|^2t}\cF(b_{0})\|_{L^2_{\xi_1}L^2_{\xi_2}}\lesssim\|\cF(b_{0})\|_{L^2_{\xi_1}L^2_{\xi_2}}\lesssim \|b_{0}\|_{L^2}.
\end{align*}

Therefore, combined with the conclusions we have obtained about $M_2$ and $M_3$, we prove the following results for $K_1$,
\begin{align*}
&\|b_{1,L}\|_{L^2}\lesssim \lan t \ran ^{-\f12}\|(u_{1,0},b_{1,0})\|_{L^1\cap L^2}+\lan t \ran ^{-\f14}\|b_{1,0}\|_{L^1_{x_1}L^2_{x_2}},\\
&\|b_{2,L}\|_{L^2}\lesssim \lan t \ran ^{-\f12}\|(u_{2,0},b_{2,0})\|_{L^1\cap L^2}+\lan t \ran ^{-\f34}\|b_{1,0}\|_{L^1_{x_1}L^2_{x_2}}.
\end{align*}

Given the fact that $K_2$ in \eqref{k2} and $I_2$ in \eqref{u1} are only different by $\f{i\xi_1}{\lam+|\xi_1|^2}$, we only need to analyze the effect of $\f{i\xi_1}{\lam+|\xi_1|^2}$. Through the analysis of $\wh{u}_L$, we know $\lam=\lam_+$ is the least decay item, so we will only consider the case of $\lam=\lam_+$ in the following paragraphs. \\

 First we consider   $|\xi_1|\leq 2$.  For $\Gamma_1$ and $\Gamma_4$, we have $|\f{i\xi_1}{\lam_+ +|\xi_1|^2}|=\f{|\xi^{'}|^2}{|\xi_1|}$. In fact, we only need to consider the most trouble item $\big(\f{i\xi_1}{\lam_+ +|\xi_1|^2}\big)^2\wh{b}_{2,0}$. By \eqref{010}, \eqref{009} and \eqref{2000}, we have the $L^2$ estimates

\begin{align*}
&\Big\|\textbf{1}_{|\xi_1|\leq 2}\cF_{x_2}^{-1}\Big(\vphi(\xi_2) \f{|\xi_1|^2}{|\xi^{'}|^4} e^{(-\f{|\xi_1|^2}{|\xi^{'}|^2}-|\xi_1|^2) t}  \f{\xi_2 }{i\xi_2+|\xi_1|}\f{|\xi^{'}|^4}{|\xi_1|^2}\cF_{y_2}\big(\chi  \wh{b}_{2, 0} \big) (-\xi_2) \Big)(x_2)\Big\|_{L^2_{\xi_1}L^2_{x_2}}\non\\
=&\Big\|\textbf{1}_{|\xi_1|\leq 2}\cF_{x_2}^{-1}\Big(\vphi(\xi_2) \f{|\xi_1|^2}{|\xi^{'}|^4} e^{(-\f{|\xi_1|^2}{|\xi^{'}|^2}-|\xi_1|^2) t}  \f{\xi_2 }{i\xi_2+|\xi_1|}\f{|\xi^{'}|^4}{|\xi_1|^2}\f{|\xi_1|}{|\xi_2|}\cF_{y_2}\big(\chi  \wh{b}_{1, 0} \big) (-\xi_2) \Big)(x_2)\Big\|_{L^2_{\xi_1}L^2_{x_2}}\non\\
\lesssim&\Big\|\vphi(\xi_2) \f{|\xi_1|}{|\xi_2|} e^{(-\f{|\xi_1|^2}{|\xi^{'}|^2}-|\xi_1|^2) t}  \cF_{y_2}\big(\chi  \wh{b}_{1, 0} \big) (-\xi_2) \Big\|_{L^2_{\xi_1}L^2_{\xi_2}}\non\\
\lesssim&\Big\| \f{|\xi_1|}{|\xi_2|} e^{(-\f{|\xi_1|^2}{|\xi_2|^2}-|\xi_1|^2) t}  \cF_{y_2}\big(\chi  \wh{b}_{1, 0} \big) (-\xi_2) \Big\|_{L^2_{\xi_1}L^2_{\xi_2}}\non\\
\lesssim&t^{-\f12}\Big\|e^{-|\xi_1|^2 t}    \Big\|_{L^2_{\xi_1}}\|\chi b_{1, 0} \|_{L^1_{x_1}L^2_{x_2}}\non\\
\lesssim& t^{-\f34}    \| b_{1, 0} \|_{L^1_{x_1}L^2_{x_2}},
\end{align*}
and
\begin{align}\label{b1l1l2}
&\Big\|\textbf{1}_{|\xi_1|\leq 2}\int_{\R}\vphi(\xi_2)\f{|\xi_1|^2}{|\xi^{'}|^4} e^{(-\f{|\xi_1|^2}{|\xi^{'}|^2}-|\xi_1|^2) t} \f{\xi_2 }{i\xi_2+|\xi_1|}|\xi_1|^{-\f12}\f{|\xi^{'}|^4}{|\xi_1|^2}\cF_{y_2}\big(\chi  \wh{b}_{2, 0} \big) (-\xi_2)d\xi_2   \Big\|_{L^2_{\xi_1}}\non\\
=&\Big\|\textbf{1}_{|\xi_1|\leq 2}\int_{\R}\vphi(\xi_2)\f{|\xi_1|^2}{|\xi^{'}|^4} e^{(-\f{|\xi_1|^2}{|\xi^{'}|^2}-|\xi_1|^2) t} \f{\xi_2 }{i\xi_2+|\xi_1|}|\xi_1|^{-\f12}\f{|\xi^{'}|^4}{|\xi_1|^2}\f{|\xi_1|}{|\xi_2|}\cF_{y_2}\big(\chi  \wh{b}_{1, 0} \big) (-\xi_2)d\xi_2   \Big\|_{L^2_{\xi_1}}\non\\
\lesssim&\Big\|\vphi(\xi_2) e^{(-\f{|\xi_1|^2}{|\xi^{'}|^2}-|\xi_1|^2) t} \f{|\xi_2|}{(|\xi_1|^2+|\xi_2|^2)^\f12}\f{|\xi_1|^\f12}{|\xi_2|}\cF_{y_2}\big(\chi  \wh{b}_{1, 0} \big) (-\xi_2)   \Big\|_{L^2_{\xi_1}L^1_{\xi_2}}\non\\
\lesssim&\Big\|\vphi(\xi_2) e^{(-\f{|\xi_1|^2}{|\xi_2|^2}-|\xi_1|^2) t} \f{|\xi_1|^\f12}{(|\xi_1|^2+|\xi_2|^2)^\f12}\cF_{y_2}\big(\chi  \wh{b}_{1, 0} \big) (-\xi_2)   \Big\|_{L^2_{\xi_1}L^1_{\xi_2}}\non\\
\lesssim&\Big\|\textbf{1}_{|\xi_2|\geq1}\f{1}{|\xi_2|^{\f12+\delta}}\Big\|_{L^2_{\xi_2}}\Big\|\f{|\xi_1|^{\f12}|\xi_2|^{\f12+\delta}}{(|\xi_1|^2+|\xi_2|^2)^\f12}e^{-\f{|\xi_1|^2}{|\xi_2|^2}t-|\xi_1|^2t}\cF_{y_2}\big(\chi  \wh{b}_{1, 0} \big) (-\xi_2)\Big\|_{L^2_{\xi_1}L^2_{\xi_2}}\non\\
&+\Big\|\textbf{1}_{|\xi_2|<1}\varphi(\xi_2)\f{1}{|\xi_2|^{\f12-\delta}}\Big\|_{L^2_{\xi_2}}\Big\|\f{|\xi_1|^{\f12}|\xi_2|^{\f12-\delta}}{(|\xi_1|^2+|\xi_2|^2)^\f12}e^{-\f{|\xi_1|^2}{|\xi_2|^2}t-|\xi_1|^2t}\cF_{y_2}\big(\chi  \wh{b}_{1, 0} \big) (-\xi_2)\Big\|_{L^2_{\xi_1}L^2_{\xi_2}}\non\\
\lesssim&\Big\|\f{|\xi_1|^{\f12}}{|\xi_2|^{\f12-\delta}}e^{-\f{|\xi_1|^2}{|\xi_2|^2}t-|\xi_1|^2t}\cF_{y_2}\big(\chi  \wh{b}_{1, 0} \big) (-\xi_2)\Big\|_{L^2_{\xi_1}L^2_{\xi_2}}+\Big\|\f{|\xi_1|^{\f12}}{|\xi_2|^{\f12+\delta}}e^{-\f{|\xi_1|^2}{|\xi_2|^2}t-|\xi_1|^2t}\cF_{y_2}\big(\chi  \wh{b}_{1, 0} \big) (-\xi_2)\Big\|_{L^2_{\xi_1}L^2_{\xi_2}}\non\\
\lesssim&t^{-(\f14-\f{\delta}{2})}\Big\||\xi_1|^\delta e^{-|\xi_1|^2 t}    \Big\|_{L^2_{\xi_1}}\Big\|\cF_{y_2}\big(\chi  \wh{b}_{1, 0} \big) (-\xi_2)\Big\|_{L^\infty_{\xi_1}L^2_{\xi_2}}+t^{-(\f14+\f{\delta}{2})}\Big\| |\xi_1|^{-\delta}e^{-|\xi_1|^2 t}    \Big\|_{L^2_{\xi_1}}\Big\|\cF_{y_2}\big(\chi  \wh{b}_{1, 0} \big) (-\xi_2)\Big\|_{L^\infty_{\xi_1}L^2_{\xi_2}}\non\\
\lesssim& t^{-\f12}    \| b_{1, 0} \|_{L^1_{x_1}L^2_{x_2}},
\end{align}
here we use
\begin{align*}
\Big\| |\xi_1|^{-\delta}e^{-|\xi_1|^2 t}    \Big\|_{L^2_{\xi_1}}=&\Big(\int_0^{\infty}|\xi_1|^{-2\delta}e^{-|\xi_1|^2 t}d\xi_1\Big)^\f12=t^{-\f14+\f{\delta}{2}}\Big(\int_0^{\infty}s^{-2\delta}e^{-s^2}ds\Big)^{\f12}\\
&=t^{-\f14+\f{\delta}{2}}\Big(\f{\Gamma({\f12-\delta})}{2}\Big)^{\f12}\lesssim t^{-\f14+\f{\delta}{2}}, \ (s=|\xi_1|t^\f12)
\end{align*}
and the boundedness of the gamma function $\Gamma(x)$ when $0<x<1$.

When $0<t<1$, we have
\begin{align*}
 \Big\|\vphi(\xi_2)\f{|\xi_1|}{|\xi_2|}  \cF_{y_2}\big(\chi  \wh{b}_{1, 0} \big) (-\xi_2)   \Big\|_{L^2_{\xi_1}L^2_{\xi_2}}\lesssim \|b_{1, 0}\|_{L^2},
\end{align*}
and
\begin{align*}
 \Big\|\vphi(\xi_2)\f{|\xi_1|^\f12}{|\xi_2|}  \cF_{y_2}\big(\chi  \wh{b}_{1, 0} \big) (-\xi_2)   \Big\|_{L^2_{\xi_1}L^2_{\xi_2}}\lesssim \|b_{1, 0}\|_{L^2}.
\end{align*}

For $\Gamma_2$, $\Gamma_3$, by \eqref{011}, we have solved it.\\

For $\Gamma_5$, since the $\wh{u}_0$ part of $\wh{b}_L$ is the same as the $\wh{b}_0$ part of $\wh{u}_L$, we only need to consider the $\wh{b}_0$ part of $\wh{b}_L$, and define it as $L_3$.

As the $\textbf{1}_{|\xi_1|>\delta}$ part, by \eqref{002}, \eqref{reom}, \eqref{ixi1}, we have
\begin{align*}
|\textbf{1}_{|\xi_1|>\delta} L_3|
=&\Big|\textbf{1}_{|\xi_1|>\delta}\f{1}{2\pi i} \int_{\Ga_5 } e^{\lam t}  \f{e^{-|\xi_1|x_2}-e^{-\om x_2}}{ \om-|\xi_1|} (\f{i\xi_1}{\lam+|\xi_1|^2})^2\int^\infty_0 e^{-\om y_2}  \wh{b}_{2,0}(y_2)dy_2 d\lam\Big|\non\\
=&\Big|\textbf{1}_{|\xi_1|>\delta}\f{1}{2\pi i} \int_{\Ga_5 } e^{\lam t}  \f{e^{-|\xi_1|x_2}-e^{-\om x_2}}{ \om-|\xi_1|} (\f{i\xi_1}{\lam+|\xi_1|^2})^2\f{1}{\om}\int^\infty_0 e^{-\om y_2}  \pa_2\wh{b}_{2,0}(y_2)dy_2 d\lam\Big|\non\\
\lesssim& \textbf{1}_{|\xi_1|>\delta}e^{\e t}\int_{\Ga_5 } \Big|  \f{e^{-|\xi_1|x_2}-e^{-\om x_2}}{ \om-|\xi_1|} (\f{i\xi_1}{\lam+|\xi_1|^2})^2\f{1}{\om}\int^\infty_0 e^{-\om y_2}  \pa_2\wh{b}_{2,0}(y_2)dy_2 \Big|d\lam\non\\
\lesssim& \textbf{1}_{|\xi_1|>\delta}e^{\e t}\int_{\Ga_5 } \Big|  \f{e^{-|\xi_1|x_2}-e^{-\om x_2}}{ \om-|\xi_1|} \Big|d\lam \sup_{\lam \in \Gamma_5}\Big|(\f{i\xi_1}{\lam+|\xi_1|^2})^2\f{1}{\om}\int^\infty_0 e^{-\om y_2}  \pa_2\wh{b}_{2,0}dy_2 \Big|\non\\
\lesssim& \e^{-2}|\xi_1|^2\textbf{1}_{|\xi_1|>\delta}e^{\e t}\int_{\Ga_5 } \Big|  \f{e^{-|\xi_1|x_2}-e^{-\om x_2}}{ \om-|\xi_1|} \Big|d\lam \f{1}{|\om|}\f{1}{|Re \om|^\f12}\|\pa_2\wh{b}_{2,0}\|_{L^2_{x_2}}\non\\
\lesssim&\e^{-\f54}|\xi_1|^\f12 \textbf{1}_{|\xi_1|>\delta}e^{\e t}\int_{\Ga_5 } \Big|  \f{e^{-|\xi_1|x_2}-e^{-\om x_2}}{ \om-|\xi_1|} \Big|d\lam\|\pa_2\wh{b}_{2, 0}\|_{L^2_{x_2}},
\end{align*}
then by \eqref{reom} and \eqref{L2-01},
\begin{align*}
\|\textbf{1}_{|\xi_1|>\delta}L_3\|_{L^2_{\xi_1}L^2_{x_2}}
\lesssim&\e^{-\f54}e^{\e t} \Big\|\textbf{1}_{|\xi_1|>\delta} |\xi_1|^\f12\int_{\Ga_5 }  \Big| \f{1}{ \om-|\xi_1|} \Big|  \|e^{-|\xi_1|x_2}-e^{-\om x_2}\|_{L^2_{x_2}}d\lam\Big\|_{L^2_{\xi_1}}\|\pa_2 b_{2, 0}\|_{L^1_{x_1}L^2_{x_2}}\non\\
\lesssim&\e^{-\f54}e^{\e t} \Big\|\textbf{1}_{|\xi_1|>\delta} \int_{\Ga_5 }|\xi_1|^\f12  \f{\e^\f12}{|\xi_1|}\Big(\f{1}{|\xi_1|^\f12}+\f{1}{(Re\om)^\f12}\Big)  d\lam\Big\|_{L^2_{\xi_1}}\|\pa_2 b_{2, 0}\|_{L^1_{x_1}L^2_{x_2}}\non\\
\lesssim&\e^{-\f54}e^{\e t} \Big\|\textbf{1}_{|\xi_1|>\delta} |\xi_1|^\f12 \f{\e^\f12}{|\xi_1|}\Big(\f{\e}{|\xi_1|^\f12}+\f{\e^\f54}{|\xi_1|^\f12}\Big)  \Big\|_{L^2_{\xi_1}}\|\pa_2 b_{2, 0}\|_{L^1_{x_1}L^2_{x_2}}\non\\
\lesssim&\e^\f14 e^{\e t}\|\textbf{1}_{|\xi_1|>\delta}\f{1}{|\xi_1|}\|_{L^2_{\xi_1}}\|\pa_2 b_{2, 0}\|_{L^1_{x_1}L^2_{x_2}}\non\\
\lesssim&\e^\f14 e^{\e t}\|\pa_2 b_{2, 0}\|_{L^1_{x_1}L^2_{x_2}} \to 0, \quad\text{for}\quad \e \to 0. \non
\end{align*}

As the $\textbf{1}_{|\xi_1|\leq\delta}$ part, by \eqref{002},  \eqref{ixi1}, we have
\begin{align*}
|\textbf{1}_{|\xi_1|\leq\delta} L_3|
=&\Big|\textbf{1}_{|\xi_1|\leq\delta}\f{1}{2\pi i} \int_{\Ga_5 } e^{\lam t}  \f{e^{-|\xi_1|x_2}-e^{-\om x_2}}{ \om-|\xi_1|} (\f{i\xi_1}{\lam+|\xi_1|^2})^2\int^\infty_0 e^{-\om y_2}  \wh{b}_{2,0}(y_2)dy_2 d\lam\Big|\non\\
=&\Big|\textbf{1}_{|\xi_1|\leq\delta}\f{1}{2\pi i} \int_{\Ga_5 } e^{\lam t}  \f{e^{-|\xi_1|x_2}-e^{-\om x_2}}{ \om-|\xi_1|} (\f{i\xi_1}{\lam+|\xi_1|^2})^2\f{1}{\om}\int^\infty_0 e^{-\om y_2}  \pa_2\wh{b}_{2,0}(y_2)dy_2 d\lam\Big|\non\\
=&\Big|\textbf{1}_{|\xi_1|\leq\delta}\f{1}{2\pi i} \int_{\Ga_5 } e^{\lam t}  \f{e^{-|\xi_1|x_2}-e^{-\om x_2}}{ \om-|\xi_1|} (\f{i\xi_1}{\lam+|\xi_1|^2})^2\f{|\xi_1|}{\om}\int^\infty_0 e^{-\om y_2}  \wh{b}_{1,0}(y_2)dy_2 d\lam\Big|\non\\
\lesssim& \textbf{1}_{|\xi_1|\leq\delta}e^{\e t}\int_{\Ga_5 } \Big|  \f{e^{-|\xi_1|x_2}-e^{-\om x_2}}{ \om-|\xi_1|} (\f{i\xi_1}{\lam+|\xi_1|^2})^2\f{|\xi_1|}{\om}\int^\infty_0 e^{-\om y_2}  \wh{b}_{1,0}(y_2)dy_2 \Big|d\lam\non\\
\lesssim& \textbf{1}_{|\xi_1|\leq\delta}e^{\e t}\int_{\Ga_5 } \Big|  \f{e^{-|\xi_1|x_2}-e^{-\om x_2}}{ \om-|\xi_1|} \Big|d\lam \sup_{\lam \in \Gamma_5}\Big|(\f{i\xi_1}{\lam+|\xi_1|^2})^2\f{|\xi_1|}{\om}\int^\infty_0 e^{-\om y_2}  \wh{b}_{1,0}dy_2 \Big|\non\\
\lesssim&\e^{-\f54}|\xi_1|^\f32 \textbf{1}_{|\xi_1|>\delta}e^{\e t}\int_{\Ga_5 } \Big|  \f{e^{-|\xi_1|x_2}-e^{-\om x_2}}{ \om-|\xi_1|} \Big|d\lam\|\wh{b}_{1, 0}\|_{L^2_{x_2}},
\end{align*}
by  \eqref{reom} and \eqref{L2-01},  we have
\begin{align*}
\|\textbf{1}_{|\xi_1|\leq\delta}L_3\|_{L^2_{\xi_1}L^2_{x_2}}
\lesssim&\e^{-\f54}e^{\e t} \Big\|\textbf{1}_{|\xi_1|\leq\delta} |\xi_1|^\f32\int_{\Ga_5 }  \Big| \f{1}{ \om-|\xi_1|} \Big|  \|e^{-|\xi_1|x_2}-e^{-\om x_2}\|_{L^2_{x_2}}d\lam\Big\|_{L^2_{\xi_1}}\| b_{1, 0}\|_{L^1_{x_1}L^2_{x_2}}\non\\
\lesssim&\e^{-\f54}e^{\e t} \Big\|\textbf{1}_{|\xi_1|\leq\delta} |\xi_1|^\f32\int_{\Ga_5 }  \f{\e^\f12}{|\xi_1|}\Big(\f{1}{|\xi_1|^\f12}+\f{1}{(Re\om)^\f12}\Big)  d\lam\Big\|_{L^2_{\xi_1}}\| b_{1, 0}\|_{L^1_{x_1}L^2_{x_2}}\non\\
\lesssim&\e^{-\f54}e^{\e t} \Big\|\textbf{1}_{|\xi_1|\leq\delta} |\xi_1|^\f32 \f{\e^\f12}{|\xi_1|}\Big(\f{\e}{|\xi_1|^\f12}+\f{\e^\f54}{|\xi_1|^\f12}\Big)  \Big\|_{L^2_{\xi_1}}\| b_{1, 0}\|_{L^1_{x_1}L^2_{x_2}}\non\\
\lesssim&\e^{\f14} e^{\e t}\|\textbf{1}_{|\xi_1|\leq\delta}\|_{L^2_{\xi_1}}\| b_{1, 0}\|_{L^1_{x_1}L^2_{x_2}}\non\\
\lesssim&\e^{\f14} e^{\e t}\| b_{1, 0}\|_{L^1_{x_1}L^2_{x_2}} \to 0, \quad\text{for}\quad \e \to 0. \non
\end{align*}

Secondly we consider $|\xi_1|> 2$, and we only need to consider the case in $\wt{\Gamma}_1$, by \eqref{013}, we have solved it.

Thus we complete the proof of the Proposition 2.2.
\end{proof}

\smallskip

Now we  give the  resolvent estimate of derivative.

\begin{proposition}\label{linear  nau}
We have the following $L^2$ estimates
 \begin{align*}
 \|\pa_1 u_L\|_{L^2}&\lesssim \lan t\ran^{-1}(\|(u_0, b_0)\|_{L^1\cap L^2  }+\|(\pa_1u_0, \pa_1b_0)\|_{ L^2 }),\\
  \|\pa_2 u_{1,L}\|_{L^2}&\lesssim \lan t\ran^{-\f12}\|(\pa_2u_0, \pa_2b_0)\|_{L^1\cap L^2},\\
  \|\pa_1 b_{1,L}\|_{L^2}&\lesssim \lan t \ran ^{-\f34}(\|(u_{0},b_{0})\|_{L^1\cap L^2\cap L^1_{x_1}L^2_{x_2}}+\|(\pa_1u_0, \pa_1b_0)\|_{ L^2  }),\\
   \|\pa_1 b_{2,L}\|_{L^2}&\lesssim \lan t \ran ^{-1}(\|(u_{0},b_{0})\|_{L^1\cap L^2\cap L^1_{x_1}L^2_{x_2}}+\|(\pa_1u_0, \pa_1b_0)\|_{ L^2    }),
     \end{align*}
 for the linearized problem \eqref{eq:MHDL}  .
\end{proposition}
 \begin{proof}

 Recalling the structure of  $\wh{u}_{L}$ and $\wh{b}_L$, we only need to consider the typical term.\\

For the $\pa_1u_L$ and $\pa_1b_L$, by the horizontal Fourier transform, define $c=\max\{c_0,c_1\}$, we know that $\wh{\pa_1 u_L}=|\xi_1|\wh{ u_L}$, $\wh{\pa_1 b_L}=|\xi_1|\wh{ b_L}$, when $t>1$, the kernel $e^{-c(|\xi_1|^2+|\xi_2|^2)t}$, $\f{|\xi_1|}{|\xi_2|^2}e^{-(|\xi_1|^2+\f{|\xi_1|^2}{|\xi_2|^2})t}$ and $e^{-(|\xi_1|^2+\f{|\xi_1|^2}{|\xi_2|^2})t}$   all have the same part $e^{-|\xi_1|^2 t}$, by $|\xi_1|e^{-|\xi_1|^2 t}\lesssim t^{-\f12}$, it is just multiplying the original term by $t^{-\f12}$.\\

The rest of the terms are $\pa_1u_L$ and $\pa_1b_L$ when $0<t<1$, and $\pa_2 u_{1,L}$ for all $t$. For $\pa_1 u_L$ and $\pa_1b_L$ when $0<t<1$, in the proof of the proposition 2.1 and proposition 2.2, we get the conclusion $\|(u_L,b_L)\|_{L^2}\leq\|(u_0,b_0)\|_{L^2}$ for $0<t<1$, then change $(u_L,b_L)$ into $(\pa_1u_L,\pa_1b_L)$, we get the conclusion $\|(\pa_1u_L,\pa_1b_L)\|_{L^2}\leq\|(\pa_1u_0,\pa_1b_0)\|_{L^2}$ for $0<t<1$. For $\pa_2u_{1,L}$, we change $u_{1,L}$ into $\pa_2 u_{1,L}$ in \eqref{uL2}, then we get the conclusion $\|\pa_2 u_{1,L}\|_{L^2}\lesssim \lan t\ran^{-\f12}\|(\pa_2 u_0, \pa_2 b_0)\|_{L^1\cap L^2 }$.\\

Thus we complete the proof of the Proposition 2.3.
\end{proof}
\smallskip

\section{The  decay  rate of nonlinear system}
Now we consider the nonlinear part of \eqref{eq:u sol2}

\begin{align}\label{nonlinear}
\left\{
\begin{array}{l}
\wh{u}_1(\xi_1,x_2, t)=\f{1}{2\pi i}\int_{\Ga}e^{\lam t}
\Big\{ N_\om[\wh{f}_{1, \lam}]+\f{i\xi_1}{\lam+|\xi_1|^2}N_\om[\wh{g}_{1, \lam}]+i\xi_1N_\om[E_{|\xi_1|}[i\xi_1  \wh{f}_{1, \lam}+\p_2\wh{f}_{2, \lam}]]\\
\qquad   +\f{i\xi_1N_\om[e^{-|\xi_1|x_2}]}{|\xi_1|E_{\om}[e^{-|\xi_1|x_2}]_0}\{ E_\om[\wh{f}_{2, \lam}]_0+\f{i\xi_1}{\lam+|\xi_1|^2}E_\om[\wh{g}_{2, \lam}]_0+E_\om[\pa_2E_{|\xi_1|}[i\xi_1  \wh{f}_{1, \lam}+\p_2\wh{f}_{2, \lam}]]_0\}\Big\}d\lam,\\
\wh{u}_2(\xi_1,x_2, t)=\f{1}{2\pi i}\int_{\Ga}e^{\lam t}
\Big\{N_\om[\wh{f}_{2, \lam}]+\f{i\xi_1}{\lam+|\xi_1|^2}N_\om[\wh{g}_{2, \lam}]+N_\om[\pa_2E_{|\xi_1|}[i\xi_1  \wh{f}_{1, \lam}+\p_2\wh{f}_{2, \lam}]]\\
\qquad   -\f{N_\om[e^{-|\xi_1|x_2}]}{E_{\om}[e^{-|\xi_1|x_2}]_0}\{ E_\om[\wh{f}_{2, \lam}]_0+\f{i\xi_1}{\lam+|\xi_1|^2}E_\om[\wh{g}_{2, \lam}]_0+E_\om[\pa_2E_{|\xi_1|}[i\xi_1  \wh{f}_{1, \lam}+\p_2\wh{f}_{2, \lam}]]_0\}  \Big\}d\lam,\\
\wh{b}_1(\xi_1,x_2, t)=\f{1}{2\pi i}\int_{\Ga}e^{\lam t}
\Big\{\f{i\xi_1}{\lam+|\xi_1|^2}\big(N_\om[\wh{f}_{1, \lam}]+\f{i\xi_1}{\lam+|\xi_1|^2}N_\om[\wh{g}_{1, \lam}]+i\xi_1N_\om[E_{|\xi_1|}[i\xi_1  \wh{f}_{1, \lam}+\p_2\wh{f}_{2, \lam}]]\\
\qquad   +\f{i\xi_1N_\om[e^{-|\xi_1|x_2}]}{|\xi_1|E_{\om}[e^{-|\xi_1|x_2}]_0}\{ E_\om[\wh{f}_{2, \lam}]_0+\f{i\xi_1}{\lam+|\xi_1|^2}E_\om[\wh{g}_{2, \lam}]_0+E_\om[\pa_2E_{|\xi_1|}[i\xi_1  \wh{f}_{1, \lam}+\p_2\wh{f}_{2, \lam}]]_0\} \big)+\f{\wh{g}_{1,\lam}}{\lam+|\xi_1|^2}  \Big\}d\lam,\\
\wh{b}_2(\xi_1,x_2, t)=\f{1}{2\pi i}\int_{\Ga}e^{\lam t}
\Big\{ \f{i\xi_1}{\lam+|\xi_1|^2}\big(N_\om[\wh{f}_{2, \lam}]+\f{i\xi_1}{\lam+|\xi_1|^2}N_\om[\wh{g}_{2, \lam}]+N_\om[\pa_2E_{|\xi_1|}[i\xi_1  \wh{f}_{1, \lam}+\p_2\wh{f}_{2, \lam}]]\\
\qquad   -\f{N_\om[e^{-|\xi_1|x_2}]}{E_{\om}[e^{-|\xi_1|x_2}]_0}\{ E_\om[\wh{f}_{2, \lam}]_0+\f{i\xi_1}{\lam+|\xi_1|^2}E_\om[\wh{g}_{2, \lam}]_0+E_\om[\pa_2E_{|\xi_1|}[i\xi_1  \wh{f}_{1, \lam}+\p_2\wh{f}_{2, \lam}]]_0\}\big)+\f{\wh{g}_{2,\lam}}{\lam+|\xi_1|^2} \Big\}d\lam.\\
\end{array}\right.
\end{align}
 $\wh{u}_{1, N}$ can be rewritten as
\begin{align*}
&\wh{u}_{1, N}(\xi_1,x_2, t)\non\\
=&\f{1}{2\pi i}\int^t_0\int_{\Ga}e^{\lam (t-\tau)} \f{1}{2\om} \int^\infty_0  (e^{-\om|x_2-y_2|}-e^{-\om(x_2+y_2)}) ( \wh{f}_{1}+\f{i\xi_1}{\lam+|\xi_1|^2} \wh{g}_{1})(\tau, y_2) dy_2 d\lam d\tau\non\\
&+\f{i\xi_1}{|\xi_1|} \f{1}{2\pi i}\int^t_0\int_{\Ga}e^{\lam (t-\tau)} \int^\infty_0  e^{-\om y_2} ( \wh{f}_{2}+\f{i\xi_1}{\lam+|\xi_1|^2} \wh{g}_{2})(\tau, y_2) dy_2  \Big( \f{e^{-|\xi_1|x_2}-e^{-\om x_2}}{\om-|\xi_1|}\Big)  d\lam d\tau\\
&+\f{1}{2\pi i}\int^t_0\int_{\Ga}e^{\lam (t-\tau)} \f{i\xi_1}{2\om} \int^\infty_0 (e^{-\om|x_2-y_2|}-e^{-\om(x_2+y_2)}) E_{|\xi_1|} [i\xi_1 \wh{f}_{1}+\p_2 \wh{f}_{2}] (\tau, y_2) dy_2 d\lam d\tau\non\\
&+\f{i\xi_1}{|\xi_1|} \f{1}{2\pi i}\int^t_0\int_{\Ga}e^{\lam (t-\tau)} \int^\infty_0  e^{-\om y_2} \p_{2} E_{|\xi_1|}[i\xi_1\wh{f}_{1}+\p_2 \wh{f}_{2}] (\tau, y_2) dy_2 \Big( \f{e^{-|\xi_1|x_2}-e^{-\om x_2}}{\om-|\xi_1|}\Big)  d\lam d\tau\\
:=&\sum_{i=1}^4 N_i.
\end{align*}
 $\wh{b}_{1, N}$ can be rewritten as
\begin{align*}
&\wh{b}_{1, N}(\xi_1,x_2, t)\non\\
=&\f{1}{2\pi i}\int^t_0\int_{\Ga}e^{\lam (t-\tau)} (\f{i\xi_1}{\lam+|\xi_1|^2}\f{1}{2\om} \int^\infty_0  (e^{-\om|x_2-y_2|}-e^{-\om(x_2+y_2)}) ( \wh{f}_{1}+\f{i\xi_1}{\lam+|\xi_1|^2} \wh{g}_{1})(\tau, y_2) dy_2+\f{\wh{g}_1}{\lam+|\xi_1|^2}) d\lam d\tau\non\\
&+\f{i\xi_1}{|\xi_1|} \f{1}{2\pi i}\int^t_0\int_{\Ga}e^{\lam (t-\tau)}\f{i\xi_1}{\lam+|\xi_1|^2} \int^\infty_0  e^{-\om y_2} ( \wh{f}_{2}+\f{i\xi_1}{\lam+|\xi_1|^2} \wh{g}_{2})(\tau, y_2) dy_2  \Big( \f{e^{-|\xi_1|x_2}-e^{-\om x_2}}{\om-|\xi_1|}\Big)  d\lam d\tau\\
&+\f{1}{2\pi i}\int^t_0\int_{\Ga}e^{\lam (t-\tau)} \f{i\xi_1}{\lam+|\xi_1|^2}\f{i\xi_1}{2\om} \int^\infty_0 (e^{-\om|x_2-y_2|}-e^{-\om(x_2+y_2)}) E_{|\xi_1|} [i\xi_1 \wh{f}_{1}+\p_2 \wh{f}_{2}] (\tau, y_2) dy_2 d\lam d\tau\non\\
&+\f{i\xi_1}{|\xi_1|} \f{1}{2\pi i}\int^t_0\int_{\Ga}e^{\lam (t-\tau)}\f{i\xi_1}{\lam+|\xi_1|^2} \int^\infty_0  e^{-\om y_2} \p_{2} E_{|\xi_1|}[i\xi_1\wh{f}_{1}+\p_2 \wh{f}_{2}] (\tau, y_2) dy_2 \Big( \f{e^{-|\xi_1|x_2}-e^{-\om x_2}}{\om-|\xi_1|}\Big)  d\lam d\tau\\
:=&\sum_{i=1}^4 O_i,
\end{align*}
where
\begin{align*}
\left\{
\begin{array}{l}
\wh{f}=\cF_{x_1} (-u\cdot\na u+b\cdot\na b)=i\xi_1 \cF_{x_1}(- u_1u+b_1 b)+\cF_{x_1}(\p_2(- u_2 u+b_2 b)),\\
\wh{g}= \cF_{x_1} (-u\cdot\na b+b\cdot\na u)=i\xi_1\cF_{x_1}(- u_1b+b_1 u)+\cF_{x_1}(\p_2(- u_2 b+b_2 u)).
\end{array}\right.
\end{align*}

\begin{proposition}\label{nonlinear}
We have
 \begin{align*}
E(t)\leq E_0+\wt{E}^2(t)+\cE_1^2(t)+\int_0^{t}\cF_1^2(\tau)d\tau,
 \end{align*}
 where
\begin{align*}
 E_0&= \|(u_0, b_0)\|_{L^1\cap L^2\cap L^1_{x_1}L^2_{x_2}}+\|\pa_2(u_0, b_0)\|_{L^1\cap L^2\cap L^1_{x_1}L^2_{x_2}},\non\\
 E(t)&= \lan t\ran^\f14 \|b_1\|_{L^2} +\lan t\ran^{\f12-\delta}\|b_2\|_{L^2}+\lan t\ran^\f12\Big( \|u\|_{L^2}+\|\pa_2u_1\|_{L^2} \Big) \\
 &\quad+\lan t\ran^\f34 \|\p_1 b_1\|_{L^2}+\lan t \ran^{\f78-2\delta}\|\p_1 b_2\|_{L^2}+\lan t \ran\|\pa_1u\|_{L^2}\\
 \wt{E}(t)&=E(t)+\cE(t),
\end{align*}
and $\cE(t)$, $\cE_1(t)$, $\cF_1(t)$  are defined in \eqref{cE}, \eqref{cE2} and \eqref{cF2}.
\end{proposition}
\begin{proof}
With the help of linear decay rate of $u_L, b_L$ in Section 2, and notice that
\begin{align*}
&\Big|\cF_{y_2}\Big(\chi E_{|\xi_1|}[i\xi_1 \wh{f}_1+\p_2 \wh{f}_2]\Big)\Big|\non\\
=&\f{1}{2|\xi_1|} \Big|\cF_{y_2} \Big(e^{-|\xi_1||y_2|}*\big(\chi ( i\xi_1 \wh{f}_1+\p_2 \wh{f}_2)\big)\Big)\Big|\non\\
=&\f{1}{2|\xi_1|} \Big|\cF_{y_2}\Big (e^{-|\xi_1||y_2|}\Big)\cF_{y_2} \Big(\chi ( i\xi_1 \wh{f}_1+\p_2 \wh{f}_2)\Big)\Big|\non\\
=&\f{1}{2|\xi_1|} \Big|\f{1}{|\xi_1|-i\xi_2}+\f{1}{|\xi_1|+i\xi_2}\Big|\Big|\cF_{y_2}\Big(\chi ( i\xi_1 \wh{f}_1+\p_2 \wh{f}_2)\Big)\Big|\non\\
\lesssim& \f{1}{|\xi_1|}(\f{|\xi_1|}{(|\xi_1|^2+|\xi_2|^2)^\f12}|\cF(\chi f_1)|+\f{|\xi_2|}{(|\xi_1|^2+|\xi_2|^2)^\f12}|\cF( \chi f_2)|)\non\\
\lesssim& \f{1}{|\xi_1|}| \cF(\chi f)|,
\end{align*}
and for $N_4$, we replace the $\wh{u}_{L} $ of the linear part with $\p_{2} E_{|\xi_1|}[i\xi_1\wh{f}_{1}+\p_2 \wh{f}_{2}] $, and we get the same result, here we used the following result:
\begin{align*}
&\|\p_{2} E_{|\xi_1|}[i\xi_1\wh{f}_{1}+\p_2 \wh{f}_{2}]\|^2_{L^2_{x_2}}\non\\
=&\int_0^\infty\Big| -\f12  \int^{x_2}_0 e^{-|\xi_1| (x_2-y_2)} (i\xi_1  \wh{f}_{1, \lam}+\p_2\wh{f}_{2, \lam}) dy_2+\f12 \int^\infty_{x_2} e^{-|\xi_1| (y_2-x_2)}( i\xi_1  \wh{f}_{1}+\p_2\wh{f}_{2} ) dy_2\Big|^2d x_2\non\\
=&\int_0^\infty\Big| \f12  \int^{\infty}_0 e^{-|\xi_1| |x_2-y_2|}\Xi(x_2-y_2) (i\xi_1  \wh{f}_{1}+\p_2\wh{f}_{2}) dy_2\Big|^2d x_2\non\\
\lesssim& \int_{-\infty}^\infty\Big|   \int^{\infty}_0 e^{-|\xi_1| |x_2-y_2|}\Xi(x_2-y_2) (i\xi_1  \wh{f}_{1}+\p_2\wh{f}_{2}) dy_2\Big|^2d x_2\non\\
=& \int_{-\infty}^\infty\Big|  \int^{\infty}_{-\infty} e^{-|\xi_1| |x_2-y_2|}\Xi(x_2-y_2) \chi(y_2)(i\xi_1  \wh{f}_{1}+\p_2\wh{f}_{2}) dy_2\Big|^2d x_2\non\\
=& \int_{-\infty}^\infty\Big|  \big( \Xi e^{-|\xi_1| |x_2|}\big)* \big(\chi(i\xi_1  \wh{f}_{1}+\p_2\wh{f}_{2}) \big)\Big|^2d x_2,\non
\end{align*}
and
\begin{align*}
& \|  \cF_{y_2}\Big( \big( \Xi e^{-|\xi_1| |x_2|}\big)* \big(\chi(i\xi_1  \wh{f}_{1}+\p_2\wh{f}_{2}) \big)\Big)\|_{L^2_{x_2}}\non\\
= &\|  \cF_{y_2}\big(  \Xi e^{-|\xi_1| |x_2|}\big)\cF_{y_2} \big(\chi(i\xi_1  \wh{f}_{1}+\p_2\wh{f}_{2}) \big)\|_{L^2_{x_2}}\non\\
= &\|  (-\f{1}{|\xi_1|-i\xi_2}+\f{1}{|\xi_1|+i\xi_2})\cF_{y_2} \big(\chi(i\xi_1  \wh{f}_{1}+\p_2\wh{f}_{2}) \big)\|_{L^2_{x_2}}\non\\
= &\|  \f{2i\xi_2}{|\xi_1|^2+|\xi_2|^2}\cF_{y_2} \big(\chi(i\xi_1  \wh{f}_{1}+\p_2\wh{f}_{2})\big)\|_{L^2_{x_2}}\non\\
\lesssim& \|   \cF (f)\|_{L^2_{x_2}},
\end{align*}
here we  define
\begin{align*}
\Xi=\left\{
\begin{array}{c}
1, \qquad \text{when} \qquad x_2\geq 0 \\
-1,\qquad \text{when} \qquad x_2<0.
\end{array}\right.
\end{align*}

$N_i, O_i \ (i=1,3)$ can be controlled similar as the whole space part, and $N_i, O_i \ (i=2,4)$ is the different part comes from the half space domain.

{\bf Step 1.} The  decay rate of $u_N.$

In the linear terms, we can see the  kernel of $u_L$ is $e^{-c(|\xi_1|^2+|\xi_2|^2)t}$ and $\f{|\xi_1|}{|\xi_2|^2}e^{-\f{|\xi_1|^2}{|\xi_2|^2}t-|\xi_1|^2t}$, and the second kernel only happens when $|\xi_2|^2\gtrsim|\xi_1|$. 

We have the $L^2$ estimate
\begin{align*}
& \int^t_0 \Big\|e^{-c(|\xi_1|^2+|\xi_2|^2)(t-\tau)}\xi_i\cF(u_iu+u_ib+b_iu+b_ib)\Big\|_{L^2}d\tau\\
&+\int^t_0 \Big\|\textbf{1}_{|\xi_2|^2\gtrsim|\xi_1|}\f{|\xi_1|}{|\xi_2|^2}e^{-\f{|\xi_1|^2}{|\xi_2|^2}(t-\tau)-|\xi_1|^2(t-\tau)}
\xi_i\cF(u_iu+u_ib+b_iu+b_ib)\Big\|_{L^2}d\tau\\
\lesssim& \int^{t-1}_0 \Big\|e^{-c(|\xi_1|^2+|\xi_2|^2)(t-\tau)}\xi_i\cF(u_iu+u_ib+b_iu+b_ib)\Big\|_{L^2}d\tau\\
&+\int^{t-1}_0 \Big\|\f{|\xi_1|}{|\xi_2|^2}e^{-\f{|\xi_1|^2}{|\xi_2|^2}(t-\tau)-|\xi_1|^2(t-\tau)}
\xi_i\cF(u_iu+u_ib+b_iu+b_ib)\Big\|_{L^2}d\tau\\
&+\int^{t}_{t-1} \Big\|\cF(u_i\p_iu+u_i\p_ib+b_i\p_iu+b_i\p_ib)\Big\|_{L^2}d\tau\\
=& \int^{t-1}_0 \Big\|e^{-c(|\xi_1|^2+|\xi_2|^2)(t-\tau)}\xi_i\cF(u_iu+u_ib+b_iu+b_ib)\Big\|_{L^2}d\tau\\
&+\int^{t-1}_0 \Big\|\f{|\xi_1|^2}{|\xi_2|^2}e^{-\f{|\xi_1|^2}{|\xi_2|^2}(t-\tau)-|\xi_1|^2(t-\tau)}
\cF(u_1u+u_1b+b_1u+b_1b)\Big\|_{L^2}d\tau\\
&+\int^{t-1}_0 \Big\|\f{|\xi_1|}{|\xi_2|}e^{-\f{|\xi_1|^2}{|\xi_2|^2}(t-\tau)-|\xi_1|^2(t-\tau)}
\cF(u_2u+u_2b+b_2u+b_2b)\Big\|_{L^2}d\tau\\
&+\int^{t}_{t-1} \Big\|\cF(u_i\p_iu+u_i\p_ib+b_i\p_iu+b_i\p_ib)\Big\|_{L^2}d\tau\\
\lesssim& \int^{t-1}_0 \lan t-\tau\ran^{-1}\|u_iu+u_ib+b_iu+b_2b\|_{L^1} d\tau+\int^{t-1}_0 \lan t-\tau\ran^{-\f34}\Big(\|b_1b_1\|_{L^2_{x_1}L^1_{x_2}}+\|b_1\pa_1b_1\|_{L^2} \Big) d\tau\non\\
&+\int^{t-1}_0  \lan t-\tau\ran^{-1}\|u_1u+u_1b+b_1u+b_1b\|_{L^2}d\tau+\int^{t-1}_0  \lan t-\tau\ran^{-\f34}\|u_2u+u_2b+b_2u+b_2b\|_{L^1_{x_1}L^2_{x_2}}d\tau\non\\
&+\int^{t}_{t-1} \|u_i\p_iu+u_i\p_ib+b_i\p_iu+b_i\p_ib\|_{L^2}d\tau\\
\lesssim& \int^{t-1}_0 \lan t-\tau\ran^{-1}\Big(\|u\|_{L^2}+\|b_2\|_{L^2}\Big)\Big(\|u\|_{L^2}+\|b\|_{L^2}\Big) d\tau\\
&+\int^{t-1}_0 \lan t-\tau\ran^{-\f34}\Big(\|b_1\|_{L^2}\|b_1\|_{L^\infty_{x_1}L^2_{x_2}}+\|b_1\|_{L^\infty}\|\pa_1b_1\|_{L^2} \Big) d\tau\non\\
&+\int^{t-1}_0  \lan t-\tau\ran^{-1}\Big(\|u_1\|_{L^\infty_{x_1}L^2_{x_2}}+\|b_1\|_{L^\infty_{x_1}L^2_{x_2}}\Big)\Big(\|u\|_{L^2_{x_1}L^\infty_{x_2}}+\|b\|_{L^2_{x_1}L^\infty_{x_2}}\Big)d\tau\non\\
&+\int^{t-1}_0  \lan t-\tau\ran^{-\f34}\Big(\|u_2\|_{L^2_{x_1}L^\infty_{x_2}}+\|b_2\|_{L^2_{x_1}L^\infty_{x_2}}\Big)\Big(\|u\|_{L^2}+\|b\|_{L^2}\Big)d\tau\non\\
&+\sup_t \Big(\|u_1\|_{L^\infty}+\|b_1\|_{L^\infty}\Big)\Big(\|\p_1u\|_{L^2}+\|\p_1b\|_{L^2}\Big)\\
&+\sup_t\Big(\|u_2\|_{L^\infty_{x_1}L^2_{x_2}}+\|b_2\|_{L^\infty_{x_1}L^2_{x_2}}\Big)\Big(\|\p_2u\|_{L^2_{x_1}L^\infty_{x_2}}+\|\p_2b\|_{L^2_{x_1}L^\infty_{x_2}}\Big)\\
\lesssim& \int^t_0 \lan t-\tau\ran^{-1}\lan \tau\ran^{-(\f34-\delta)}+\lan t-\tau\ran^{-\f34}\lan \tau\ran^{-\f34}  d\tau  \,\wt{E}^2(t) \non\\
\lesssim& \lan t\ran^{-\f12} \wt{E}^2(t),
\end{align*}
here we use
\begin{align*}
& \Big\|e^{-c(|\xi_1|^2+|\xi_2|^2)t}\xi_1\cF(b_1b_1)\Big\|_{L^2}\\
\lesssim& \Big\||\xi_1|e^{-c(|\xi_1|^2+|\xi_2|^2)t}\Big\|_{L^\infty_{\xi_1}L^2_{\xi_2}}\Big\|\cF(b_1b_1)\Big\|_{L^2_{x_1}L^1_{x_2}}\\
\lesssim& t^{-\f34}\|b_1\|_{L^2}\|b_1\|_{L^\infty_{x_1}L^2_{x_2}},
\end{align*}
 and
\begin{align*}
& \Big\|\f{|\xi_1|}{|\xi_2|}e^{-\f{|\xi_1|^2}{|\xi_2|^2}t-|\xi_1|^2t}
\cF(u_2u+u_2b+b_2u+b_2b)\Big\|_{L^2}\\
\lesssim& t^{-\f12}\Big\|e^{-|\xi_1|^2t}\Big\|_{L^2_{\xi_1}}\Big\|\cF(u_2u+u_2b+b_2u+b_2b)\Big\|_{L^1_{x_1}L^2_{x_2}}\\
\lesssim& t^{-\f34}\|u_2u+u_2b+b_2u+b_2b\|_{L^1_{x_1}L^2_{x_2}},
\end{align*}
and
\begin{align*}
& \|u\|_{L^\infty_{x_1}L^2_{x_2}}\lesssim\|u\|_{L^2}^\f12\|\pa_1u\|_{L^2}^\f12,\\
& \|u\|_{L^2_{x_1}L^\infty_{x_2}}\lesssim\|u\|_{L^2}^\f12\|\pa_2u\|_{L^2}^\f12,
\end{align*}
and
\begin{align}
& \|\pa_2b_1\|_{L^2}\lesssim\|b_1\|_{L^2}^\f12\|\pa_2^2b_1\|_{L^2}^\f12\lesssim \lan t\ran^{-\f18} \wt{E}^2(t),\label{p2b1ls}\\
& \|b_2\|_{L^\infty_{x_1}L^2_{x_2}}\|\pa_2b_1\|_{L^2_{x_1}L^\infty_{x_2}}\lesssim\|b_2\|_{L^2}^\f12\|\pa_1b_2\|_{L^2}^\f12\|\pa_2b_1\|_{L^2}^\f12\|\pa_2^2b_1\|_{L^2}^\f12\lesssim \lan t\ran^{-(\f{13}{16}-\f32\delta)} \wt{E}^2(t).\non
\end{align}

{\bf Step 2.} The decay rate of  $ b_N$.

In the linear terms, we can see the  kernel of $b_L$ is $e^{-c(|\xi_1|^2+|\xi_2|^2)t}$, $\f{|\xi_1|}{|\xi_2|^2}e^{-\f{|\xi_1|^2}{|\xi_2|^2}t-|\xi_1|^2t}$ and $e^{-\f{|\xi_1|^2}{|\xi_2|^2}t-|\xi_1|^2t}$, and the last two kernel is defined in $|\xi_2|^2\gtrsim|\xi_1|$. Besides, $b_N$ also has the kernel $|\xi_1|^{-\f12}\f{|\xi_2|}{(|\xi_1|^2+|\xi_2|^2)^\f12}e^{-\f{|\xi_1|^2}{|\xi_2|^2}t-|\xi_1|^2t}$ for the $-u\cdot\na b_2+b\cdot \na u_2=\pa_1(-u_1b_2+b_1u_2)$ in the half place part by \eqref{b1l1l2}. 

We have the $L^2$ estimate,
\begin{align*}
\|b_{1,N}\|_{L^2}
\lesssim& \int^t_0 \Big\|e^{-c(|\xi_1|^2+|\xi_2|^2)(t-\tau)}\xi_i\cF(u_iu+u_ib+b_iu+b_ib)\Big\|_{L^2}d\tau\\
&+\int^t_0 \Big\|\textbf{1}_{|\xi_2|^2\gtrsim|\xi_1|}\f{|\xi_1|}{|\xi_2|^2}e^{-\f{|\xi_1|^2}{|\xi_2|^2}(t-\tau)-|\xi_1|^2(t-\tau)}
\xi_i\cF(u_iu+u_ib+b_iu+b_ib)\Big\|_{L^2}d\tau\\
&+\int^t_0 \Big\|e^{-\f{|\xi_1|^2}{|\xi_2|^2}(t-\tau)-|\xi_1|^2(t-\tau)}
\cF(u_i\p_ib+b_i\p_iu)\Big\|_{L^2}d\tau\\
&+\int^t_0 \Big\|e^{-\f{|\xi_1|^2}{|\xi_2|^2}(t-\tau)-|\xi_1|^2(t-\tau)}
|\xi_1|^{-\f12}\f{|\xi_2|}{(|\xi_1|^2+|\xi_2|^2)^\f12}|\xi_1|\cF(b_1u_2+u_1b_2)\Big\|_{L^2_{\xi_1}L^1_{\xi_2}}d\tau\\
\lesssim& \int^{t-1}_0 \Big\|e^{-c(|\xi_1|^2+|\xi_2|^2)(t-\tau)}\xi_i\cF(u_iu+u_ib+b_iu+b_ib)\Big\|_{L^2}d\tau\\
&+\int^{t-1}_0 \Big\|\f{|\xi_1|}{|\xi_2|^2}e^{-\f{|\xi_1|^2}{|\xi_2|^2}(t-\tau)-|\xi_1|^2(t-\tau)}
\xi_i\cF(u_iu+u_ib+b_iu+b_ib)\Big\|_{L^2}d\tau\\
&+\int^{t-1}_0 \Big\|e^{-\f{|\xi_1|^2}{|\xi_2|^2}(t-\tau)-|\xi_1|^2(t-\tau)}
\cF(u_i\p_ib+b_i\p_iu)\Big\|_{L^2}d\tau\\
&+\int^{t-1}_0 \Big\|e^{-\f{|\xi_1|^2}{|\xi_2|^2}(t-\tau)-|\xi_1|^2(t-\tau)}
|\xi_1|^{-\f12}\f{|\xi_2|}{(|\xi_1|^2+|\xi_2|^2)^\f12}|\xi_1|\cF(b_1u_2+u_1b_2)\Big\|_{L^2_{\xi_1}L^1_{\xi_2}}d\tau\\
&+\int^{t}_{t-1} \|u_i\p_iu+u_i\p_ib+b_i\p_iu+b_i\p_ib\|_{L^2}d\tau\\
\lesssim& \lan t\ran^{-\f12}\wt{E}^2(t)+\int^{t-1}_0 \Big\|e^{-\f{|\xi_1|^2}{|\xi_2|^2}(t-\tau)-|\xi_1|^2(t-\tau)}
\cF(u_i\p_ib+b_i\p_iu)\Big\|_{L^2}d\tau\\
&+\int^{t-1}_0 \Big\|e^{-\f{|\xi_1|^2}{|\xi_2|^2}(t-\tau)-|\xi_1|^2(t-\tau)}
|\xi_1|^{-\f12}\f{|\xi_2|}{(|\xi_1|^2+|\xi_2|^2)^\f12}|\xi_1|\cF(b_1u_2+u_1b_2)\Big\|_{L^2_{\xi_1}L^1_{\xi_2}}d\tau\\
=& \lan t\ran^{-\f12}\wt{E}^2(t)+\cT_1+\cT_2,
\end{align*}
here we use \eqref{2000} and
\begin{align*}
&\int^{t}_{t-1} \Big\|e^{-\f{|\xi_1|^2}{|\xi_2|^2}(t-\tau)-|\xi_1|^2(t-\tau)}
|\xi_1|^{-\f12}\f{|\xi_2|}{(|\xi_1|^2+|\xi_2|^2)^\f12}|\xi_1|\cF(b_1u_2+u_1b_2)\Big\|_{L^2_{\xi_1}L^1_{\xi_2}}d\tau\non\\
\lesssim&\int^{t}_{t-1} \Big\|
\f{|\xi_1|^\f12}{(|\xi_1|^2+|\xi_2|^2)^\f12}\cF\Big(\pa_2(b_1u_2+u_1b_2)\Big)\Big\|_{L^2_{\xi_1}L^1_{\xi_2}}d\tau\non\\
\lesssim&\int^{t}_{t-1} \Big\|
\f{|\xi_1|^\f12}{(|\xi_1|^2+|\xi_2|^2)^\f12}\Big\|_{L^\infty_{\xi_1}L^2_{\xi_2}}\Big\|\cF\Big(\pa_2(b_1u_2+u_1b_2)\Big)\Big\|_{L^2}d\tau\non\\
\lesssim&\int^{t}_{t-1} \Big\|\cF\Big(\pa_2(b_1u_2+u_1b_2)\Big)\Big\|_{L^2}d\tau\non\\
\lesssim&\int^{t}_{t-1} \|\pa_2(b_1u_2+u_1b_2)\|_{L^2}d\tau\non\\
\lesssim&\int^{t}_{t-1} \|u_i\p_iu+u_i\p_ib+b_i\p_iu+b_i\p_ib\|_{L^2}d\tau.
\end{align*}

Next we  consider $\cT_1$,
\begin{align*}
\cT_1=&\int^{t-1}_0 \Big\|e^{-\f{|\xi_1|^2}{|\xi_2|^2}(t-\tau)-|\xi_1|^2(t-\tau)}
\cF(u_1\p_1b)\Big\|_{L^2}d\tau\\
&+\int^{t-1}_0 \Big\|e^{-\f{|\xi_1|^2}{|\xi_2|^2}(t-\tau)-|\xi_1|^2(t-\tau)}
\cF(u_2\p_2b_1)\Big\|_{L^2}d\tau\\
&+\int^{t-1}_0 \Big\|e^{-\f{|\xi_1|^2}{|\xi_2|^2}(t-\tau)-|\xi_1|^2(t-\tau)}
\cF(u_2\p_2b_2)\Big\|_{L^2}d\tau\\
&+\int^{t-1}_0 \Big\|e^{-\f{|\xi_1|^2}{|\xi_2|^2}(t-\tau)-|\xi_1|^2(t-\tau)}
\cF(b_1\p_1u)\Big\|_{L^2}d\tau\\
&+\int^{t-1}_0 \Big\|e^{-\f{|\xi_1|^2}{|\xi_2|^2}(t-\tau)-|\xi_1|^2(t-\tau)}
\cF(b_2\p_2u)\Big\|_{L^2}d\tau\\
=&\cT_{11}+\cT_{12}+\cT_{13}+\cT_{14}+\cT_{15},
\end{align*}
and
\begin{align*}
&\cT_{11}+\cT_{13}+\cT_{14}+\cT_{15}\\
\lesssim&\int^{t-1}_0 \lan t-\tau\ran^{-\f14}\Big(
\|u_1\p_1b\|_{L^1_{x_1}L^2_{x_2}}+\|u_2\p_2b_2\|_{L^1_{x_1}L^2_{x_2}}+\|b_1\p_1u\|_{L^1_{x_1}L^2_{x_2}}+\|b_2\p_2u\|_{L^1_{x_1}L^2_{x_2}}\Big)d\tau\\
\lesssim&\int^{t-1}_0 \lan t-\tau\ran^{-\f14}\Big(
\|u_1\|_{L^2_{x_1}L^\infty_{x_2}}\|\p_1b\|_{L^2}+\|u_2\|_{L^2_{x_1}L^\infty_{x_2}}\|\p_2b_2\|_{L^2}+\|b_1\|_{L^2_{x_1}L^\infty_{x_2}}\|\p_1u\|_{L^2}+\|b_2\|_{L^2_{x_1}L^\infty_{x_2}}\|\p_2u\|_{L^2}\Big)d\tau\\
\lesssim&\int^{t-1}_0 \lan t-\tau\ran^{-\f14}\lan \tau \ran^{-\f98}d\tau\wt{E}^2(t)\\
\lesssim& \lan t \ran^{-\f14}\wt{E}^2(t).
\end{align*}

The most difficult item is $\cT_{12}$, we divide $u_2=u_{2, <\lan \tau\ran^{-s_1}}+u_{2, >\lan \tau\ran^{-s_2}} +u_{2, \sim} $ corresponding to $\cT_{12}=\cT_{121}+\cT_{122}+\cT_{123}$, here $s_1>1, s_2=\f14-\sigma$. For the case $u_{2, <\lan \tau\ran^{-s_1}}$ and $u_{2, >\lan \tau\ran^{-s_2}}$, we have
\begin{align*}
\cT_{121}+\cT_{122}=&\int^{t-1}_0 \|e^{-\f{|\xi_1|^2}{|\xi_2|^2}(t-\tau)-|\xi_1|^2(t-\tau)}
\cF(u_{2,<\lan \tau\ran^{-s_1}}\p_2b_1+u_{2, >\lan \tau\ran^{-s_2}}\p_2b_1)\|_{L^2}d\tau\\
\lesssim& \int^t_0 \lan t-\tau\ran^{-\f14}\Big(\|u_{2,<\lan \tau\ran^{-s_1}}\|_{L^2}^\f12\|\p_2u_{2,<\lan \tau\ran^{-s_1}}\|_{L^2}^\f12\|\p_2b_1\|_{L^2}\\
&+\|u_{2,>\lan \tau\ran^{-s_2}}\|_{L^2}^\f12\|\p_2u_{2,>\lan \tau\ran^{-s_2}}\|_{L^2}^\f12\|\p_2b_{1,<\lan \tau\ran^{\f18}}+\p_2b_{1,>\lan \tau\ran^{\f18}}\|_{L^2} \Big) d\tau\non\\
\lesssim& \int^t_0 \lan t-\tau\ran^{-\f14}\Big(\lan \tau\ran^{-\f{s_1}{2}}\|u_{2,<\lan \tau\ran^{-s_1}}\|_{L^2}\|\p_2b_1\|_{L^2}\\
&+\lan \tau\ran^{\f{s_2}{2}}\|\p_2u_{2,>\lan \tau\ran^{-s_2}}\|_{L^2}\big(\lan \tau\ran^{\f18}\|b_{1,<\lan \tau\ran^{\f18}}\|_{L^2}+\lan \tau\ran^{-\f18}\|\p_2^2b_{1,>\lan \tau\ran^{\f18}}\|_{L^2}\big) \Big) d\tau\non\\
\lesssim& \int^t_0 \lan t-\tau\ran^{-\f14}\lan \tau\ran^{-(\f{s_1}{2}+\f12)}+\lan t-\tau\ran^{-\f14}\lan \tau\ran^{-(1-\f{s_2}{2}+\f18)}  d\tau  \,\wt{E}^2(t) \non\\
\lesssim& \lan t\ran^{-\f14}\wt{E}^2(t) .
\end{align*}

For the $\cT_{123}$, we use magnetic equation
\begin{align*}
\cT_{123}=&\int^{t-1}_0 \||\xi_2|e^{-\f{|\xi_1|^2}{|\xi_2|^2}(t-\tau)-|\xi_1|^2(t-\tau)}
\cF(u_{2,\sim}b_1)\|_{L^2}d\tau\\
=&\int^{t-1} _0\Big\| |\xi_2| e^{-\f{|\xi_1|^2}{|\xi_2|^2}(t-\tau)-|\xi_1|^2(t-\tau)} \cF\Big( (-\Delta)^{-1} \na^T\cdot \p_1 u_\sim  \,b_1\Big) \Big\|_{L^2} d\tau\non\\
=&\int^{t-1} _0\Big\| |\xi_2| e^{-\f{|\xi_1|^2}{|\xi_2|^2}(t-\tau)-|\xi_1|^2(t-\tau)} \cF\Big( (-\Delta)^{-1} \na^T\cdot (\p_\tau b-\p_1^2b+u\cdot\na b-b\cdot\na u)_{\sim}  \,b_1\Big) \Big\|_{L^2} d\tau\\
\lesssim&\int^{t-1} _0\Big\| |\xi_2| e^{-\f{|\xi_1|^2}{|\xi_2|^2}(t-\tau)-|\xi_1|^2(t-\tau)} \cF\Big( (-\Delta)^{-1} \na^T\cdot (\p_\tau b+u\cdot\na b-b\cdot\na u)_{\sim}  \,b_1\Big) \Big\|_{L^2} d\tau\\
&+\int^{t-1} _0\Big\| |\xi_2| e^{-\f{|\xi_1|^2}{|\xi_2|^2}(t-\tau)-|\xi_1|^2(t-\tau)} \cF\Big( \p_1 b_{2,\sim}  \,b_1\Big) \Big\|_{L^2} d\tau,
\end{align*}
since the nonlinear term will bring more decay, here we only estimate
\begin{align}\label{9000}
&\int^{t-1} _0\Big\| |\xi_2| e^{-\f{|\xi_1|^2}{|\xi_2|^2}(t-\tau)-|\xi_1|^2(t-\tau)} \cF\Big( (-\Delta)^{-1} \na^T\cdot \p_\tau b_{\sim}  \,b_1\Big) \Big\|_{L^2} d\tau\non\\
&\lesssim\Big\| |\xi_2|  \cF\Big( (-\Delta)^{-1} \na^T\cdot b_{\sim}  \,b_{1}\Big) \Big\|_{L^2}+\Big\| |\xi_2| e^{-\f{|\xi_1|^2}{|\xi_2|^2}t-|\xi_1|^2t} \cF\Big( (-\Delta)^{-1} \na^T\cdot b_{\sim, 0}  \,b_{1, 0}\Big) \Big\|_{L^2} \non\\
&\quad +\int^{t-1} _0\Big\||\xi_2| (\f{|\xi_1|^2}{|\xi_2|^2}+|\xi_1|^2)e^{-\f{|\xi_1|^2}{|\xi_2|^2}(t-\tau)-|\xi_1|^2(t-\tau)} \cF\Big( (-\Delta)^{-1} \na^T\cdot  b_{\sim}  \,b_1\Big) \Big\|_{L^2} d\tau\non\\
&\quad +\int^{t-1} _0\Big\| |\xi_2| e^{-\f{|\xi_1|^2}{|\xi_2|^2}(t-\tau)-|\xi_1|^2(t-\tau)} \cF\Big( (-\Delta)^{-1} \na^T\cdot  b_{\sim}  \,\p_\tau b_1\Big) \Big\|_{L^2} d\tau\non\\
&\lesssim  \Big\| \na\Big( (-\Delta)^{-1} \na^T\cdot b_{\sim}  \,b_{1}\Big) \Big\|_{L^2} +\lan t\ran^{-\f14} \Big\| \na  \Big( (-\Delta)^{-1} \na^T\cdot b_{\sim, 0}  \,b_{1, 0}\Big) \Big\|_{L^1_{x_1}L^2_{x_2}}\non\\
&\quad +\int^{t-1}_0 \lan t-\tau\ran^{-1} \Big\| \na\Big( (-\Delta)^{-1} \na^T\cdot b_{\sim}  \,b_{1}\Big) \Big\|_{L^2} d\tau\non\\
&\quad +\int^{t-1} _0\Big\| |\xi_2|e^{-\f{|\xi_1|^2}{|\xi_2|^2}(t-\tau)-|\xi_1|^2(t-\tau)} \cF\Big( (-\Delta)^{-1} \na^T\cdot  b_{\sim}  \,(\p_1^2b_1+\p_1 u_1-u\cdot\na b+b\cdot\na u\Big) \Big\|_{L^2} d\tau.
\end{align}

We have
\begin{align}\label{10000}
\Big \|\na\Big( (-\Delta)^{-1} \na^{T} \cdot b_{\sim}b_1\Big)\Big\|_{L^2}
&\lesssim \| b_{ \sim}\|_{L^2}\| b_1\|_{L^\infty}+\|(-\Delta)^{-1} \na^{T} \cdot b_{\sim}\|_{L^\infty}\| \na b_1\|_{L^2}\non\\
&\lesssim \| b_{ \sim}\|_{L^2}\| b_1\|_{L^\infty}+\Big \|\cF\Big( (-\Delta)^{-1} \na^{T} \cdot b_{\sim}\Big)\Big\|_{L^1_{\xi_1}L^1_{\xi_2}}\| \na b_1\|_{L^2}\non\\
&\lesssim \lan t \ran^{-\f12-\delta} \wt{E}^2(t),
\end{align}
here we use
\begin{align}\label{na-1}
\Big \|\cF\Big( (-\Delta)^{-1} \na^{T} \cdot b_{\sim}\Big)\Big\|_{L^1_{\xi_1}L^1_{\xi_2}}
&\lesssim\Big \|\cF\Big(|\na |^{-1} b_{2, \sim}\Big)\Big\|_{L^1_{\xi_1}L^1_{\xi_2}}\non\\
&\lesssim \sum_{\lan t \ran^{-s_1}\leq j\leq \lan t \ran^{-s_2} }\Big \|\cF\Big( P_j (|\na |^{-1} b_{2}) \Big)\Big\|_{L^1_{\xi_1}L^1_{\xi_2}}\non\\
&\lesssim \sum_{\lan t \ran^{-s_1}\leq j\leq \lan t \ran^{-s_2} } j^{-1} \Big \|\cF\Big( P_j  b_{2} \Big)\Big\|_{L^1_{\xi_1}L^1_{\xi_2}}\non\\
&\lesssim  \|b_2\|_{L^2} \sum_{\lan t \ran^{-s_1}\leq j\leq \lan t\ran^{-s_2} } j^{\f{\delta}{2s_1}}j^{-\f{\delta}{2s_1}} \non\\
&\lesssim \lan t \ran^{-\f12+\f32\delta} \wt{E}(t).
\end{align}

And we have
\begin{align*}
&\int^{t-1} _0\Big\| |\xi_2|e^{-\f{|\xi_1|^2}{|\xi_2|^2}(t-\tau)-|\xi_1|^2(t-\tau)} \cF\Big( (-\Delta)^{-1} \na^T\cdot  b_{\sim}  \,\p_1^2b_1\Big) \Big\|_{L^2} d\tau \\
\lesssim &\int^{t-1} _0(\lan t-\tau \ran^{-\f12}+\lan \tau \ran^{-s_2})\Big\| |\xi_2|e^{-\f{|\xi_1|^2}{|\xi_2|^2}(t-\tau)-|\xi_1|^2(t-\tau)} \cF\Big( (-\Delta)^{-1} \na^T\cdot  b_{\sim}  \,\p_1b_1\Big) \Big\|_{L^2} d\tau\\
\lesssim &\int^{t-1} _0(\lan t-\tau \ran^{-\f12}+\lan \tau \ran^{-s_2})\Big(\lan t-\tau \ran^{-\f12}\Big \|\na\Big( (-\Delta)^{-1} \na^{T} \cdot b_{\sim}b_1\Big)\Big\|_{L^2}+\lan t-\tau \ran^{-\f14}\big \| b_{2,\sim}b_1\big\|_{L^1_{x_1}L^2_{x_2}} \Big)d\tau\\
\lesssim &\int^{t-1} _0(\lan t-\tau \ran^{-\f12}+\lan \tau \ran^{-s_2})\Big(\lan t-\tau \ran^{-\f12} \big(\|b_{\sim}\|_{L^\infty}\|b_1\|_{L^2}+ \| (-\Delta)^{-1} \na^{T} \cdot b_{\sim}\|_{L^\infty}\|b_1\|_{L^2}\big)\\
&+\lan t-\tau \ran^{-\f14} \| b_{2,\sim}\|_{L^2_{x_1}L^\infty_{x_2}}\|b_1\|_{L^2} \Big)d\tau\\
\lesssim &\lan t \ran^{-\f14} \wt{E}^2(t),
\end{align*}
and
\begin{align*}
&\int^{t-1} _0\Big\| |\xi_2|e^{-\f{|\xi_1|^2}{|\xi_2|^2}(t-\tau)-|\xi_1|^2(t-\tau)} \cF\Big( (-\Delta)^{-1} \na^T\cdot  b_{\sim}  \,\p_1u_1\Big) \Big\|_{L^2} d\tau \\
\lesssim &\int^{t-1} _0\Big(\lan t-\tau \ran^{-\f12}\Big \|\na\Big( (-\Delta)^{-1} \na^{T} \cdot b_{\sim}u_1\Big)\Big\|_{L^2}+\lan t-\tau \ran^{-\f14}\big \| b_{2,\sim}u_1\big\|_{L^1_{x_1}L^2_{x_2}} \Big)d\tau\\
\lesssim &\int^{t-1} _0\Big(\lan t-\tau \ran^{-\f12} \big(\|b_{\sim}\|_{L^\infty}\|u_1\|_{L^2}+ \| (-\Delta)^{-1} \na^{T} \cdot b_{\sim}\|_{L^\infty}\|\na u_1\|_{L^2}\big)+\lan t-\tau \ran^{-\f14} \| b_{2,\sim}\|_{L^2}\|u_1\|_{L^2_{x_1}L^\infty_{x_2}} \Big)d\tau\\
\lesssim &\lan t \ran^{-\f14} \wt{E}^2(t),
\end{align*}
here we use \eqref{na-1}.

Next we consider $\cT_2$, by \eqref{2000}, we have
\begin{align*}
\cT_2\lesssim&\int^{t-1}_0 \Big\|e^{-\f{|\xi_1|^2}{|\xi_2|^2}(t-\tau)-|\xi_1|^2(t-\tau)}
\f{|\xi_1|^{\f12}|\xi_2|}{(|\xi_1|^2+|\xi_2|^2)^\f12}\cF\Big(b_1u_2+u_1b_2\Big)\Big\|_{L^2_{\xi_1}L^1_{\xi_2}}d\tau\\
\lesssim&\int^{t-1}_0 \Big\|\textbf{1}_{|\xi_2|<1}e^{-\f{|\xi_1|^2}{|\xi_2|^2}(t-\tau)-|\xi_1|^2(t-\tau)}
\f{|\xi_1|^{\f12}|\xi_2|}{(|\xi_1|^2+|\xi_2|^2)^\f12}\cF\Big(b_1u_2+u_1b_2\Big)\Big\|_{L^2_{\xi_1}L^1_{\xi_2}}d\tau\\
&+\int^{t-1}_0 \Big\|\textbf{1}_{|\xi_2|\geq1}e^{-\f{|\xi_1|^2}{|\xi_2|^2}(t-\tau)-|\xi_1|^2(t-\tau)}
\f{|\xi_1|^{\f12}|\xi_2|}{(|\xi_1|^2+|\xi_2|^2)^\f12}\cF\Big(b_1u_2+u_1b_2\Big)\Big\|_{L^2_{\xi_1}L^1_{\xi_2}}d\tau\\
\lesssim&\int^{t-1}_0 \Big\|\textbf{1}_{|\xi_2|<1}e^{-\f{|\xi_1|^2}{|\xi_2|^2}(t-\tau)-|\xi_1|^2(t-\tau)}
|\xi_1|^{\f12}\cF\Big(b_1u_2+u_1b_2\Big)\Big\|_{L^2_{\xi_1}L^2_{\xi_2}}\Big\|\textbf{1}_{|\xi_2|<1}\Big\|_{L^2_{\xi_2}}d\tau\\
&+\int^{t-1}_0 \Big\|\textbf{1}_{|\xi_2|\geq1}e^{-\f{|\xi_1|^2}{|\xi_2|^2}(t-\tau)-|\xi_1|^2(t-\tau)}
|\xi_1|^{\f12}\cF\Big(\pa_2(b_1u_2+u_1b_2)\Big)\Big\|_{L^2_{\xi_1}L^2_{\xi_2}}\Big\|\textbf{1}_{|\xi_2|\geq1}\f{1}{|\xi_2|}\Big\|_{L^2_{\xi_2}}d\tau\\
\lesssim&\int^{t-1}_0 \Big\|\textbf{1}_{|\xi_2|<1}e^{-\f{|\xi_1|^2}{|\xi_2|^2}(t-\tau)-|\xi_1|^2(t-\tau)}
|\xi_1|^{\f12}\cF\Big(b_1u_2+u_1b_2\Big)\Big\|_{L^2_{\xi_1}L^2_{\xi_2}}d\tau\\
&+\int^{t-1}_0 \Big\|\textbf{1}_{|\xi_2|\geq1}e^{-\f{|\xi_1|^2}{|\xi_2|^2}(t-\tau)-|\xi_1|^2(t-\tau)}
|\xi_1|^{\f12}\cF\Big(\pa_2(b_1u_2+u_1b_2)\Big)\Big\|_{L^2_{\xi_1}L^2_{\xi_2}}d\tau\\
=&\cT_{21}+\cT_{22},
\end{align*}
and
\begin{align*}
\cT_{21}\lesssim&\int^{t-1}_0 \Big\||\xi_1|^\f12e^{-|\xi_1|^2(t-\tau)}\Big\|_{L^2_{\xi_1}}\Big\|
b_1u_2+u_1b_2\Big\|_{L^1_{x_1}L^2_{x_2}}d\tau\\
\lesssim&\int^{t-1}_0\lan t-\tau\ran^{-\f12}\Big(\|b_1\|_{L^2}\|u_2\|_{L^2_{x_1}L^\infty_{x_2}}+\|u_1\|_{L^2}\|b_2\|_{L^2_{x_1}L^\infty_{x_2}}\Big)
d\tau\\
\lesssim&\int^{t-1}_0\lan t-\tau\ran^{-\f12}\lan\tau\ran^{-1}
d\tau \wt{E}^2(t)\\
\lesssim&\lan t\ran^{-(\f12-\delta)}
\wt{E}^2(t).
\end{align*}

Next we consider $\cT_{22}$. $\cT_{22}$ can be seen as $|\xi_1|^\f12\cT_{1}  $, and by $|\xi_1|^\f12e^{-|\xi_1|^2(t-\tau)}\lesssim\lan  t-\tau \ran^{-\f12}$ when $t-\tau>1$, we only need to give the calculation of the difference items in \eqref{9000}.
\begin{align*}
&\int^{t-1} _0\Big\||\xi_1|^{\f12} |\xi_2| e^{-\f{|\xi_1|^2}{|\xi_2|^2}(t-\tau)-|\xi_1|^2(t-\tau)} \cF\Big( (-\Delta)^{-1} \na^T\cdot \p_\tau b_{\sim}  \,b_1\Big) \Big\|_{L^2} d\tau\non\\
&\lesssim  \Big\| \na\Big( (-\Delta)^{-1} \na^T\cdot b_{\sim}  \,b_{1}\Big) \Big\|_{L^2} +\lan t\ran^{-\f12} \Big\| \na  \Big( (-\Delta)^{-1} \na^T\cdot b_{\sim, 0}  \,b_{1, 0}\Big) \Big\|_{L^1_{x_1}L^2_{x_2}}\non\\
&\quad +\int^{t-1}_0 \lan t-\tau\ran^{-\f54} \Big\| \na\Big( (-\Delta)^{-1} \na^T\cdot b_{\sim}  \,b_{1}\Big) \Big\|_{L^2} d\tau\non\\
&\quad +\int^{t-1} _0\lan t-\tau\ran^{-\f14}\Big\| |\xi_2|e^{-\f{|\xi_1|^2}{|\xi_2|^2}(t-\tau)-|\xi_1|^2(t-\tau)} \cF\Big( (-\Delta)^{-1} \na^T\cdot  b_{\sim}  \,(\p_1^2b_1+\p_1 u_1-u\cdot\na b+b\cdot\na u)\Big) \Big\|_{L^2} d\tau\\
\lesssim &\lan t \ran^{-\f12} \wt{E}^2(t),
\end{align*}
here we use \eqref{10000}.

Thus we have
\begin{align*}
\cT_{22}\lesssim\lan t \ran^{-\f12} \wt{E}^2(t),
\end{align*}
and then we proof the result
\begin{align*}
\|b_{1,N}\|_{L^2}\lesssim \lan t \ran^{-\f14}\wt{E}^2(t).
\end{align*}

Next we consider $b_{2,N}$, since $-u\cdot \na b_2+b\cdot\na u_2=\pa_1(b_1u_2-u_1b_2)$, we have
\begin{align*}
\|b_{2,N}\|_{L^2}
\lesssim& \int^{t-1}_0 \Big\|e^{-c(|\xi_1|^2+|\xi_2|^2)(t-\tau)}\xi_i\cF(u_iu+u_ib+b_iu+b_ib)\Big\|_{L^2}d\tau\\
&+\int^{t-1}_0 \Big\|\f{|\xi_1|}{|\xi_2|^2}e^{-\f{|\xi_1|^2}{|\xi_2|^2}(t-\tau)-|\xi_1|^2(t-\tau)}
\xi_i\cF(u_iu+u_ib+b_iu+b_ib)\Big\|_{L^2}d\tau\\
&+\int^{t-1}_0 \|e^{-\f{|\xi_1|^2}{|\xi_2|^2}(t-\tau)-|\xi_1|^2(t-\tau)}
\xi_1\cF(b_1u_2+u_1b_2)\|_{L^2}d\tau\\
&+\int^{t-1}_0 \Big\|e^{-\f{|\xi_1|^2}{|\xi_2|^2}(t-\tau)-|\xi_1|^2(t-\tau)}
|\xi_1|^{-\f12}\f{|\xi_2|}{(|\xi_1|^2+|\xi_2|^2)^\f12}|\xi_1|\cF(b_1u_2+u_1b_2)\Big\|_{L^2_{\xi_1}L^1_{\xi_2}}d\tau\\
&+\int^{t}_{t-1} \|u_i\p_iu+u_i\p_ib+b_i\p_iu+b_i\p_ib\|_{L^2}d\tau\\
\lesssim& \lan t\ran^{-(\f12-\delta)} \wt{E}^2(t)+\int^{t-1}_0 \|e^{-\f{|\xi_1|^2}{|\xi_2|^2}(t-\tau)-|\xi_1|^2(t-\tau)}
\xi_1\cF(b_1u_2+u_1b_2)\|_{L^2}d\tau\\
=& \lan t\ran^{-(\f12-\delta)} \wt{E}^2(t)+\cT',
\end{align*}
here we use the fact that the $\wh{b}_2$ and $\wh{b}_1$ are only different in $\cT'$, and we have proved that $\cT_2$ have the $\lan t \ran^{-(\f12-\delta)}$ decay rate. For $\cT'$, we have
\begin{align*}
\cT'=&\int^{t-1}_0 \|e^{-\f{|\xi_1|^2}{|\xi_2|^2}(t-\tau)-|\xi_1|^2(t-\tau)}
\xi_1\cF(b_1u_2+u_1b_2)\|_{L^2}d\tau\\
\lesssim&\int^{t-1}_0 \lan t-\tau\ran^{-\f12}
\|b_1u_2+u_1b_2\|_{L^2}d\tau\\
\lesssim&\int^{t-1}_0 \lan t-\tau\ran^{-\f12}
\Big(\|b_1\|_{L^\infty_{x_1}L^2_{x_2}}\|u_2\|_{L^2_{x_1}L^\infty_{x_2}}+\|u_1\|_{L^\infty_{x_1}L^2_{x_2}}\|b_2\|_{L^2_{x_1}L^\infty_{x_2}}\Big)d\tau\\
\lesssim&\int^{t-1}_0 \lan t-\tau\ran^{-\f12}
\lan \tau\ran^{-\f54}d\tau\\
\lesssim& \lan t\ran^{-\f12} \wt{E}^2(t),
\end{align*}
and  others cases we have solved in $\wh{b}_{1,N}$.

{\bf Step 3.} The  decay rate of $\|\pa_1 u_N\|_{L^2}$.

We have the following estimate,
\begin{align*}
&\|\pa_1u_N\|_{L^2}\\
=&\int^t_0 \Big\|e^{-c(|\xi_1|^2+|\xi_2|^2)(t-\tau)}\xi_i\xi_1\cF(u_iu+u_ib+b_iu+b_ib)\Big\|_{L^2}d\tau\\
&+\int^t_0 \Big\|\textbf{1}_{|\xi_2|^2\gtrsim|\xi_1|}\f{|\xi_1|}{|\xi_2|^2}e^{-\f{|\xi_1|^2}{|\xi_2|^2}(t-\tau)-|\xi_1|^2(t-\tau)}
\xi_i\xi_1\cF(u_iu+u_ib+b_iu+b_ib)\Big\|_{L^2}d\tau\\
\lesssim& \int^{t-1}_0 \lan t-\tau\ran^{-1}\|\pa_1(u_iu+u_ib+b_iu+b_ib_2)\|_{L^1} d\tau\non\\
&+ \int^{\f{t}{2}}_0 \lan t-\tau\ran^{-\f54}\|b_1b_1\|_{L^2_{x_1}L^1_{x_2}}d\tau + \int^{t-1}_{\f{t}{2}} \lan t-\tau\ran^{-\f34}\|b_1\pa_1b_1\|_{L^2_{x_1}L^1_{x_2}}d\tau \non\\
&+\int^{t-1}_0  \lan t-\tau\ran^{-1}\|\pa_1(u_1u+u_1b+b_1u+b_1b)\|_{L^2}d\tau+\int^{t-1}_0  \lan t-\tau\ran^{-1}\|u_2u+u_2b+b_2u+b_2b\|_{L^2}d\tau\non\\
&+ \int^{t}_{t-1} \|\pa_1(u_i\pa_iu+u_i\pa_ib+b_i\pa_iu+b_i\pa_ib)\|_{L^2}  d\tau\non\\
\lesssim& \int^{t-1}_0 \lan t-\tau\ran^{-1}\Big(\|\pa_1u\|_{L^2}(\|u\|_{L^2}+\|b\|_{L^2})+\|\pa_1b\|_{L^2}\|u\|_{L^2}+\|\pa_1b_2\|_{L^2}\|b\|_{L^2}+\|\pa_1b\|_{L^2}\|b_2\|_{L^2} \Big) d\tau\non\\
&+ \int^{\f{t}{2}}_0 \lan t-\tau\ran^{-\f54}\|b_1\|_{L^2}\|b_1\|_{L^\infty_{x_1}L^2_{x_2}}d\tau + \int^{t-1}_{\f{t}{2}} \lan t-\tau\ran^{-\f34}\|\pa_1b_1\|_{L^2}\|b_1\|_{L^\infty_{x_1}L^2_{x_2}}d\tau \non\\
&+\int^{t-1}_0  \lan t-\tau\ran^{-1}\Big((\|\pa_1u_1\|_{L^\infty_{x_1}L^2_{x_2}}+\|\pa_1b_1\|_{L^\infty_{x_1}L^2_{x_2}})(\|u\|_{L^2_{x_1}L^\infty_{x_2}}+\|b\|_{L^2_{x_1}L^\infty_{x_2}})\\
&\quad+(\|\pa_1u\|_{L^\infty_{x_1}L^2_{x_2}}+\|\pa_1b\|_{L^\infty_{x_1}L^2_{x_2}})(\|u_1\|_{L^2_{x_1}L^\infty_{x_2}}+\|b_1\|_{L^2_{x_1}L^\infty_{x_2}}) \Big)d\tau\non\\
&+\int^{t-1}_0  \lan t-\tau\ran^{-1}\Big((\|u_2\|_{L^2_{x_1}L^\infty_{x_2}}+\|b_2\|_{L^2_{x_1}L^\infty_{x_2}})(\|u\|_{L^\infty_{x_1}L^2_{x_2}}+\|b\|_{L^\infty_{x_1}L^2_{x_2}}) \Big)d\tau\non\\
&+ \int^{t}_{t-1} \|\pa_1(u_i\pa_iu+u_i\pa_ib+b_i\pa_iu+b_i\pa_ib)\|_{L^2}  d\tau\non\\
\lesssim& \lan t\ran^{-1} \Big(E^2(t)+\wt{E}^2(t)+\int_0^t\cF_1^2(\tau)d\tau\Big),
\end{align*}
here we use
\begin{align}\label{p1p2b2}
\|\pa_1\pa_2b_2\|_{L^2}&\lesssim\|\pa_1b_2\|_{L^2}^\f12\|\pa_1\pa_2^2b_2\|_{L^2}^\f12=\|\pa_1b_2\|_{L^2}^\f12\|\pa_1^2\pa_2b_1\|_{L^2}^\f12\non\\
&\lesssim\|\pa_1b_2\|_{L^2}^\f12\|\pa_1^2b_1\|_{L^2}^\f14\|\pa_1^2\pa_2^2b_1\|_{L^2}^\f14=\|\pa_1b_2\|_{L^2}^\f12\|\pa_1\pa_2b_2\|_{L^2}^\f14\|\pa_1^2\pa_2^2b_1\|_{L^2}^\f14,
\end{align}
which imply
$\|\pa_1\pa_2b_2\|_{L^2}\lesssim\|\pa_1b_2\|_{L^2}^\f23\|\pa_1^2\pa_2^2b_1\|_{L^2}^\f13$, and
\begin{align*}
\int_{t-1}^t\|b_2\pa_1\pa_2b_1\|_{L^2}d\tau&\lesssim\int_{t-1}^t\|b_2\|_{L^\infty}\|\pa_1\pa_2b_1\|_{L^2}d\tau\\
&\lesssim\int_{t-1}^t\|b_2\|_{L^2}^\f14\|\pa_1b_2\|_{L^2}^\f14\|\pa_2b_2\|_{L^2}^\f14\|\pa_1\pa_2b_2\|_{L^2}^\f14\|\pa_1b_1\|_{L^2}^\f12\|\pa_1\pa_2^2b_1\|_{L^2}^\f12d\tau\\
&\lesssim\int_{t-1}^t\|b_2\|_{L^2}^\f14\|\pa_1b_2\|_{L^2}^{\f{5}{12}}\|\pa_2b_2\|_{L^2}^\f34\|\pa_1\pa_2^2b_1\|_{L^2}^\f12\|\pa_1^2\pa_2^2b_1\|_{L^2}^\f{1}{12}d\tau\\
&\lesssim\int_{t-1}^t \lan \tau\ran^{-(\f{101}{96}-\f{13}{12}\delta)}\wt{E}^{\f{17}{12}}(\tau)\cE_1^{\f12}(\tau)\cF_1^{\f{1}{12}}(\tau)d\tau\\
&\lesssim\lan t-1\ran^{-(\f{101}{96}-\f{13}{12}\delta)}\int_{t-1}^t \wt{E}^{\f{17}{12}}(\tau)\cE_1^{\f12}(\tau)\cF_1^{\f{1}{12}}(\tau)d\tau\\
&\lesssim\lan t\ran^{-(\f{101}{96}-\f{13}{12}\delta)}\int_{t-1}^t \wt{E}^2(\tau)+\cE_1^{2}(\tau)+\cF_1^{2}(\tau)d\tau\\
&\lesssim\lan t\ran^{-(\f{101}{96}-\f{13}{12}\delta)} \Big(\wt{E}^2(t)+\cE^2_1(t)+\int_0^t\cF_1^2(\tau)d\tau\Big).
\end{align*}
And by  Plancherel theorem,
\begin{align*}
\int_{t-1}^t\|\cF(b_1\pa_1^2b_1)\|_{L^2_{\xi}}d\tau\lesssim\int_{t-1}^t\|\textbf{1}_{|\xi_1|\leq1}\cF(b_1\pa_1^2b_1)\|_{L^2}d\tau+\int_{t-1}^t\|\textbf{1}_{|\xi_1|>1}\cF(b_1\pa_1^2b_1)\|_{L^2}d\tau=\cK_1+\cK_2,
\end{align*}
with
\begin{align*}
\cK_1&\lesssim\int_{t-1}^t\|\textbf{1}_{|\xi_1|\leq1}|\xi_1|\cF(b_1\pa_1b_1)\|_{L^2}+\int_{t-1}^t\|\textbf{1}_{|\xi_1|\leq1}\cF(\pa_1b_1\pa_1b_1)\|_{L^2}\\
&\lesssim\|b_1\pa_1b_1\|_{L^2}+\|\pa_1b_1\pa_1b_1\|_{L^2}\\
&\lesssim\|b_1\|_{L^\infty}\|\pa_1b_1\|_{L^2}+\|\pa_1b_1\|_{L^\infty}\|\pa_1b_1\|_{L^2}\\
&\lesssim\|b_1\|_{L^2}^\f14\|\pa_1b_1\|_{L^2}^\f54\|\p_2b_1\|_{L^2}^\f14\|\pa_1\pa_2b_1\|_{L^2}^\f14+\|\pa_1b_1\|_{L^2}^\f54\|\pa_1^2b_1\|_{L^2}^\f14\|\pa_1\pa_2b_1\|_{L^2}^\f14\|\pa_1^2\pa_2b_1\|_{L^2}^\f14\\
&\lesssim\|b_1\|_{L^2}^\f38\|\pa_1b_1\|_{L^2}^\f54\|\p_2^2b_1\|_{L^2}^\f18\|\pa_1\pa_2b_1\|_{L^2}^\f14+\|\pa_1b_1\|_{L^2}^\f{11}{8}\|\pa_1^2b_1\|_{L^2}^\f14\|\pa_1\pa_2^2b_1\|_{L^2}^\f18\|\pa_1^2\pa_2b_1\|_{L^2}^\f14\\
&\lesssim\lan t\ran^{-\f{33}{32}} \Big(\wt{E}^2(t)+\cE^2_1(t)\Big),
\end{align*}
and
\begin{align*}
\cK_2&\lesssim\int_{t-1}^t\|\textbf{1}_{|\xi_1|>1}|\xi_1|\cF(b_1\pa_1^2b_1)\|_{L^2}d\tau\\
&\lesssim\int_{t-1}^t\|\cF(\pa_1b_1\pa_1^2b_1)\|_{L^2}d\tau+\int_{t-1}^t\|\cF(b_1\pa_1^3b_1)\|_{L^2}d\tau\\
&=\cK_{21}+\cK_{22}.
\end{align*}
By \eqref{p1p2b2},
\begin{align*}
\cK_{21}&\lesssim\int_{t-1}^t\|\pa_1b_1\pa_1^2b_1\|_{L^2}d\tau\lesssim\int_{t-1}^t\|\pa_1b_1\|_{L^\infty}\|\pa_1^2b_1\|_{L^2}d\tau\\
&\lesssim\int_{t-1}^t\|\pa_1b_1\|_{L^2}^\f14\|\pa_1^2b_1\|_{L^2}^\f54\|\pa_1\pa_2b_1\|_{L^2}^\f14\|\pa_1^2\pa_2b_1\|_{L^2}^\f14d\tau\\
&\lesssim\int_{t-1}^t\|\pa_1b_1\|_{L^2}^\f38\|\pa_1b_2\|_{L^2}^\f{5}{6}\|\pa_1\pa_2^2b_1\|_{L^2}^\f18\|\pa_1^2\pa_2b_1\|_{L^2}^\f14\|\pa_1^2\pa_2^2b_1\|_{L^2}^\f{5}{12}d\tau\\
&\lesssim\lan t\ran^{-(\f{97}{96}-\f53\delta)} \Big(\wt{E}^2(t)+\cE^2_1(t)+\int_0^t\cF_1^2(\tau)d\tau\Big),
\end{align*}
and
\begin{align*}
\cK_{22}&=\int_{t-1}^t\|\cF\Big(b_1\pa_1(b_{1,t}-\pa_1u_1+u_i\pa_ib-b_i\pa_iu)\Big)\|_{L^2}d\tau,
\end{align*}
since the nonlinear term will bring more decay, here we only estimate
\begin{align*}
&\int_{t-1}^t\|\cF(b_1\pa_1b_{1,t})\|_{L^2}d\tau\\
\lesssim&\|b_1\pa_1b_1\|_{L^2}+\int_{t-1}^t\|\cF(b_{1,t}\pa_1b_1)\|_{L^2}d\tau\\
=&\|b_1\pa_1b_1\|_{L^2}+\int_{t-1}^t\Big\|\cF\Big((\pa_1^2b_1+\pa_1u-u_i\pa_ib+b_i\pa_iu)\pa_1b_1\Big)\Big\|_{L^2}d\tau,
\end{align*}
and
\begin{align*}
\|b_1\pa_1u_1\|_{L^2}\lesssim\|b_1\|_{L^\infty}\|\pa_1u_1\|_{L^2}\lesssim\lan t\ran^{-(1+\delta)} \Big(\wt{E}^2(t)+\cE^2_1(t)\Big)\\
\|b_1\pa_1b_1\|_{L^2}\lesssim\|b_1\|_{L^\infty}\|\pa_1b_1\|_{L^2}\lesssim\lan t\ran^{-(1+\delta)} \Big(\wt{E}^2(t)+\cE^2_1(t)\Big)\\
\|\pa_1u\pa_1b_1\|_{L^2}\lesssim\|\pa_1u\|_{L^2}\|\pa_1b_1\|_{L^\infty}\lesssim\lan t\ran^{-(1+\delta)} \Big(\wt{E}^2(t)+\cE^2_1(t)\Big).
\end{align*}
Thus we complete the proof of the the  decay rate of $\|\pa_1 u_{N}\|_{L^2}$.

{\bf Step 4.} The  decay rate of $\|\pa_2 u_{1,N}\|_{L^2}$.

We have the following estimate,
\begin{align*}
&\|\pa_2u_{1,N}\|_{L^2}\\
=&\int^t_0 \Big\|e^{-c(|\xi_1|^2+|\xi_2|^2)(t-\tau)}\xi_i\xi_2\cF(u_iu+u_ib+b_iu+b_ib)\Big\|_{L^2}d\tau\\
&+\int^t_0 \Big\|\textbf{1}_{|\xi_2|^2\gtrsim|\xi_1|}\f{|\xi_1|}{|\xi_2|^2}e^{-\f{|\xi_1|^2}{|\xi_2|^2}(t-\tau)-|\xi_1|^2(t-\tau)}
\xi_i\xi_2\cF(u_iu+u_ib+b_iu+b_ib)\Big\|_{L^2}d\tau\\
\lesssim& \int^{t-1}_0 \lan t-\tau\ran^{-1}\|u_iu+u_ib+b_iu+b_ib\|_{L^2} d\tau+\int^{t-1}_0  \lan t-\tau\ran^{-1}\|u_1u+u_1b+b_1u+b_1b\|_{L^2}d\tau\\
&+\int^{t-1}_0  \lan t-\tau\ran^{-\f34}\|u_2u+u_2b+b_2u+b_2b\|_{L^1_{x_1}L^2_{x_2}}d\tau+ \int^{t}_{t-1} \|\pa_2(u_i\pa_iu+u_i\pa_ib+b_i\pa_iu+b_i\pa_ib)\|_{L^2}  d\tau\non\\
\lesssim& \int^{t-1}_0 \lan t-\tau\ran^{-1}\Big(\|u\|_{L^2_{x_1}L^\infty_{x_2}}(\|u\|_{L^\infty_{x_1}L^2_{x_2}}+\|b\|_{L^\infty_{x_1}L^2_{x_2}})+\|b\|_{L^2}\|b\|_{L^\infty} \Big) d\tau\non\\
&+\int^{t-1}_0  \lan t-\tau\ran^{-\f34}\Big((\|u_2\|_{L^2_{x_1}L^\infty_{x_2}}+\|b_2\|_{L^2_{x_1}L^\infty_{x_2}})(\|u\|_{L^2}+\|b\|_{L^2}) \Big)d\tau\non\\
&+\int^t_{t-1} \|\pa_2(u_i\pa_iu+u_i\pa_ib+b_i\pa_iu+b_i\pa_ib)\|_{L^2} d\tau\non\\
\lesssim& \lan t\ran^{-\f12}  \Big(\wt{E}^2(t)+\cE^2_1(t)\Big),
\end{align*}
here we use
\begin{align*}
\|b_1\pa_1\pa_2b_1\|_{L^2}&\lesssim\|b_1\|_{L^\infty}\|\pa_1\pa_2b_1\|_{L^2}\lesssim\|b_1\|_{L^2}^\f14\|\pa_1b_1\|_{L^2}^\f14\|\pa_2b_1\|_{L^2}^\f14\|\pa_1\pa_2b_1\|_{L^2}^\f54\\
&\lesssim\|b_1\|_{L^2}^\f38\|\pa_1b_1\|_{L^2}^\f78\|\pa_2^2b_1\|_{L^2}^\f18\|\pa_1\pa_2^2b_1\|_{L^2}^\f58\lesssim\lan t\ran^{-\f{3}{4}} \Big(\wt{E}^2(t)+\cE^2_1(t)\Big),
\end{align*}
and
\begin{align*}
\|b_2\pa_2^2b_1\|_{L^2}&\lesssim\|b_2\|_{L^\infty}\|\pa_2^2b_1\|_{L^2}\lesssim\|b_2\|_{L^2}^\f14\|\pa_1b_2\|_{L^2}^\f14\|\pa_2b_2\|_{L^2}^\f14\|\pa_1\pa_2b_2\|_{L^2}^\f14\|\pa_2^2b_1\|_{L^2}\\
&\lesssim\|b_2\|_{L^2}^\f14\|\pa_1b_2\|_{L^2}^\f{5}{12}\|\pa_2b_2\|_{L^2}^\f14\|\pa_1\pa_2^2b_2\|_{L^2}^\f{1}{12}\|\pa_2^2b_1\|_{L^2}\lesssim \lan t\ran^{-(\f{65}{96}-\f{13}{12}\delta)} \Big(\wt{E}^2(t)+\cE^2_1(t)\Big).
\end{align*}

{\bf Step 5.} The decay rate of  $ \pa_1b_N$.\\

Similarly Step 2, we have the $L^2$ estimate,
\begin{align*}
\|\pa_1b_{1,N}\|_{L^2}
\lesssim& \int^{t-1}_0 \Big\|e^{-c(|\xi_1|^2+|\xi_2|^2)(t-\tau)}|\xi_i||\xi_1|\cF(u_iu+u_ib+b_iu+b_ib)\Big\|_{L^2}d\tau\\
&+\int^{t-1}_0 \Big\|\f{|\xi_1|}{|\xi_2|^2}e^{-\f{|\xi_1|^2}{|\xi_2|^2}(t-\tau)-|\xi_1|^2(t-\tau)}
|\xi_i||\xi_1|\cF(u_iu+u_ib+b_iu+b_ib)\Big\|_{L^2}d\tau\\
&+\int^{t-1}_0 \Big\|e^{-\f{|\xi_1|^2}{|\xi_2|^2}(t-\tau)-|\xi_1|^2(t-\tau)}
|\xi_1|\cF(u_i\p_ib+b_i\p_iu)\Big\|_{L^2}d\tau\\
&+\int^{t-1}_0 \Big\|e^{-\f{|\xi_1|^2}{|\xi_2|^2}(t-\tau)-|\xi_1|^2(t-\tau)}
|\xi_1|^{-\f12}\f{|\xi_2|}{(|\xi_1|^2+|\xi_2|^2)^\f12}|\xi_1|^2\cF(b_1u_2+u_1b_2)\Big\|_{L^2_{\xi_1}L^1_{\xi_2}}d\tau\\
&+\int^{t}_{t-1} \|\pa_1(u_i\p_iu+u_i\p_ib+b_i\p_iu+b_i\p_ib)\|_{L^2}d\tau\\
\lesssim& \lan t\ran^{-\f34} \Big(\wt{E}^2(t)+\cE^2_1(t)+\int_0^t\cF_1^2(\tau)d\tau\Big),
\end{align*}
and
\begin{align*}
\|\pa_1b_{2,N}\|_{L^2}
\lesssim& \int^{t-1}_0 \Big\|e^{-c(|\xi_1|^2+|\xi_2|^2)(t-\tau)}|\xi_i||\xi_1|\cF(u_iu+u_ib+b_iu+b_ib)\Big\|_{L^2}d\tau\\
&+\int^{t-1}_0 \Big\|\f{|\xi_1|}{|\xi_2|^2}e^{-\f{|\xi_1|^2}{|\xi_2|^2}(t-\tau)-|\xi_1|^2(t-\tau)}
|\xi_i||\xi_1|\cF(u_iu+u_ib+b_iu+b_ib)\Big\|_{L^2}d\tau\\
&+\int^{t-1}_0 \|e^{-\f{|\xi_1|^2}{|\xi_2|^2}(t-\tau)-|\xi_1|^2(t-\tau)}
|\xi_1|^2\cF(b_1u_2+u_1b_2)\|_{L^2}d\tau\\
&+\int^{t-1}_0 \Big\|e^{-\f{|\xi_1|^2}{|\xi_2|^2}(t-\tau)-|\xi_1|^2(t-\tau)}
|\xi_1|^{-\f12}\f{|\xi_2|}{(|\xi_1|^2+|\xi_2|^2)^\f12}|\xi_1|^2\cF(b_1u_2+u_1b_2)\Big\|_{L^2_{\xi_1}L^1_{\xi_2}}d\tau\\
&+\int^{t}_{t-1} \|\pa_1(u_i\p_iu+u_i\p_ib+b_i\p_iu+b_i\p_ib)\|_{L^2}d\tau\\
\lesssim& \lan t\ran^{-(\f78-2\delta)}  \Big(\wt{E}^2(t)+\cE^2_1(t)+\int_0^t\cF_1^2(\tau)d\tau\Big),
\end{align*}
here we use when $t-\tau>1$, it has the same structure as $b_{N}$, and when $t-\tau\leq1$,
 \begin{align*}
\int_{t-1}^t\|\pa_1(u_i\p_iu+u_i\p_ib+b_i\p_iu+b_i\p_ib)\|_{L^2}d\tau\lesssim\lan t\ran^{-1} \Big(\wt{E}^2(t)+\cE^2_1(t)+\int_0^t\cF_1^2(\tau)d\tau\Big)
 \end{align*}
 we have solved in Step 3, and we give the calculation of the different items of $b_{N}$ following \eqref{9000},

\begin{align*}
&\int^{t-1} _0\Big\| |\xi_1||\xi_2| e^{-\f{|\xi_1|^2}{|\xi_2|^2}(t-\tau)-|\xi_1|^2(t-\tau)} \cF\Big( (-\Delta)^{-1} \na^T\cdot \p_\tau b_{\sim}  \,b_1\Big) \Big\|_{L^2} d\tau\non\\
&\lesssim  \Big\| \pa_1\na\Big( (-\Delta)^{-1} \na^T\cdot b_{\sim}  \,b_{1}\Big) \Big\|_{L^2} +\lan t\ran^{-\f34} \Big\| \na  \Big( (-\Delta)^{-1} \na^T\cdot b_{\sim, 0}  \,b_{1, 0}\Big) \Big\|_{L^1_{x_1}L^2_{x_2}}\non\\
&\quad +\int^{t-1}_0 \lan t-\tau\ran^{-1} \Big\| \pa_1\na\Big( (-\Delta)^{-1} \na^T\cdot b_{\sim}  \,b_{1}\Big) \Big\|_{L^2} d\tau\non\\
&\quad +\int^{t-1} _0\lan t-\tau\ran^{-\f12}\Big\| |\xi_2|e^{-\f{|\xi_1|^2}{|\xi_2|^2}(t-\tau)-|\xi_1|^2(t-\tau)} \cF\Big( (-\Delta)^{-1} \na^T\cdot  b_{\sim}  \,(\p_1^2b_1+\p_1 u_1-u\cdot\na b+b\cdot\na u\Big) \Big\|_{L^2} d\tau\\
\lesssim &\lan t \ran^{-\f34} (\wt{E}^2(t)+\cE_1^2(t)),
\end{align*}
and
\begin{align*}
&\int^{t-1} _0\Big\||\xi_1|^{\f32} |\xi_2| e^{-\f{|\xi_1|^2}{|\xi_2|^2}(t-\tau)-|\xi_1|^2(t-\tau)} \cF\Big( (-\Delta)^{-1} \na^T\cdot \p_\tau b_{\sim}  \,b_1\Big) \Big\|_{L^2} d\tau\non\\
&\lesssim  \Big\| \pa_1\na\Big( (-\Delta)^{-1} \na^T\cdot b_{\sim}  \,b_{1}\Big) \Big\|_{L^2} +\lan t\ran^{-1} \Big\| \na  \Big( (-\Delta)^{-1} \na^T\cdot b_{\sim, 0}  \,b_{1, 0}\Big) \Big\|_{L^1_{x_1}L^2_{x_2}}\non\\
&\quad +\int^{t-1}_0 \lan t-\tau\ran^{-\f54} \Big\| \pa_1\na\Big( (-\Delta)^{-1} \na^T\cdot b_{\sim}  \,b_{1}\Big) \Big\|_{L^2} d\tau\non\\
&\quad +\int^{t-1} _0\lan t-\tau\ran^{-\f34}\Big\| |\xi_2|e^{-\f{|\xi_1|^2}{|\xi_2|^2}(t-\tau)-|\xi_1|^2(t-\tau)} \cF\Big( (-\Delta)^{-1} \na^T\cdot  b_{\sim}  \,(\p_1^2b_1+\p_1 u_1-u\cdot\na b+b\cdot\na u\Big) \Big\|_{L^2} d\tau\\
\lesssim &\lan t \ran^{-(\f78-2\delta)} (\wt{E}^2(t)+\cE_1^2(t)),
\end{align*}
here we use
\begin{align*}
\Big \|\pa_1\na\Big( (-\Delta)^{-1} \na^{T} \cdot b_{\sim}b_1\Big)\Big\|_{L^2}
&\lesssim \| \pa_1b_{ \sim}\|_{L^2}\| b_1\|_{L^\infty}+ \| b_{ \sim}\|_{L^\infty}\| \pa_1b_1\|_{L^2}+\| b_{2,\sim}\|_{L^2_{x_1}L^\infty_{x_2}}\| \na b_1\|_{L^\infty_{x_1}L^2_{x_2}}\\
&\quad+\|(-\Delta)^{-1} \na^{T} \cdot b_{\sim}\|_{L^\infty}\| \pa_1\na b_1\|_{L^2}\non\\
&\lesssim \| \pa_1b_{ \sim}\|_{L^2}\| b_1\|_{L^\infty}+ \| b_{ \sim}\|_{L^\infty}\| \pa_1b_1\|_{L^2}+\| b_{2,\sim}\|_{L^2_{x_1}L^\infty_{x_2}}\| \na b_1\|_{L^\infty_{x_1}L^2_{x_2}}\\
&\quad+\|(-\Delta)^{-1} \na^{T} \cdot b_{\sim}\|_{L^\infty}\| \pa_1 b_1\|_{L^2}^\f12\| \pa_1 \na^2b_1\|_{L^2}^\f12\non\\
&\lesssim \lan t \ran^{-(\f78-\f32\delta)} (\wt{E}^2(t)+\cE_1^2(t)).
\end{align*}

This completes the proof of Proposition 3.1.
\end{proof}

\smallskip

\section{acknowledgement}

Xiaoxia Ren was supported by NSF of China under Grants  (No. 12001195) and the Fundamental Research Funds for the Central Universities (No. 2023MS078).

\section{conflict of interest statement}
The authors declared that they have no conflicts of interest to this work.

\newpage

\appendix
\begin{appendix}
\section{Local well-posedness}

We give the local well-posedness of systems \eqref{eq:MHDT} in half space $\Omega$  without proof for completeness.
\begin{theorem}\label{Local}
	Assume that the initial data $(u_0,b_0)\in H^2(\Omega)$, $u_0\in L^2_0(\Om)$, and $\div{u_0}=0$ in $\Omega$, $b_{1, 0}=0$ on $\p\Om$, $\mathcal{P}(\pa_2^2 u_0 -u_0 \cdot \na u_0 +b_0 \cdot \na b_0) \in L_0^2(\Om)$. Then there exists a $T>0$ such that system  \eqref{eq:MHDT} admits a unique solution $(u,b)$ on $[0,T]$ satisfying
	\begin{align*}
		\displaystyle
		(u,b) \in C([0,T];H^2(\Omega)),
	\end{align*}
\end{theorem}
In fact, we can first construct an interation scheme for the system \eqref{eq:MHDT} in half space to obtain the approximate solution and then derive uniform bounds to pass the limit (Similar  Theorem 3.1 in \cite{RXZ}). This procedure is more or less standard and thus we omit their details.

\section{useful lemma}
\begin{lemma}\label{lemma1}
For the complex number $z \in C$, we have
\beno
(Re z)^2=\f{\sqrt{\Big(Re (z^2)\Big)^2+\Big(Im (z^2)\Big)^2}+Re (z^2)}{2}.  \eeno
\end{lemma}
\begin{proof} We make a simple calculation
\begin{align*}
z&=Re z+iIm z\\
z^2&=(Re z)^2-(Im z)^2+2iRe z Im z\\
(Re z)^2+(Im z)^2&=\sqrt{\big((Re z)^2-(Im z)^2\big)^2+\big(2Re z Im z\big)^2}\\
&=\sqrt{\Big(Re (z^2)\Big)^2+\Big(Im (z^2)\Big)^2}\\
(Re z)^2&=\f{(Re z)^2+(Im z)^2+(Re z)^2-(Im z)^2}{2}\\
&=\f{\sqrt{\Big(Re (z^2)\Big)^2+\Big(Im (z^2)\Big)^2}+Re (z^2)}{2}.
\end{align*}
Thus we complete the proof.\end{proof}

\section{Higher order energy estimate}
We will prove the additional  energy estimate which is needed  in obtaining the  asymptotic behavior.

We define  the  energy
\begin{align}\label{cE2}
\cE^2_1(t):=\|\pa_1u(t)\|^2_{H^2}+\|\pa_1b(t)\|^2_{H^2}+\|\pa_1\na p\|^2_{L^2}+\|(\pa_1u_t,\pa_1b_t)\|^2_{L^2},
\end{align}
and the dissipated energy
\begin{align}\label{cF2}
\cF^2_1(t):&=\|\pa_1\na u\|_{H^1}^2+\|\p_1^2 b\|_{H^2}^2 +\|\pa_1 u_t\|_{L^2}^2+\|\pa_1b_t\|_{L^2}^2+\|\pa_1^3\pa_2 u\|_{L^2}^2\non\\
&+\|\pa_1\pa_2 u_t\|_{L^2}^2+\|\pa_1^2 b_t\|_{L^2}^2+ \|\pa_1\na p\|_{L^2}^2+\|\pa_1\cP\pa_2^2u\|_{H^1}^2.
\end{align}
\begin{proposition}\label{higher order}
Assume that the solution $(u,b)$ of the system \eqref{eq:MHDT} satisfies
\beno
\sup\limits_{0\le t\le T}\big(\|\pa_1u(t)\|_{H^2}^2+\|\pa_1b(t)\|_{H^2}^2\big)\leq  c_0^2.  \eeno
If $ c_0  $ is suitable small, then there hold that
\beno
\cE^2_1(t)+ \int_0^t\cF^2_1(s)ds\lesssim \|\pa_1u_0\|^2_{H^2}+\|\pa_1b_0\|^2_{H^2}
\eeno
for any  $t\in [0,T]$.
\end{proposition}

\begin{proof} 
In fact, the prove is very similar as Proposition \ref{high order} (just replace  $(u, b)$ to $(\pa_1u, \p_1 b)$), since $\pa_1$ does not change the boundary conditions.

But, when we apply $\pa_1$ to equation  $(\ref{eq:MHDT})_1$, the nonlinear item is $-\pa_1(u\cdot \na u)+\pa_1(b\cdot \na b)$ instead of $-\pa_1u\cdot\na \pa_1u+\pa_1b\cdot\na \pa_1b$. We can add $-\pa_1u\cdot\na \pa_1u+\pa_1b\cdot\na \pa_1b$ and subtract it to get the following equations
\begin{align}
\p_t \pa_1u-\pa_{2}^2 \pa_1u-\p_1 \pa_1b+\na \pa_1p=&-\pa_1u\cdot\na \pa_1u+\pa_1b\cdot \na \pa_1b+\Big(\pa_1u\cdot\na \pa_1u-\pa_1b\cdot \na \pa_1b\non\\
\qquad \qquad \qquad \qquad \qquad \qquad \qquad \qquad&-\pa_1u\cdot\na u-u\cdot\na \pa_1u+\pa_1b\cdot \na b+b\cdot \na \pa_1b\Big),\label{MHDTp1u}
\end{align}
similarly
\begin{align}
\p_t \pa_1b-\pa_{1}^2 \pa_1b-\p_1\pa_1u=&-\pa_1u\cdot\na \pa_1b+\pa_1b\cdot \na \pa_1u+\Big(\pa_1u\cdot\na \pa_1b-\pa_1b\cdot \na \pa_1u\non\\
\qquad \qquad \qquad \qquad \qquad \qquad \qquad \qquad&-\pa_1u\cdot\na b-u\cdot\na \pa_1b+\pa_1b\cdot \na u+b\cdot \na \pa_1u\Big).\label{MHDTp1b}
\end{align}

We just give a sketch of  Step 5', and other Steps'  can be obtained using the same way. Recall 
\begin{align*}
&\lan \p_1^2 u_1 ,\pa_1^2(u\cdot\na u_1)\ran
\lesssim& \|\pa_1^2 u_1\|_{L^2_{x_1}L^\infty_{x_2}}\|\pa_1\p_2u\|_{L^\infty_{x_1}L^2_{x_2}}\|\pa_1u_1\|_{L^2}+\|\pa_1^2 u_1\|_{L^2}\|\pa_1^2u_2\|_{L^2_{x_1}L^\infty_{x_2}}\|\pa_2u_1\|_{L^\infty_{x_1}L^2_{x_2}},
\end{align*}
in \eqref{p12u,p12b} in Step 5,  we similar have
\begin{align*}
&\lan \p_1^3 u_1 ,\pa_1^2(\pa_1u\cdot\na \pa_1u_1-\pa_1u\cdot\na u_1-u\cdot\na \pa_1u_1)\ran
\non\\
\lesssim& \|\pa_1^3 u_1\|_{L^2_{x_1}L^\infty_{x_2}}\Big(\|\pa_1^2\p_2u\|_{L^\infty_{x_1}L^2_{x_2}}+\|\pa_1\p_2u\|_{L^\infty_{x_1}L^2_{x_2}}\Big)\Big(\|\pa_1^2u_1\|_{L^2}+\|\pa_1u_1\|_{L^2}\Big)\\
&+\|\pa_1^3 u_1\|_{L^2}\Big(\|\pa_1^3u_2\|_{L^2_{x_1}L^\infty_{x_2}}+\|\pa_1^2u_2\|_{L^2_{x_1}L^\infty_{x_2}}\Big)\Big(\|\pa_1\pa_2u_1\|_{L^\infty_{x_1}L^2_{x_2}}+\|\pa_2u_1\|_{L^\infty_{x_1}L^2_{x_2}}\Big),
\end{align*}
where the highest order item $\|\pa_1^2\p_2u\|_{L^\infty_{x_1}L^2_{x_2}}\|\pa_1^2u_1\|_{L^2}$ comes from the highest order nonlinear item $\pa_1^2(\pa_1u\cdot\na\pa_1u)$,   the lowest order item $\|\pa_1\p_2u\|_{L^\infty_{x_1}L^2_{x_2}}\|\pa_1u_1\|_{L^2}$ comes from lowest order nonlinear item $\pa_1^2(u\cdot\na u)$, and the cross terms comes from $\pa_1^2(\pa_1u\cdot\na u)$ and $\pa_1^2(u\cdot\na \pa_1u)$.

Then we can get the result in step 5'
\begin{align*}
&\f12\f d {dt}(\|\pa_1^3 u\|_{L^2}^2+\|\pa_1^3 b\|_{L^2}^2)+(\|\p_1^3 \pa_2 u\|_{L^2}^2+\|\p_1^4  b\|_{L^2}^2)\non\\
\lesssim&\Big(\|(\pa_1u,\pa_1b)\|_{H^2}+\|(\pa_1u,\pa_1b)\|_{H^2}^2+\|(u,b)\|_{H^2}+\|(u,b)\|_{H^2}^2\Big)\\
 &\Big(\|\p_1^2b_t\|_{L^2}+ \|\p_1^2 b\|_{H^2}^2+\|\p_1b_t\|_{L^2}+ \|\p_1 b\|_{H^2}^2+\|\pa_1^3\pa_2 u\|_{L^2}^2+\|\pa_1^2\pa_2 u\|_{L^2}^2+\|\pa_1\pa_2 u\|_{L^2}^2+\|\pa_2 u\|_{L^2}^2\Big).
\end{align*}
 For simplicity, we omit the proof of other Steps'.

\end{proof}

\end{appendix}


\begin{thebibliography}{99}



\bibitem{A} H. Alfven,  {\it Existence of electromagnetic-hydrodynamic waves},  Nature, 150(1942), 405-406.

\bibitem{B} C. Bardos, C. Sulem and P.L. Sulem, {\it Longtime dynamics of a conductive fluid in the presence of a strong magnetic field}, Trans. Am. Math. Soc. 305 (1988), 175-191.




\bibitem{Cab} H. Cabannes,  {\it Theoretical Magneto-Fluid Dynamics}, Academic Press, New York, London, 1970.




\bibitem{CC} F. Califano and C. Chiuderi,  {\it Resistivity-independent dissipation of magnetrodydrodynamic waves in an inhomogeneous plasma},
Phy. Rev. E,  60(1999), part B, 4701-4707.




\bibitem{CL} Y. Cai and Z. Lei, {\it Global well-posedness of the incompressible magnetohydrodynamics}, Arch. Ration. Mech. Anal. 2018 (228), No.3, 969-993.

\bibitem{CW} C. Cao and J. Wu, {\it Global regularity for the 2D MHD equations with mixed partial dissipation and magnetic diffusion}, Adv. Math. 226 (2011), 1803-1822.


\bibitem{D} G. Duvaut and J. Lions, {\it Inequations en thermoelasticite et magnetohydrodynamique}, Arch. Rational Mech. Anal. 46 (1972), 241-279.


\bibitem{DR} L. Dong and X. Ren, {\it Asymptotic stability of the 2D MHD equations without magnetic diffusion}, J. Math. Phys. DOI: 10.1063/5.0112577.


\bibitem{F} C. Fefferman, D. McCormick, J. Robinson and J. Rodrigo, {\it Local existence for the non-resistive MHD equations in nearly optimal Sobolev spaces}, Arch. Ration. Mech. Anal. 223 (2017), 677-691.

\bibitem{G} B. Gallet, M. Berhanu, N. Mordant, {\it Influence of an external magnetic field on forced turbulence in a swirling flow of liquid metal}, Phys. Fluids 21 (2009) 085107.

\bibitem{GT} Y. Guo and I. Tice, {\it Almost exponential decay of periodic viscous surface waves without surface tension},
Arch. Rat. Mech. Anal., 207(2013), 459-531.




\bibitem{JJ} F. Jiang and S. Jiang, {\it On magnetic inhibition theory in non-resistive magnetohydrodynamic fluids}, Arch. Ration. Mech. Anal. 233 (2019), 749-798.

\bibitem{J} J. Jin, Y. Kagei, X. Ren, L. Wang and C. Zhai, {\it Asymptotic behavior of solution of the non-resistive 2D MHD equations on the half space}, arXiv:2401.104456v1.


\bibitem{KK} Y. Kagei and T. Kobayashi, {\it Asymptotic behavior of solutions of the compressible Navier-Stokes equations on the half space}. Arch. Rat. Mech. Anal. 177 (2005), 231-330.


\bibitem{LXZ} F. Lin, L. Xu and P. Zhang,
{\it Global small solutions of 2-D incompressible MHD system},
J.  Differential Equations, (259) 2015, 5440-5485.



\bibitem{L} H. Lin, R. Ji, J. Wu and L. Yan, {\it Stability of perturbations near a background
magnetic field of the 2D incompressible MHD
equations with mixed partial dissipation}, J. Funct. Anal. 279 (2020): 108519.


\bibitem{PZZ} R. Pan, Y. Zhou and Y. Zhu, {\it Global classical solutions of three dimensional viscous MHD system without magnetic diffusion on periodic boxes}, Arch. Rational Mech. Anal. 227 (2018), 637-662.

\bibitem{RWXZ} X. Ren, J. Wu, Z. Xiang and Z. Zhang,
 {\it Global existence and decay of smooth solution for the 2-D MHD equations without magnetic diffusion}, J. Funct. Anal. 267(2), 503-541,

\bibitem{RXZ} X. Ren,  Z. Xiang and Z. Zhang,
 {\it Global well-posedness for the 2D MHD equations without magnetic diffusion in a strip  domain},
Nonlinearity, 29(2016) 1257-1291.

\bibitem{S} M. Sermange and R. Temam, {\it Some mathematical questions related to the MHD equations}, Comm. Pure Appl. Math. 36 (1983), 635-664.

\bibitem{TW} Z. Tan and Y. Wang,
{\it Global well-posedness of an initial-boundary value problem for viscous non-resistive {MHD} systems}, SIAM J. Math. Anal., 50(2018),1432-1470.



\bibitem{W} R. Wan,
{\it On the temporal decay for the 2D non-resistive incompressible MHD equations}, Annales Henri Poincare, DOI: 10.1007/s00023-023-01268-3.

\bibitem{WZ1} D. Wei and Z. Zhang, {\it Global well-posedness of the MHD equations in a homogeneous magnetic field}, Anal. PDE 10 (2017), 1361-1406.


\bibitem{WZ2} D. Wei and Z. Zhang, {\it Global well-posedness for the 2-D MHD equations with magnetic diffusion}, Commun. Math. Res. 36 (2020), 377-389.



\bibitem{Z} T. Zhang, {\it An elementary proof of the global existence and uniqueness theorem to 2-D incompressible non-resistive MHD system}, arXiv:1404.3081.








\end{thebibliography}
\end{document}